\documentclass[10pt]{article}

\usepackage{etex}
\usepackage[margin={1.2in}]{geometry}
\usepackage[titletoc]{appendix}

\usepackage{amsmath,blkarray,booktabs,bigstrut}
\usepackage{tikz-cd}

\usepackage{enumerate}

\usepackage{tikz,eso-pic}
\usepackage{tikz-3dplot}

\usepackage{tikz}
\usetikzlibrary{matrix}

\usepackage{eucal}
\usepackage{mathrsfs}
\usepackage{stmaryrd}

\usepackage{hyperref} 

\usepackage{rotating}

\usepackage{longtable}

\usepackage[sc]{mathpazo}
\usepackage{courier}

\usepackage[utf8]{inputenc}
\usepackage[T1]{fontenc}

\usepackage[english]{babel}
\usepackage{blindtext}

\usepackage{amsmath}

\usepackage{microtype}

\usepackage{amsmath}
\usepackage{amssymb}
\usepackage{amsthm}
\usepackage{color}
\usepackage{latexsym,extarrows}

\usepackage[all]{xy}
\usepackage{graphicx}
\usepackage{subcaption}

\usepackage[stable,multiple]{footmisc}

\RequirePackage[font=small,format=plain,labelfont=bf,textfont=it]{caption}
\makeatletter
\newsavebox{\@brx}
\newcommand{\llangle}[1][]{\savebox{\@brx}{\(\m@th{#1\langle}\)}%
  \mathopen{\copy\@brx\kern-0.5\wd\@brx\usebox{\@brx}}}
\newcommand{\rrangle}[1][]{\savebox{\@brx}{\(\m@th{#1\rangle}\)}%
  \mathclose{\copy\@brx\kern-0.5\wd\@brx\usebox{\@brx}}}
\makeatother

\def\e#1\e{\begin{equation}#1\end{equation}}
\def\ea#1\ea{\begin{align}#1\end{align}}

\theoremstyle{plain}
\newtheorem{thm}{Theorem}[section]
\newtheorem{lem}[thm]{Lemma}
\newtheorem{prop}[thm]{Proposition}
\newtheorem{cor}[thm]{Corollary}
\theoremstyle{definition}
\newtheorem{dfn}[thm]{Definition}
\newtheorem{ex}[thm]{Example}
\newtheorem{rem}[thm]{Remark}

\newtheorem{lemma}[thm]{Lemma}

\newtheorem{corollary}[thm]{Corollary}
\theoremstyle{definition}
\newtheorem{definition}[thm]{Definition}
\newtheorem{example}[thm]{Example}

\newcommand{\op}{\operatorname}
\newcommand{\C}{\mathbb{C}}

\newcommand{\Z}{\mathbb{Z}}
\renewcommand{\H}{\mathbf{H}}

\newcommand{\im}{\op{im}}

\newcommand{\pii}{2\pi \mathbf{i}}

\def\XXint#1#2#3{{\setbox0=\hbox{$#1{#2#3}{\int}$}
\vcenter{\hbox{$#2#3$}}\kern-.5\wd0}}

\numberwithin{equation}{section}

\makeatletter
\newcommand{\subjclass}[2][2010]{%
  \let\@oldtitle\@title%
  \gdef\@title{\@oldtitle\footnotetext{#1 \emph{Mathematics Subject Classification.} #2}}%
}
\newcommand{\keywords}[1]{%
  \let\@@oldtitle\@title%
  \gdef\@title{\@@oldtitle\footnotetext{\emph{Key words and phrases.} #1.}}%
}
\makeatother

\makeatletter
\let\orig@afterheading\@afterheading
\def\@afterheading{%
   \@afterindenttrue
  \orig@afterheading}
\makeatother

\begin{document}
\title{\bf On the Cohomology Classes of Certain Eisenstein-Kronecker Forms}
\author{Zhixuan Sheng and Jie Zhou}
\date{}

\maketitle

\begin{abstract}
	We study the de Rham cohomology and Deligne splitting for a class of Eisenstein-Kronecker forms by introducing a modular weight grading.
	As an application, 
we compute a particular cohomology class that gives rise to generating series of Gromov-Witten invariants of elliptic curves.

\end{abstract}
\setcounter{tocdepth}{1} \tableofcontents

\section{Introduction}

Deligne's theory of mixed Hodge structures provides a fundamental framework for studying mirror symmetry and enumerative geometry, see \cite{Deligne:1997,  Cox:1999, Voisin:1999, Katzarkov:2008}. For example, for a large class of Calabi-Yau manifolds, the mirror principle
\cite{Givental:1996, Givental:1998, Lian:1997, Lian:1998}
establishes the genus-zero mirror symmetry correspondence
 ``Gromov-Witten generating series = period integrals'', where both sides admit natural formulations in terms of mixed Hodge structures \cite{Cox:1999}.
It is   widely expected  that a similar Hodge-theoretic picture persists at higher genera, see e.g., \cite{Bershadsky:1993, Bershadsky:1994, Costello:2012}.

Recently, building on \cite{Bloch:2000, Okounkov:2006}, it is shown in \cite{Zhou:2023GWgeneratingseries} that this is indeed the case for elliptic curves.
More precisely, the all-genus Gromov-Witten invariants of elliptic curves can be expressed as explicit configuration space integrals of cohomological classes constructed from sections of Poincaré bundles. Moreover, these invariants admit a purely mixed-Hodge-theoretic formulation in de Rham cohomology \cite{Zhou:2023cohomologicalpairings, Zhou:2026iteratedintegrals}.

Let us briefly recall the relevant results.
Let $E$ be an elliptic curve.
For each $n\geq 1$, there is an associated simple normal crossing divisor $\mathsf{D}\subseteq E^{n}$
and a concrete element (see Definition \ref{dfndeterminantformula})
\begin{equation}\label{eqnGWdifferential}
\omega_{n}\in H^{0}(E^{n}, \Omega_{E^{n}}^{n}(\log \mathsf{D}))\,,
\end{equation}
where $\Omega_{E^{n}}^{n}(\log \mathsf{D})$ is the sheaf of $n$-forms with at worst logarithmic poles  along $\mathsf{D}$. 
The form $\omega_{n}$ can be conveniently expressed in terms of the Eisenstein-Kronecker functions
which arise as Laurent coefficients of the Kronecker theta function (see Section \ref{basicsonellipticfunctions}).
It determines a de Rham cohomology class 
\[
[\omega_{n}]_{\mathrm{dR}}\in H^{n}_{\mathrm{dR}}(U)\,,\quad U=E^{n}-\mathsf{D}\,.
\]
For the quasi-projective space $U$, its cohomology space  $H^{n}(U)$ carries a mixed Hodge structure $(H^{n}(U), W_{\bullet},F^{\bullet})$ and furthermore admits a Deligne canonical splitting \cite[Lemma 1.2.8]{Deligne:1971}, which generalizes the Hodge decomposition for pure Hodge structures,
\begin{equation}
H^{n}_{\mathrm{dR}}(U)=\bigoplus_{\substack{p,q\\ n\leq p+q\leq 2n}}I^{p,q}\,.
\end{equation}
Denote the $I^{n,0}$-component of the class $[\omega_{n}]_{\mathrm{dR}}$ by $[\omega_{n}]_{\mathrm{Hodge}}$.
It is shown  \cite[Proposition 4.8]{Zhou:2023cohomologicalpairings}, \cite[Equation 1.5]{Zhou:2026iteratedintegrals} that, up to the so-called elliptic completion (described in Section \ref{basicsonellipticfunctions}), the class $[\omega_{n}]_{\mathrm{Hodge}}$ coincides with the Gromov-Witten generating series of the elliptic curve mirror to $E$.\\

Motivated partially by these Hodge-theoretic aspects of mirror symmetry, the present work aims to study a broader class of Eisenstein-Kronecker forms, with particular emphasis on their de Rham cohomology classes and Deligne splittings. Our main results may be summarized as follows.
\begin{itemize}
    \item We introduce an additional structure, namely the \emph{modular weight grading}, on a certain space of Eisenstein-Kronecker forms (Definition \ref{dfnmodularweightgradingonTD}), and show that this grading behaves well under the differential (Lemma \ref{lemdclosednessformodularweightgrading}).
    \item We characterize the space 
    $F^{n}H^{n}_{\mathrm{dR}}(U)$  in terms of Eisenstein-Kronecker forms (Proposition \ref{propFnHndescription}), and construct an explicit generating set for this space (Proposition \ref{propcharacterizationofkerbarpartial}).
    \item 
    For forms in $F^{n}H^{n}_{\mathrm{dR}}(U)$ expressed in terms of these generators, we determine their Deligne splittings (Proposition \ref{proppqtypeinDelignesplitting}).
 
\end{itemize}

Among these results, the following explicit realization of the Deligne splitting on the level of forms is particularly relevant to the application considered in this work.

\begin{thm}(Theorem \ref{thmDelignecanonicalsplittingonformlevel})
For the form $\omega_{n}$ in \eqref{eqnGWdifferential}, there exist explicit $d$-closed forms $\phi_{j}$ (Lemma \ref{lemexpressgradingjpartintermsofbuldingblocks}), induced by the modular weight grading (Definition 	\ref{dfnmodularweightgradingonTD}), such that $\omega_{n}$ admits the decomposition 
\begin{equation}
\omega_{n}
	=\sum_{j=0}^{n-1}	\phi_{j}\in  H^{0}(X, \Omega_{E^{n}}^{n}(\log \mathsf{D}))\,.
\end{equation}
This decomposition induces the Deligne canonical splitting  $[	\omega_{n} ]_{\mathrm{dR}}=\sum_{j=0}^{n-1}	[\phi_{j}]_{\mathrm{dR}}\in\bigoplus_{j=0}^{n-1}I^{n,j}$.
\end{thm}

Due to the many disguises of the Kronecker theta function and Eisenstein-Kronecker forms in various areas, such as periods \cite{Zagier:1991}, elliptic polylogarithms \cite{Levin:2007, Calaque:2009, Brown:2011} (see also \cite{Bannai:2007, Bannai:2010}) and conformal field theory \cite{Raina:1989}, 
our results seem to be interesting on their own, besides their application to Gromov-Witten invariants.

\subsection*{Structure of the paper}

Section 2 reviews the basic properties of the Kronecker theta function and Eisenstein–Kronecker forms. We also establish several functional relations among Eisenstein–Kronecker functions that will be used throughout the paper.

In Section 3, we introduce the modular weight grading on a subspace of Eisenstein–Kronecker forms and study its compatibility with the de Rham differential.

Section 4 begins with a brief review of mixed Hodge structures. We then construct a class of 
$d$-closed Eisenstein-Kronecker forms and study their de Rham cohomology classes and Deligne splittings.

In Section 5, we apply the results of Sections 2–4 to a particular Eisenstein–Kronecker form arising in the study of Gromov–Witten invariants of elliptic curves. In particular, we construct its Deligne splitting explicitly on the level of differential forms.

The proofs of several technical results are deferred to Section 6 and Appendix A.
	
	\subsection*{Acknowledgments}
	
	We thank Yizhou Li for helpful discussions. This work was supported by the National Key Research and Development	Program of China (No. 2022YFA1007100).

\subsection*{Notation and conventions}
	\begin{itemize}
\item 
Let $\mathbb{N}$ be the set of nonnegative integers and $\mathbb{N}_{+}$ be the set of positive integers.
        \item 
    Denote  $[n]=\{1,2,\cdots, n\}$. The cardinality of a set $S=\{v_{1},\cdots,v_{p}\}$ or a sequence $I=(v_{1},\cdots, v_{p})$ is denoted by $|S|$, $|I|$, respectively. For a variety $X$, denote $X_{[n]}=X_{1}\times X_{2}\cdots\times X_{n}$, where each $X_{i}$ is a copy of $X$, indexed by $i\in [n]$.
    \item For a nonempty set $S$, 
    the set of partitions of $S$ is denoted by $\Pi_{S}$. An element in $\Pi_{S}$ is denoted by $\pi=\{\pi_{1},\cdots,\pi_{\ell}\}$.
    \item 
			Denote collectively the coordinates $(z_1,\cdots, z_n)$ by $z$.
			For notational simplicity, we shall often 
		suppress the arguments of a function whenever they are clear from the surrounding text.

		\item Let $D$ be an effective divisor in the smooth complex manifold $X$.
        Denote by $\Omega_{X}^{p}$, $\Omega_{X}^{p}(\log D)$, $\Omega_{X}^{p}(\star D)$ be the sheaf of  $p$-forms that are holomorphic, holomorphic except for logarithmic poles along $D$,
      holomorphic except for poles along $D$, respectively. Denote by
		$H^{0}(X,-)$ or $\Gamma(X,-)$ the corresponding space of global sections over $X$.
        
		Similarly, denote
      by $\mathcal{A}^{p,q}_{X}$, $\mathcal{A}^{p,q}_{X}(\log D)$, $\mathcal{A}^{p,q}_{X}(\star D)$ the sheaf of smooth $(p,q)$-form versions.
		Denote the corresponding space of global sections by $A_{X}^{p,q}, A^{p,q}_{X}(\log D), A^{p,q}_{X}(\star D)$ respectively. Denote also $A^{n}_{X}=\bigoplus_{p+q=n}A^{p,q}_{X},  A^{n}_{X}(\log D)=\bigoplus_{p+q=n}A^{p,q}_{X}(\log D), A^{n}_{X}(\star D)=\bigoplus_{p+q=n}A^{p,q}_{X}(\star D)$.		
	\end{itemize}

\section{Preliminaries}
\label{secpreliminaries}

We collect some preliminaries on Jacobi theta functions and Eisenstein-Kronecker forms in this part. 

For any  point $\tau$ on the upper half-plane $\H$, let
\begin{equation}
	E=\C/\Lambda_{\tau}\,,\quad \Lambda_{\tau}:=\Z\oplus \Z \tau
\end{equation}
be the corresponding elliptic curve.
We fix once and for all a
linear holomorphic coordinate $z$ on the universal cover $\mathbb{C}\rightarrow E$ and a holomorphic volume form $dz$ on $E$.
The origin on $E$ is taken to be the image of $0\in \mathbb{C}$ under the universal covering map $\mathbb{C}\rightarrow E$.
We also fix a weak Torelli marking $\{A,B\}$, namely a symplectic frame for $H_{1}(E,\mathbb{Z})$, where the class $A$ is taken to be the image  under the map $\mathbb{C}\rightarrow E$ of the segment connecting $ \tau$ and $ \tau+1$ on $\mathbb{C}$.

Throughout this work,
constructions on the elliptic curve $E_{\tau}=\C/\Lambda_{\tau}$ (and on the corresponding universal family
$\mathcal{E}_{\H}\rightarrow \H$) will be identified with their lifts
to the universal cover.
For example, meromorphic functions on $E$
are identified with those on $\mathbb{C}$
that are periodic/elliptic under the translation action by $\Lambda_{\tau}$.

\subsection{Kronecker theta function}
\label{basicsonellipticfunctions}

We first recall some standard facts about the Kronecker theta function, following the exposition in \cite[Section 2, Section 3]{Zhou:2026iteratedintegrals}.
Define $\theta$ as the normalized Jacobi theta function $\vartheta_{({1\over 2},{1\over 2})}$: 
\begin{equation}\label{eqnJacobitheta}
\theta(z)={\vartheta_{({1\over 2},{1\over 2})}(z)\over \vartheta'_{({1\over 2},{1\over 2})}(0)}={\vartheta_{({1\over 2},{1\over 2})}(z)\over -\pii \eta^3}\,,
\end{equation}
where $\eta$ is the Dedekind eta function.
The Jacobi theta function has the following expansion (see e.g., \cite[Chapter VI.3]{Silverman:2009arithmetic})
\begin{eqnarray}\label{eqnthetazetaexpansion}
		\theta(z)
	=ze^{-G_{2}z^{2}}\prod_{\lambda\in\Lambda_{\tau}\setminus\{(0,0)\}}\left(1-{z\over \lambda}\right)e^{{z\over \lambda}+{1\over 2}({z\over \lambda})^2}=z\exp\left(\sum_{k\geq 1}-2G_{2k}{z^{2k}\over 2k}\right)
	\,.
\end{eqnarray}	
where the $G_{2k}$'s are the classical Eisenstein series.
One also has the following automorphy for the theta function 
 \begin{equation}\label{eqnautomorphyoftheta}
	\theta(z+1)=-\theta(z)\,,\quad 
	\theta(z+\tau)=-e^{-\pi \mathbf{i}\tau}e^{-\pii z}\theta(z)\,,
\end{equation}
Let $\zeta,\wp$ be the Weierstrass $\zeta$-function and $\wp$-function,
with $\wp=-\partial_{z}\zeta$. 
One defines the following semi-period
\begin{equation}
	\int_{A}\wp \,dz=-\eta_{1}\,,\quad 
    \eta_{1}=2G_{2}\,.
\end{equation}

For later use, we introduce the following notation.
\begin{dfn}\label{dfnhatthetazeta}
	Define
\begin{equation}
	\mathbold{Y}=-{\pi\over \mathrm{im}\,\tau}\,,\quad 
	\mathbold{A}=	(\bar{z}-z)\mathbold{Y}=\pii\,{\bar{z}-z\over \bar{\tau}-\tau}\,.
\end{equation}
	Introduce 	
	\begin{eqnarray}\label{eqnhatthetazeta}
		\widehat{\eta}_1&:=&
		\eta_{1}+\mathbold{Y}\,,\nonumber\\
			\widehat{\theta}(z)&:=&\theta(z)\cdot \exp\left(-2\pi {(\mathrm {im}~z)^2\over \mathrm{im}~\tau}\right)\,.
		\end{eqnarray}	
\end{dfn}
The Jacobi theta function is then related to the Weierstrass elliptic functions by 
\begin{equation}\label{eqnderivativeofhatthetazeta}
	\partial_{z} \ln {\theta}=\zeta(z)-z{\eta}_{1}\,,\quad 
	\partial_{z} \ln \widehat{\theta}=\zeta(z)-z\widehat{\eta}_{1}		+\mathbold{Y}\bar{z}\,.
\end{equation}

Hereafter, we suppress the notation for the variable $\bar{z}$
that is linearly independent of $z$.
However, when considering properties such as parity, a notation 
$f(-z)$
stands for the quantity obtained from replacing both $z,\bar{z}$ in 
$f$ by $-z,-\bar{z}$.\\

Of crucial importance in this work are the Kronecker theta function $	S_{c}(z)$ \cite{Weil:1976} or Szego kernel, and their Laurent coefficients.

\begin{dfn}\label{dfnSzegokernel}
Set
\begin{equation}
	S_{c}(z):={\theta(z+c)\over \theta(z)\theta(c)}\,,\quad 
	\widehat{S}_{c}(z):=
	e^{c\mathbold{A}(z)}{\theta(z+c)\over \theta(z)\theta(c)}\,.
\end{equation}
	These functions are meromorphic in $c$. Denote their Laurent expansions in $c$ by
\begin{equation}
	S_{c}(z)=
	{1\over c}\sum_{m\geq 0} c^{m}e_{m}(z) \,,\quad 
		\widehat{S}_{c}(z)={1\over c}\sum_{m\geq 0} c^{m} \widehat{e}_{m}(z) 
	\,,
\end{equation}
with the convention that $e_{m}(z)=0= \widehat{e}_{m}(z)$ for $m\leq -1$.
\end{dfn}

\begin{dfn}\label{dfnringofquasiellipticfunctions}
	Let the notation be as above.
	Define the ring $\widetilde{\mathcal{R}}$ of \emph{quasi-elliptic} 
    functions to be the $\mathbb{C}$-algebra\footnote{The generators $\{{e}_{m}(z)\}_{m\geq 1}, \{\widehat{e}_{m}(z)\}_{m\geq 1}$ 
	satisfy relations such as ones induced by the Weierstrass relations. These relations, however, 
	do not play an important role in this work.} generated by $ \{e_{m}\}_{m\geq 1}$,
    and 
    the ring $\widehat{\mathcal{R}}$ of \emph{almost-elliptic} functions to be the $\mathbb{C}$-algebra
    generated by $ \{\widehat{e}_{m}\}_{m\geq 1}$.
	These rings are equipped the \emph{modular weight}, denoted by $M_{\bullet}$, obtained by assigning weight $m$ to 
	the generators $e_{m},\widehat{e}_{m} $.	
\end{dfn}

\begin{rem}\label{remmodularweight}
	The generators $\{e_{m}(z)\}_{m\geq 1}, \{\widehat{e}_{m}(z)\}_{m\geq 1}$ are actually  quasi-Jacobi and almost-Jacobi forms of index zero with the weight given by $M_{\bullet}$ above.
    This weight is an incarnation of the $\mathrm{SL}_{2}(\mathbb{Z})$-equivariance of constructions on the universal family $\mathcal{E}_{\H}\rightarrow \H$.
	See  \cite{Eichler:1985,  Kaneko:1995, Libgober:2011, Goujard:2016counting}
	for more details.
	This explains the nomenclature of modular weight.
\end{rem}

From Definition \ref{dfnSzegokernel} and the automorphy
\eqref{eqnautomorphyoftheta}, we see that the functions  $\{e_{m}(z)\}_{m\geq 1}$ are meromorphic in $z$ but are not elliptic.
In contrast, due to the presence of the extra factor in $\widehat{S}_{c}(z)$, the functions $\{\widehat{e}_{m}(z)\}_{m\geq 1}$ depend on $\bar{z}$ but are elliptic.
The following results regarding the non-holomorphicity, proved using the generating series technique, will be frequently used in this work.
\begin{lem}\cite[Lemma 2.7, Lemma 2.9, Lemma 3.2]{Zhou:2026iteratedintegrals}\label{lemanalyticpropertiesofemz}
Let the notation be as above.
\begin{enumerate}[i).]
    \item The functions $\widehat{e}_{m}(z),m\geq 0$ are elliptic
    in $z$, satisfying $\widehat{e}_{m}(-z)=(-1)^{m}\widehat{e}_{m}(z)$.
    \item 
    The function $\widehat{e}_{m}(z)$
    is smooth everywhere except for a simple pole along $z\in \Lambda_{\tau}$ for $m\geq 1$. More precisely, the polar part of $\widehat{e}_{m}(z)$ at $z=0$ is $({\mathbold{Y}\bar{z})^{m-1}/(m-1)!z}$.
      In particular, their residues satisfy
\begin{equation}\label{eqnresidueofemdz}
\mathrm{res}_{z=0}(\widehat{e}_{1}(z)dz)=1\,,\quad \mathrm{res}_{z=0}(\widehat{e}_{m}(z)dz)=0\,~\mathrm{if} ~m\neq 1\,.
\end{equation}
\item     
One has 
\begin{equation}\label{eqnbarpartialem}
\bar{\partial}\widehat{e}_{m}(z)=\mathbold{Y}\,\widehat{e}_{m-1}(z)d\bar{z}\,,\quad m\geq 0\,.
\end{equation}
\end{enumerate}    
\end{lem}

Using \eqref{eqnthetazetaexpansion} and \eqref{eqnderivativeofhatthetazeta}, it is straightforward to see that 
the first few terms of $\{\widehat{e}_{m}(z)\}_{m\geq 0}$ are given by 
\begin{equation}\label{eqnfirstfewhatems}
	\widehat{e}_{0}(z)=1\,,\quad 
	\widehat{e}_{1}(z)=\partial_{z} \ln \widehat{\theta}\,,\quad
	\widehat{e}_{2}(z)={1\over 2}(-\wp(z)+(\partial_{z} \ln \widehat{\theta})^{2})\,.
\end{equation}
The functions $\{e_{m}(z)\}_{m\geq 0}$ are the holomorphic limits of their $\hat{}$-versions.
In general, they can be expressed in terms of the so-called Eisenstein-Kronecker series. 
Further analytic  properties and algebraic relations of these coefficients can be found in \cite[Section 2.2]{Zhou:2026iteratedintegrals}.\\

We shall frequently make use of the following relation without explicit mentioning
\[
\widehat{\mathcal{R}}\subseteq  \widetilde{\mathcal{R}}\otimes \mathbb{C}[\mathbold{A}]\,.
\]
The map $\widehat{\mathcal{R}}\rightarrow\widetilde{\mathcal{R}} $, obtained by taking the degree-zero term  in the polynomial expansion in $\mathbold{A}$, offers an isomorphism of these two rings. Following \cite{Kaneko:1995},
we call this map the \emph{holomorphic limit} 
and its inverse map the \emph{elliptic completion}.

\subsection{Eisenstein-Kronecker forms}

We now introduce spaces of Eisenstein-Kronecker forms, which are the main players in this work.

The space of
 holomorphic volume forms on $E_{[n]}$ is given by
\begin{equation}\label{eqnholomorphicvolumeform}
H^{0} (E_{[n]},\Omega^{n}_{E_{[n]}})=\mathbb{C}\Omega\,,\qquad \Omega:=dz_{1}\wedge \cdots \wedge dz_{n}
\,.
\end{equation}

\begin{dfn}\label{dfnEKforms}
	Let the notation and conventions be as above.
	Consider the following subvector space of $\widehat{\mathcal{R}}$ consisting of the linear span of the $\widehat{e}_{m}$ generators
	\begin{equation}
		\mathcal{V}=\sum_{m\geq 0}\mathbb{C}\,\widehat{e}_{m}(z)\,.
	\end{equation}

\end{dfn}

Motivated by \cite{Zhou:2023GWgeneratingseries} and \cite{Zhou:2026iteratedintegrals}, we make the following definition.

\begin{dfn}\label{dfnfunctionspaceassociatedtoD}
	Let the notation be as above.
	Consider the smooth hypersurface arrangement $\mathsf{D}$ in $E_{[n]}$ given by (the image under the quotient $\mathbb{C}_{[n]}\rightarrow E_{[n]}$)
	\begin{equation}
\label{eqnskl}\mathsf{D}=\bigcup_{\substack{k,\ell \in [n]\\
				k\neq \ell}} \big\{s_{k,\ell}=0\big\}\,,\quad s_{k,\ell}:=c_{k,\ell}+z_{k}-z_{\ell}\,.
	\end{equation}    	
	We also regard $s_{k,l}$ as defining a morphism $E_{[n]}\rightarrow E$.
    Let $c_{k,\ell},k,\ell\in [n]$ be generic, such that 
 $\mathsf{D}$ is a simple normal crossing divisor.
Define the multiplication map (here $U=E_{[n]}-\mathsf{D}$)
\[
\mu:
\bigotimes_{k,\ell\in [n]}s_{k,\ell}^{*}\widehat{ \mathcal{R}}
\longrightarrow C^\infty(U),\qquad
\bigotimes_{k,\ell}f_{k,\ell}
\longmapsto
\prod_{k,\ell}f_{k,\ell}\,,
\]
and consequently the following spaces\footnote{Relations in $\ker \mu$ will be discussed in Section \ref{seccombinatorialrelations} and Section \ref{subsecgraphsforpolardivisors}
 below.} 
	\begin{equation}\label{eqndfnFDED}		\mathcal{R}_{\mathsf{D}}=\mu\left(\bigotimes_{k,\ell\in [n]}s_{k,\ell}^{*}\widehat{ \mathcal{R}}\right)
     \,,\quad 
	\mathcal{V}_{\mathsf{D}}=\mu\left(\bigotimes_{k,\ell\in [n]}s_{k,\ell}^{*} \mathcal{V}\right)\,,
	\end{equation}
    with the cyclic convention $ z_{n+1}=z_{1},c_{n,n+1}=c_{n,1},\cdots$.
\end{dfn}

	In this work we will usually identify functions with the corresponding forms via the trivialization $\Omega$ in \eqref{eqnholomorphicvolumeform}, and call $\mathcal{R}_{\mathsf{D}},\mathcal{V}_{\mathsf{D}}$ the spaces of Eisenstein-Kronecker forms associated to $\mathsf{D}\subseteq E_{[n]}$.
    Our emphasis will be the vector space $\mathcal{V}_{\mathsf{D}}$,
for the purpose of studying mixed Hodge structures using the logarithmic complex (see Proposition \ref{propFnHndescription} in Section \ref{seccohomologyclass}).  This space is of great interest in that it underlies a model for the cohomology of $E_{[n]}-\mathsf{D}$, as shown in \cite{Brown:2011}.

\subsection{Functional relations}
\label{seccombinatorialrelations}

In this part, we prove a generalization of Fay's triscecant identity.
Developing their Laurent expansions lead to functional relations among Eisenstein-Kronecker forms that are useful in later discussions.
\begin{lem}\label{lem:identityforhatS}
	Suppose $p\geq 2$ and the values of $a_{i},i=1,2,\cdots, p$ are distinct. Then one has
	\begin{equation}
		\prod_{i=1}^{p} \widehat{S}_{a_i}(z_i) = \sum_{k=1}^{p} \widehat{S}_{a_k}(\sum_{i=1}^{p}z_i) \prod_{\substack{i=1\\i\neq k}}^{p} \widehat{S}_{a_i-a_k}(z_i)\,.
	\end{equation}
\end{lem}
\begin{proof}
	We prove it by induction. 
	The base case $p=2$ reads (with $a_{1}=-a, a_{2}=b, z_{1}=-x, z_{2}=y$)
		\[ 
	\widehat{S}_{a}(x)\widehat{S}_{b}(y)=\widehat{S}_{a}(x-y)\widehat{S}_{a+b}(y)+\widehat{S}_{a+b}(x)\widehat{S}_{b}(y-x)\,.
	\]
	This is the elliptic completion of the ordinary Fay's trisecant identity \cite[Equation 45]{Fay:1973}, \cite[Chapter III b.2]{Mumford:tataIII}.
	Suppose we have
	\[
	\prod_{i=1}^{p-1} \widehat{S}_{a_i}(z_i) = \sum_{k=1}^{p-1} \widehat{S}_{a_k}(\sum_{i=1}^{p-1}z_i) \prod_{\substack{i=1\\i\neq k}}^{p-1} \widehat{S}_{a_i-a_k}(z_i)\,.
	\]
	Multiplying both sides by $\widehat{S}_{a_p}(z_p)$, we obtain 
	\begin{eqnarray*}
		\prod_{i=1}^{p} \widehat{S}_{a_i}(z_i) &=& 
		\sum_{k=1}^{p-1} \widehat{S}_{a_k}(\sum_{i=1}^{p-1}z_i) \widehat{S}_{a_p}(z_p) \prod_{\substack{i=1\\i\neq k}}^{p-1} \widehat{S}_{a_i-a_k}(z_i) \\
		&=& \sum_{k=1}^{p-1} \left(\widehat{S}_{a_k}(\sum_{i=1}^{p}z_i) \widehat{S}_{a_p-a_k}(z_p) + \widehat{S}_{a_p}(\sum_{i=1}^{p}z_i) \widehat{S}_{a_k-a_p}(\sum_{i=1}^{p-1}z_i)\right) \prod_{\substack{i=1\\i\neq k}}^{p-1} \widehat{S}_{a_i-a_k}(z_i) \\
		&=& \sum_{k=1}^{p-1} \widehat{S}_{a_k}(\sum_{i=1}^{p}z_i) \prod_{\substack{i=1\\i\neq k}}^{p} \widehat{S}_{a_i-a_k}(z_i) 
		+ \widehat{S}_{a_p}(\sum_{i=1}^{p}z_i) \sum_{k=1}^{p-1} \widehat{S}_{a_k-a_p}(\sum_{i=1}^{p-1}z_i) \prod_{\substack{i=1\\i\neq k}}^{p-1} \widehat{S}_{a_i-a_k}(z_i) \\
		&=& \sum_{k=1}^{p-1} \widehat{S}_{a_k}(\sum_{i=1}^{p}z_i) \prod_{\substack{i=1\\i\neq k}}^{p} \widehat{S}_{a_i-a_k}(z_i) + \widehat{S}_{a_p}(\sum_{i=1}^{p}z_i) \prod_{i=1}^{p-1} \widehat{S}_{a_i-a_p}(z_i)  \\
		&=& \sum_{k=1}^{p} \widehat{S}_{a_k}(\sum_{i=1}^{p}z_i) \prod_{\substack{i=1\\i\neq k}}^{p} \widehat{S}_{a_i-a_k}(z_i)\,.
	\end{eqnarray*}
This finishes the proof for the $p$ case.
\end{proof}

For later simplification, we introduce the following shorthand notation.
\begin{dfn}\label{eqnemzshorthanddefinition}
 For each $\mathrm{m}=(m_{1},m_{2},\cdots, m_{p})\in \mathbb{N}^{p}$, define 
	\[
	|\mathrm{m}|:=\sum_{i=1}^{p}m_{i}\,,\quad \mu_0(\mathrm{m}) := |\{i \in \{1,\cdots,p\}|m_i=0\}|\,.
	\]
	Define also 
	\begin{equation}\label{eqnshorthandnotationforemz}
	\widehat{e}_{\mathrm{m}}(z)=\prod_{k=1}^{p}\widehat{e}_{m_{k}}(z_{k})
	\,.
	\end{equation}
	For brevity, denote $\mathbold{1}=(1,\cdots, 1)\in \mathbb{N}^{p}$.	
\end{dfn}

\begin{prop}\label{propgeneralpropofloopterm}
	Suppose $\mathrm{m}\in \mathbb{N}^{p}$ and $|\mathrm{m}|=M$.
Then one has
	\[
\widehat{e}_{\mathrm{m}}- \delta_{\mu_{0}(\mathrm{m})\neq 0}\, \widehat{e}_{\mathrm{m}} + \delta_{\mu_{0}(\mathrm{m})=0}\,
\sum_{\substack{m_0'+|\mathrm{m}'|=M\\ \mu_0(\mathrm{m}')\neq 0}}
	\prod_{i=1}^{p} \binom{m_i'-1}{m_i-1} \cdot
	\widehat{e}_{m_0'}(-\sum_{i=1}^{p}z_i)    \widehat{e}_{\mathrm{m}'} = 0 \,,
	\]
	where $\delta$ is the characteristic function.
\end{prop}

\begin{proof}
	Replacing $a_i$ by $ta_i, 1 \leq i \leq p$ in Lemma \ref{lem:identityforhatS}, we have 
	\[
	\prod_{i=1}^{p} \widehat{S}_{ta_i}(z_i) = \sum_{k=1}^{p} \widehat{S}_{ta_k}(\sum_{i=1}^{p}z_i) \prod_{\substack{i=1\\i\neq k}}^{p} \widehat{S}_{t(a_i-a_k)}(z_i) \,.
	\]
	We first extract the $t^{M-p}$-coefficient on both sides.
The left hand side gives 
	\begin{equation}
	\sum_{|\mathrm{m}|=M}  \prod_{i=1}^{p} a_i^{m_i-1} \cdot \widehat{e}_{\mathrm{m}}\,.
	\end{equation}
For the right hand side, it is 
\[
\sum_{k=1}^{p}a_k^{m_0-1} \prod_{\substack{i=1\\i\neq k}}^{p} (a_i-a_k)^{m_i-1}	\widehat{e}_{m_0} (\sum_{i=1}^{p}z_i)  \prod_{\substack{i=1\\
i\neq k}}^{p}e_{m_{i}}(z_{i})\,.
\]
Recall that $\widehat{e}_{0}=1$. This reduces to 
	\begin{equation}\label{eqnRHSreduced}
	\sum_{\substack{m_{0}+\mathrm{m}=M\\
	\mu_0(\mathrm{m})\neq 0}}\widehat{e}_{m_0} (\sum_{i=1}^{p}z_i)  \widehat{e}_{\mathrm{m}}
\sum_{k \in \mu_0}a_k^{m_0-1} \prod_{\substack{i=1\\i\neq k}}^{p} (a_i-a_k)^{m_i-1}	\,.
	\end{equation}
		Applying Lemma \ref{generallemmaforcomputingcoefficients},
		the quantity \eqref{eqnRHSreduced} becomes 
			\begin{equation*}\label{eqnRHSreducedreduced}
			\sum_{\substack{m'_{0}+\mathrm{m}'=M\\
					\mu_0(\mathrm{m}')\neq 0}}\widehat{e}_{m'_0} (\sum_{i=1}^{p}z_i) 
				 \widehat{e}_{\mathrm{m}'}					 	
	\left(
	\delta_{m'_{0},0}	\prod_{i=1}^{p}a_i^{m'_i-1} +(-1)^{m_{0}-1} \sum_{\substack{d_1,\cdots,d_p \geq 0 \\ \sum_{i=1}^{p}d_i = M-p}}
	\prod_{i=1}^{p} \binom{m'_i-1}{d_i} a_i^{d_i}
	\right) 	\,.	
		\end{equation*}
			We next compare the coefficients of $\prod_{i=1}^{p} a_i^{m_i-1}$ on the $t^{M-p}$-coefficients. The desired relation then follows immediately  using the parity relation in Lemma \ref{lemanalyticpropertiesofemz} (i).
\end{proof}

\begin{cor}\label{corforlooptermofwidehate_1}
Let the notation be as above.
Then one has
	\[
	\sum_{\substack{m_{0}\geq 0\,, \mathrm{m}\in \mathbb{N}^{p}\\m_{0}+|\mathrm{m}|=p}}\widehat{e}_{m_0} (-\sum_{i=1}^{p}z_i)  \widehat{e}_{\mathrm{m}} = 0\,.
	\]
\end{cor}
\begin{proof}
	Set $\mathrm{m}=\mathbold{1}$ in Proposition \ref{propgeneralpropofloopterm}.
	Then one has  $\binom{m_i'-1}{m_i-1} =1$ for all $1\leq i\leq p$.
In the summation determined by  $m_0'+|\mathrm{m}'|=p$ in Proposition \ref{propgeneralpropofloopterm}, allowing
$\mu_{0}(\mathrm{m}')=0$ forces $m_{0}'=0$ and $ m_{i}'=1, i=1,2,\cdots,p$.
This gives exactly the first term $\widehat{e}_{\mathbold{1}}$ in Proposition \ref{propgeneralpropofloopterm}. This concludes the proof.
\end{proof}

\section{Modular weight grading}
\label{secmodularweightgrading}

We shall construct a subspace of the Eisenstein-Kronecker forms that is naturally equipped with
a modular weight grading inherited from the one given in Definition \ref{dfnringofquasiellipticfunctions}  on  $\widehat{\mathcal{R}}$ .
	This grading plays an important role in studying cohomology classes and Deligne splittings in  Section \ref{secGWapplication}  below.
    
Recall Definition \ref{dfnfunctionspaceassociatedtoD},
the generators of the vector space $\mathcal{R}_{\mathsf{D}}$ are of the form 
\begin{equation}\label{eqnemr}
\widehat{e}_{\mathrm{m}}^{\mathrm{r}}(s):=\prod_{k= \ell}\widehat{e}^{r_{k,\ell}}_{m_{k,\ell}}(s_{k,\ell}) \cdot 
	\prod_{k\neq \ell}\widehat{e}^{r_{k,\ell}}_{m_{k,\ell}}(s_{k,\ell})\,,\quad 
m_{k,\ell}\,, r_{k,\ell}\geq 0\,.
\end{equation}
Using the  facts  that $s_{k,k}$ is constant from \eqref{eqnskl} and   $\widehat{e}_{0}=1$ from \eqref{eqnfirstfewhatems}, we see that 
$\widehat{e}_{\mathrm{m}}^{\mathrm{r}}(s)$ is a constant multiple of 
\begin{equation}\label{eqneminfty}
		\widehat{e}_{\mathrm{m},\infty}^{\mathrm{r}}(s):=\prod_{\substack{k\neq \ell\\ m_{k,\ell},\,r_{k,\ell}\geq 1}}\widehat{e}^{r_{k,\ell}}_{m_{k,\ell}}(s_{k,\ell})\in \mathcal{R}_{\mathsf{D}}\,.
	\end{equation}
We now take the generating set of the vector space $\mathcal{R}_{\mathsf{D}}$ to be 
\begin{equation}\label{eqngeneratingsetofFD}
	\mathcal{B}(\mathcal{R}_{\mathsf{D}})=
	\Big\{\widehat{e}_{\mathrm{m},\infty}^{\mathrm{r}}~|~m_{k,\ell}\,, r_{k,\ell}\geq 1
	\Big\}\,.
\end{equation}
Similarly, the subspace $\mathcal{V}_{\mathsf{D}}$
has a generating set 
\begin{equation}\label{eqngeneratingsetofVD}
	\mathcal{B}(\mathcal{V}_{\mathsf{D}})=
	\Big\{\widehat{e}_{\mathrm{m},\infty}^{\mathrm{r}} ~|~ m_{k,\ell}\geq 1\,,  r_{k,\ell}= 1
	\Big\}\,.
\end{equation}
    
\subsection{Graphs corresponding to polar divisors}
\label{subsecgraphsforpolardivisors}

Although not absolutely necessary, the language of graphs provides 
a convenient tool to describe  geometric properties. 
Many constructions and results regarding the Eisenstein-Kronecker forms are most 
easily formulated and visualized by graphs.

\begin{dfn}\label{dfndivisorgraph}
To each monomial  $\widehat{e}_{\mathrm{m}}^{\mathrm{r}}(s)$ in \eqref{eqnemr},  attach 
a  labeled, directed, colored, decorated graph  $\Gamma(\widehat{e}_{\mathrm{m}}^{\mathrm{r}})$ as follows.
\begin{itemize}
\item
Associate a labeled edge connecting $k,\ell$
to the polar divisor $\{s_{k,\ell}=0\}$ of $\widehat{e}^{r_{k,\ell}}_{m_{k,\ell}}$, if $k\neq \ell, m_{k,\ell}\geq 1,r_{k,\ell}\geq 1$.
\item The direction of the edge is from $k$ to $ \ell$. This is unambiguous since $s_{k,\ell}\neq \pm s_{\ell,k}$ and  thus
$\widehat{e}^{r_{k,\ell}}_{m_{k,\ell}}(s_{k,\ell})$ can not be expressed as $\pm \widehat{e}^{r_{k,\ell}}_{m_{k,\ell}}(s_{\ell,k})$.
\item 
Color the edge by the multiplicity $r_{k,\ell}$ of the  corresponding polar divisor.
\item 
Decorate the edge by the modular weight $m_{k,\ell}$.
\end{itemize}
Note that $\widehat{e}_{0}=1$ is regular and thus does not give rise to an edge in the polar divisor graph.
Extend the construction linearly from the set $\mathcal{B}(\mathcal{R}_{\mathsf{D}})$ to the formal graph space $\oplus_{\Gamma}\mathbb{C}\Gamma$.
\end{dfn}

\begin{dfn}\label{dfngraphterminologies}
	Suppose a graph $\Gamma$ in Definition \ref{dfndivisorgraph} contains $\ell$ connected components, denoted by  $\Gamma_1,\cdots,\Gamma_{\ell}$. 
	Denote the set of edges and vertices of $\Gamma_{k}$ by $E(\Gamma_{k}), V(\Gamma_{k})$ respectively,
	and those of $\Gamma$ by $E(\Gamma), V(\Gamma)$ respectively.	
	Denote by $m(e)$ the decoration of an edge $e$ and 
	\begin{equation}
	m(\Gamma_{k})=\sum_{e\in E(\Gamma_{k})}m(e)\,,\quad 
	m(\Gamma)=\sum_{k=1}^{\ell}m(\Gamma_{k})\,.
	\end{equation}
\end{dfn}

Lemma \ref{lemanalyticpropertiesofemz} (ii) tells that the quantity $\widehat{e}_{\mathrm{m},\infty}^{\mathrm{r}}$ in \eqref{eqneminfty} and its associated graph $\Gamma$ in Definition \ref{dfndivisorgraph} determines each other uniquely.
For this reason, we shall label the generating set by the graphs by introducing the notation
\begin{equation}\label{eqneGamma}
\widehat{e}_{\Gamma}:=\widehat{e}_{\mathrm{m},\infty}^{\mathrm{r}}=\prod_{\substack{k\neq \ell\\ m_{k,\ell}, \,r_{k,\ell}\geq 1}}\widehat{e}^{r_{k,\ell}}_{m_{k,\ell}}(s_{k,\ell})\in \mathcal{B}(\mathcal{R}_{\mathsf{D}})\,.
	\end{equation}

Using the graphs, we define the following special subspace of $\mathcal{V}_{\mathsf{D}}$, which enjoys particularly nice properties as we shall see below.
\begin{dfn}\label{dfntreetypefunctions}
Follow the notation in \eqref{eqneGamma}.
Define the vector space of \emph{forest-type} Eisenstein-Kronecker functions $\mathcal{V}_{\mathsf{D}}^{\mathrm{for}}$, by declaring its generating set to be
	\begin{equation}\label{eqngeneratingsetTD}
	\mathcal{B}(\mathcal{V}_{\mathsf{D}}^{\mathrm{for}}):=
	\{
		\widehat{e}_{\Gamma}\in \mathcal{B}(\mathcal{V}_{\mathsf{D}})~|~
		\Gamma~\mathrm{contains~ no ~loop}
	\}\,.
	\end{equation}
Here, "$\Gamma$ contains no loop" means that $\Gamma$ is a forest of trees.
Denote the set of  graphs corresponding to $\mathcal{B}(\mathcal{V}_{\mathsf{D}}^{\mathrm{for}})$ by $\mathcal{G}_{\mathrm{for}}$.
\end{dfn}

In contrast to the generating set $\mathcal{B}(\mathcal{V}_{\mathsf{D}})$ of $\mathcal{V}_{\mathsf{D}}$, whose elements are not linearly independent (see e.g.,  the relations in Proposition \ref{propgeneralpropofloopterm}), 
the generating set $\mathcal{B}(\mathcal{V}_{\mathsf{D}}^{\mathrm{for}})$
in fact gives a basis of the vector space $\mathcal{V}_{\mathsf{D}}^{\mathrm{for}}$ by the following result.

\begin{lemma}\label{lemelementsinB(E_D)arelinearlyindependent}
    The elements in $\mathcal{B}(\mathcal{V}_{\mathsf{D}}^{\mathrm{for}})$ are linearly independent.
\end{lemma}
\begin{proof}
Following the notation in \eqref{eqneGamma},
    take any generator 
    \[
    \widehat{e}_{\Gamma}=
	\prod_{\substack{k\neq \ell\\m_{k,\ell}\geq 1}}\widehat{e}_{m_{k,\ell}}(s_{k,\ell})\in \mathcal{B}(\mathcal{V}_{\mathsf{D}}^{\mathrm{for}})\,.\] Its polar divisor  is  
    \[
    \bigcup_{\substack{k\neq \ell \\m_{k,\ell}\geq 1}} \{s_{k,\ell}=0\} \,.
    \]
    The forest-type condition promises that the defining equations
    $\{s_{k,\ell}~|~k \neq \ell,m_{k,\ell}\geq 1\}$ for the irreducible components of the polar divisor are linearly independent.
    By Lemma \ref{lemanalyticpropertiesofemz} (ii), the collection of singular behaviors of $\widehat{e}_{\Gamma}$ near these irreducible components  is
    \[ 
    \Big\{ \frac{1}{s_{k,\ell}} \frac{(\mathbold{Y}\bar{s}_{k,\ell})^{m_{k,\ell}-1}}{(m_{k,\ell}-1)!}~|~ k \neq \ell \,, m_{k,\ell}\geq 1\Big\} \,.
    \]
    Construct the map
    \begin{equation}
        \rho: \mathcal{B}(\mathcal{V}_{\mathsf{D}}^{\mathrm{for}})\rightarrow \mathbb{C}[s_{k,\ell}^{-1},\overline{s}_{k,\ell}, k\neq \ell]\,,\quad   
    \widehat{e}_{\Gamma}\mapsto \prod_{\substack{k\neq \ell\\m_{k,\ell}\geq 1}}       \frac{1}{s_{k,\ell}} \frac{(\mathbold{Y}\bar{s}_{k,\ell})^{m_{k,\ell}-1}}{(m_{k,\ell}-1)!}\,.
    \end{equation}
    Here the $\bar{s}_{k,\ell}$ variables are regarded as independent complex coordinates from the $s_{k,\ell}$'s. 
       It is straightforward to see that $\rho$ is injective. 
       
    Consider a linear relation $\sum \lambda_{\Gamma}\rho(\widehat{e}_{\Gamma})=0$ in the image of $\rho$. 
    The sum is a rational function in the complex coordinates $s_{k,\ell},\bar{s}_{k,\ell}$.   
    Group terms having the same denominators in the sum together and take the multi-residues along the corresponding intersection of the simple normal crossing polar divisor.
   By the genericity\footnote{The linear independence of different groups might fail for non-generic values of the parameters $c_{k,\ell}$. For example, when $c_{k,\ell}=0$, one has the familiar Arnold relation $1/(z_{ij}z_{jk})+1/(z_{jk}z_{ki})+1/( z_{ki}z_{ij})=0$ for distinct $i,j,k$, where $z_{ij}=z_{i}-z_{j}$.} of the $c_{k,\ell}$'s, we see that the summation over each group vanishes individually.          
       For the summation over each group, the vanishing gives a polynomial identity in the $\overline{s}_{k,\ell}$'s, which are linearly independent by the fact that $\Gamma$ has no loop. This shows that each individual term in the vanishing is zero. 

By the injectivity of $\rho$, we conclude that 
        the generators $\widehat{e}_{\Gamma}\in  
     \mathcal{B}(\mathcal{V}_{\mathsf{D}}^{\mathrm{for}})$
     themselves are linearly independent.    
\end{proof}

\begin{rem}
In Definition \ref{dfntreetypefunctions}, one can replace the vector space $\mathcal{V}_{\mathsf{D}}$ by the algebra  $\mathcal{R}_{\mathsf{D}}$.
The same argument in the proof of Lemma \ref{lemelementsinB(E_D)arelinearlyindependent} above goes through if 
one collects the leading part of the polar part in the map $\rho$ instead.
\end{rem}

\subsection{Delooping to forest-types}
\label{secdivisorgraph}

The data $\mathcal{B}(\mathcal{V}_{\mathsf{D}}),\mathcal{B}(\mathcal{V}_{\mathsf{D}}^{\mathrm{for}})$  give
two \emph{vector spaces  equipped with generating sets}, namely
$(\mathcal{V}_{\mathsf{D}},\mathcal{B}(\mathcal{V}_{\mathsf{D}}))$ and  $( \mathcal{V}_{\mathsf{D}}^{\mathrm{for}},\mathcal{B}(\mathcal{V}_{\mathsf{D}}^{\mathrm{for}}))$.
In this part, we shall construct a delooping map to relate the vector spaces $\mathcal{V}_{\mathsf{D}} ,\mathcal{V}_{\mathsf{D}}^{\mathrm{for}} $  using these generating sets.

Similar to the shorthand notation in  Definition \ref{dfnfunctionspaceassociatedtoD}, we introduce the following notation adapted to the setting of Eisenstein-Kronecker forms.
The quantities introduced will be the building blocks in constructing $d$-closed forms and computing cohomology classes  in Section \ref{seccohomologyclass} below.

\begin{definition}\label{dfnpsiandPsi}
	For a sequence $I= (v_1,\cdots, v_p)$ with length $p\geq 2$ and $\mathrm{m} \in \mathbb{N}^{p}$,
	denote 
	\begin{equation}
		\psi_{\mathrm{m},I} := \prod_{k=1}^{p} \widehat{e}_{m_k}(s_{v_k,v_{k+1}})
		\,.
	\end{equation}
	where we keep the cyclic convention that $v_{p+k}=v_{k}$.
\end{definition}

\begin{dfn} 	For a sequence $I= (v_1,\cdots, v_p)$ with length $p\geq 2$ and $\mathrm{m} \in \mathbb{N}^{p}$ satisfying $\mu_{0}(\mathrm{m})=0$, define its \emph{delooping}
	\begin{equation}\label{eqndeloopingmaponsimpleloop}
	\mathfrak{d}(		\psi_{\mathrm{m},I}):=
		-\sum_{\substack{m_0'+|\mathrm{m}'|=M\\ \mu_0(\mathrm{m}')\neq 0}}
	\prod_{i=1}^{p} \binom{m_i'-1}{m_i-1} \cdot
	\widehat{e}_{m_0'}(-\sum_{i=1}^{p}s_{v_i,v_{i+1}})    	\psi_{\mathrm{m}',I}\,.
	\end{equation}
Note that the sum $\sum_{k=1}^{p}s_{v_{k},v_{k+1}}$ is a constant in $\{z_{v_{i}}\}_{i}$  due to the cyclic structure of the variables.
Extend it to $\mathcal{B}(\mathcal{V}_{\mathsf{D}}^{\mathrm{for}})$ multiplicatively and to
$\mathcal{V}_{\mathsf{D}}^{\mathrm{for}}$ linearly using Lemma \ref{lemelementsinB(E_D)arelinearlyindependent}.
\end{dfn}
Proposition \ref{propgeneralpropofloopterm} tells that 
\[
\psi_{\mathrm{m},I}=	\mathfrak{d}(		\psi_{\mathrm{m},I})\,.
\]
 Examining the polar divisors, we see that the graph of $\psi_{\mathrm{m},I}$
is a simple loop while that of $	\psi_{\mathrm{m}',I}$ consists of 
chains and isolated vertices by the condition $\mu_0(\mathrm{m}')\neq 0$. This is why $	\mathfrak{d}$ is called the delooping map.

\begin{lem}\label{lemexpressionbytreeterms}

The delooping map \eqref{eqndeloopingmaponsimpleloop} induces a map
\begin{equation}\label{eqnreducedformofED}
\mathfrak{d}: 	\mathcal{B}(\mathcal{V}_{\mathsf{D}})\setminus	\mathcal{B}(\mathcal{V}_{\mathsf{D}}^{\mathrm{for}})\rightarrow 	\mathcal{V}_{\mathsf{D}}^{\mathrm{for}}\,.
\end{equation}
In particular, for the vector spaces equipped with generating sets $(\mathcal{V}_{\mathsf{D}},  \mathcal{B}(\mathcal{V}_{\mathsf{D}}))  ,(\mathcal{V}_{\mathsf{D}}^{\mathrm{for}},\mathcal{B}(\mathcal{V}_{\mathsf{D}}^{\mathrm{for}}))$, one has 	$	\mathfrak{d}(\mathcal{V}_{\mathsf{D}}) =\mathcal{V}_{\mathsf{D}}^{\mathrm{for}}$.
\end{lem}
\begin{proof}
Take any generator $\widehat{e}_{\Gamma}=\widehat{e}_{\mathrm{m},\infty}\in \mathcal{B}(\mathcal{V}_{\mathsf{D}}) \setminus	\mathcal{B}(\mathcal{V}_{\mathsf{D}}^{\mathrm{for}})$.
Take any loop $L$ in  the graph $\Gamma$. Using the parity of $\widehat{e}_{m}$ given in Lemma \ref{lemanalyticpropertiesofemz} (i) we can assume, by a relabeling and a redefinition of the $c_{k,\ell}$'s involved in $L$ if necessary, that
		$L$ is a labeled directed graph of the form 
		\[
		L=\left(v_{1}\rightarrow v_{2}\rightarrow \cdots \rightarrow  v_{p}\rightarrow v_{1}\right)\,,
		\]	
		corresponding to the monomial
		\[
	\psi_{\mathrm{m},I}=
		\prod_{i=1}^{p}\widehat{e}_{m_{v_{i}}}(s_{v_{i}, v_{i}+1})\,,\quad 
		I=(v_{1},\cdots, v_{p})\,.
		\]
Here we have used the fact that  the color $r_{k,\ell}$ of any edge in $\Gamma$ is $1$ 
according to \eqref{eqngeneratingsetofVD}  and Definition
\ref{dfndivisorgraph}.
Also by Definition \ref{dfndivisorgraph},  the sequence $\mathrm{m}=(m_{v_{1}},\cdots m_{v_{p}})$ of decorations in the graph $L$ satisfies
$m_{v_{i}}\geq 1$ for $ i=1,2,\cdots, p$.
Applying the delooping map \eqref{eqndeloopingmaponsimpleloop} gives
 $\mathfrak{d}(	\psi_{\mathrm{m},I})$ as an element in the space $(\mathcal{V}_{\mathsf{D}}^{\mathrm{for}},\mathcal{B}(\mathcal{V}_{\mathsf{D}}^{\mathrm{for}}))$.
Note that the relabeling and redefinition of the $c_{k,\ell}$'s  do not affect this fact.
For each monomial  $\psi_{\mathrm{m}',I}$ in $\mathfrak{d}(	\psi_{\mathrm{m},I})$, its graph $\gamma$
is supported on the same set of vertices $v_{1},\cdots, v_{p}$.
These vertices do not form the loop $L$ anymore by the condition $\mu_{0}(\mathrm{m}')\neq 0$ in \eqref{eqndeloopingmaponsimpleloop} (but can be contained in a large loop in $\Gamma$ for instance).
This delooping procedure thus decreases the number of loops in the graph $\Gamma$ at least by $1$.
Applying this delooping procedure  iteratively to the remaining loops, we then obtain the desired map in \eqref{eqnreducedformofED}.
 Lemma \ref{lemelementsinB(E_D)arelinearlyindependent} implies that the map is independent of choices.
\end{proof}

\subsection{Space equipped with modular weight grading}
\label{secmodularweightgradingonTD}

Since $\mathcal{V}_{\mathsf{D}}^{\mathrm{for}}\subseteq \mathcal{V}_{\mathsf{D}}$ by Definition \ref{dfntreetypefunctions},
Lemma \ref{lemexpressionbytreeterms} implies that $\mathcal{V}_{\mathsf{D}}$ coincides with $\mathcal{V}_{\mathsf{D}}^{\mathrm{for}}$ via the delooping map $\mathfrak{d}$.
It thus reduces the study on the space  $(\mathcal{V}_{\mathsf{D}},\mathcal{B}(\mathcal{V}_{\mathsf{D}}))$ to the space $(\mathcal{V}_{\mathsf{D}}^{\mathrm{for}},\mathcal{B}(\mathcal{V}_{\mathsf{D}}^{\mathrm{for}}))$.
In this part, we show that the latter enjoy particularly nice properties.

By 
Definition \ref{dfntreetypefunctions} and Lemma \ref{lemelementsinB(E_D)arelinearlyindependent}, a function
$f\in (\mathcal{V}_{\mathsf{D}}^{\mathrm{for}},\mathcal{B}(\mathcal{V}_{\mathsf{D}}^{\mathrm{for}}))$ has a unique expression
\begin{equation}\label{eqndecomposeTDbygraphs}
	f=\sum_{\Gamma\in\mathcal{G}_{\mathrm{for}}}\lambda_{\Gamma}(f)\widehat{e}_{\Gamma}\,,
\end{equation}
where $\lambda_{\Gamma}(f)$ are constants that are independent of the variables $\{z_{k}\}_{k}$ but can depend on the parameters $\{c_{k,\ell}\}_{k,\ell}$.

\begin{dfn}\label{dfnmodularweightgradingonTD}
Let the  notation be as in Definition \ref{dfntreetypefunctions}. Equip $ (\mathcal{V}_{\mathsf{D}}^{\mathrm{for}},\mathcal{B}(\mathcal{V}_{\mathsf{D}}^{\mathrm{for}}))$  with the modular weight grading $M_{\bullet}$, by declaring
each quantity $\widehat{e}_{m}(s_{k,l})$ with  $k\neq \ell, m\geq 1$ to have modular weight $m$ and zero otherwise, then extending multiplicatively to the generating set 
$\mathcal{B}(\mathcal{V}_{\mathsf{D}}^{\mathrm{for}})$ and linearly from $\mathcal{B}(\mathcal{V}_{\mathsf{D}}^{\mathrm{for}})$ to $\mathcal{V}_{\mathsf{D}}^{\mathrm{for}}$.
By abuse of notation, we also use $M_{\bullet}$ to denote the graded vector space $M_{\bullet}\mathcal{V}_{\mathsf{D}}^{\mathrm{for}}$.
\end{dfn}

Strictly speaking, in order for $M_{\bullet}$ to be well-defined,
one needs to check that 
\begin{equation}\label{eqngradingcondition}
M_{k}\cap \sum_{\ell\neq k}M_{\ell}=\{0\}\quad \mathrm{for ~any}~k\,.
\end{equation}
This indeed follows from Lemma \ref{lemelementsinB(E_D)arelinearlyindependent} above.

\begin{rem}\label{remmodularweightgradingonTD}
Definition \ref{dfnmodularweightgradingonTD} amounts to declaring that 
the coefficients in \eqref{eqndecomposeTDbygraphs}, which only involve the parameters $\{c_{k,\ell}\}$, have modular weight grading $0$. 
In the graph language in Definition \ref{dfngraphterminologies},
the modular weight grading $M_{\bullet}$ of $\widehat{e}_{\Gamma}$ in \eqref{eqneGamma} is the decoration sum
$m(\Gamma)$. Related to the comment made  earlier in Remark \ref{remmodularweight}, this grading is the relative $\mathrm{SL}_{2}(\mathbb{Z})$-equivariance on the space $E_{[n]}\times S\rightarrow T$, where $S=\mathrm{Spec}\,\mathbb{C}[\{c_{k,\ell}\}_{ k,\ell}]$  and 
$T=\mathrm{Spec}\,\mathbb{C}[\{c_{k,k}\}_{k}]$.
\end{rem}

Thanks to the modular weight grading $M_{\bullet}$ on the spaces $ (\mathcal{V}_{\mathsf{D}}^{\mathrm{for}},\mathcal{B}(\mathcal{V}_{\mathsf{D}}^{\mathrm{for}}))$, we have the following result regarding 
the $\bar{\partial}$-closedness, which plays an important role in later discussions.

\begin{lem}\label{lemdclosednessformodularweightgrading}
	Suppose $f\in  (\mathcal{V}_{\mathsf{D}}^{\mathrm{for}},\mathcal{B}(\mathcal{V}_{\mathsf{D}}^{\mathrm{for}}))$.
	\begin{enumerate}[i).]
		\item Suppose  $f$	
	 satisfies $\bar{\partial}f=0$.
Then in the decomposition 
	\begin{equation}\label{eqnmodularweightgradingdecompositionoffunction}
	f=\sum_{k}f_{k}\in \bigoplus_{k}M_{k}\,,
	\end{equation}
one has $\bar{\partial}f_{k}=0$.
	\item 
	Concretely, the decomposition \eqref{eqnmodularweightgradingdecompositionoffunction}
	is given by 
	\begin{equation}\label{eqnconcretemodularweightgradingdecompositionoffunction}
	f_{k}=\sum_{\substack{\Gamma\in\mathcal{G}_{\mathrm{for}}\\
			m(\Gamma)=k}}\lambda_{\Gamma}(f)\widehat{e}_{\Gamma}\,,
	\end{equation}
	if $f$ is given by \eqref{eqndecomposeTDbygraphs}.
		\end{enumerate}
\end{lem}
\begin{proof}
		\begin{enumerate}[i).]
			\item 
	By assumption, we see that for any $i\in [n]$, one has
	\[
	\sum_{k}\bar{\partial}_{\bar{i}}f_{k}=\bar{\partial}_{\bar{i}}f=0\,.
	\]
	According to \eqref{eqnbarpartialem} and the definition of $\mathcal{R}_{\mathsf{D}}$ in Definition \ref{dfnfunctionspaceassociatedtoD}, we see that for any $i\in [n]$
\begin{equation}\label{eqncompatibilitybetweenbarpartialandM}
	\bar{\partial}_{\bar{i}}M_{k}\subseteq \mathbold{Y} M_{k-1}\,.
\end{equation}
	Combining, this gives
	\[
	\sum_{k}\bar{\partial}_{\bar{i}}f_{k}\in 	\mathbold{Y}\bigoplus_{k}M_{k-1} \,.
	\]
	The direct sum structure then tells that $\bar{\partial}_{\bar{i}}f_{k}=0$ for any $k$, as desired.
	\item 
	This follows from the fact that the modular weight grading of $\widehat{e}_{\Gamma}$ is the decoration sum $m(\Gamma)$, as mentioned in Remark \ref{remmodularweightgradingonTD}.
		\end{enumerate}
\end{proof}

\begin{rem}\label{remonmodularweightgrading}
A few remarks regarding the modular weight grading are in order.
\begin{enumerate}[i).]
\item The modular weight grading $M_{\bullet}$ can not be defined directly on the space $(\mathcal{V}_{\mathsf{D}},\mathcal{B}(\mathcal{V}_{\mathsf{D}}))$ in \eqref{eqndfnFDED},
due to the relations in Proposition \ref{propgeneralpropofloopterm} which violate the requirement \eqref{eqngradingcondition}.
In fact, on these spaces, $M_{\bullet}$ only defines a filtration instead of grading. 
However, by Lemma \ref{lemexpressionbytreeterms} we see that
 Lemma \ref{lemdclosednessformodularweightgrading}  holds on the delooped space $\mathfrak{d}(\mathcal{V}_{\mathsf{D}},\mathcal{B}(\mathcal{V}_{\mathsf{D}}))$.
\item One can also define the modular weight by declaring all $\widehat{e}_{m}(s_{kl}), k, \ell$ to have weight $m$ instead. This is the modular weight $\widehat{e}_{m}(s_{kl})$ in both the variables $\{z_{k}\}$ and parameters $\{c_{k,\ell}\}$ in $s_{k,\ell}=z_{k}-z_{\ell}+c_{k,\ell}$.
This gives gradings on the spaces $\mathcal{V}_{\mathsf{D}},\mathcal{R}_{\mathsf{D}}$ on which Lemma \ref{lemdclosednessformodularweightgrading} holds true.
 However, the resulting decomposition is not fine enough for the later purpose of studying Deligne splittings, see Proposition \ref{proppqtypeinDelignesplitting} and Example \ref{exn=3} below.
 \end{enumerate}
\end{rem}

\section{Cohomology classes of Eisenstein-Kronecker forms}
\label{seccohomologyclass}

In this part, we shall first review the basics of mixed Hodge structures. 
Based on the functional relations in Section \ref{seccombinatorialrelations}, we construct a class of $d$-closed Eisenstein-Kronecker forms.
We then study the cohomology classes and Deligne splittings for them.

\subsection{Review on the basics of Deligne's mixed Hodge structure}

We recall Deligne's spectral sequence \cite{Deligne:1971} in computing the mixed Hodge structures for the
complement $U$ of a simple normal crossing divisor $D$ in a nonsingular projective algebraic variety $X$ of dimension $n$.
Put a linear order on the components of $D$ as $D_{a},a=1,2,\cdots, \mathcal{N}$.

For any ordered tuple $I\subseteq [\mathcal{N}]:=\{1,2,\cdots,\mathcal{N}\}$,
define
\begin{equation}\label{eqnstratumDI}
	D_{I}=\bigcap_{a\in I} D_{a}\,,\quad
	D_{[j]}=\coprod_{I:\,|I|=j}D_{I}\,, \quad
	D_\varnothing := X \,.
\end{equation}
and
\begin{equation}\label{eqnstratumDa}
D(a) = D_{a}\cap (\cup_{b\neq a} D_{b})\,.
\end{equation}
There are natural closed immersions
\[\iota_{[j]}: D_{[j]}\rightarrow X\,,\quad 
\iota_{a} : D_{a} \hookrightarrow X \,.
\]

Consider the bifiltered complex of sheaves of differential forms on $X$ which may have logarithmic poles along $D$, i.e., $(\mathcal{A}^{\bullet}_{X}(\log D),d)$, equipped with the usual 
decreasing Hodge filtration $F^{\bullet}$ induced by the 
$dz,d\bar{z}$ types of differential forms,
and the increasing 
filtration $W_{\bullet}$ that keeps track of the pole order along $D$.
Explicitly, one has 
\begin{equation}\label{eqndfnoffiltrations}
	F^{p}\mathcal{A}_{X}^{\bullet}(\log D)
	:=\bigoplus_{k\geq p}\mathcal{A}^{k,\bullet-k}_{X}(\log  D )\,,\quad 
	W_{m}\mathcal{A}^{\bullet}_{X}(\log D)
	:=\mathcal{A}_{X}^{\bullet-m}\wedge \mathcal{A}^{m}_{X}(\log D)\,.
\end{equation}
One has $\mathbb{H}^{k}(X, \mathcal{A}^{\bullet}_{X}(\log D))
\cong H^{k}(X,A_{X}^{\bullet}(\log D))\cong H^{k}(U)$.
The above data on $\mathcal{A}^{\bullet}_{X}(\log D)$ induces 
filtrations on $\mathbb{H}^{k}(X,\mathcal{A}^{\bullet}_{X}(\log D))\cong H^{k}(U)$
\begin{eqnarray}\label{eqnfiltrationsoncohomology}
	F^{p}H^{k}(X,A_{X}^{\bullet}(\log D))&=&\mathrm{im} \left(H^{k}(X,F^{p}A_{X}^{\bullet}(\log D))
	\rightarrow H^{k}(X,A_{X}^{\bullet}(\log D))\right)
	\,,\nonumber\\
	W_{j}H^{k}(X,A_{X}^{\bullet}(\log D))&=&\mathrm{im} \left(H^{k}(X,W_{j}A_{X}^{\bullet}(\log D))
	\rightarrow H^{k}(X,A_{X}^{\bullet}(\log D))\right)
	\,.
\end{eqnarray}
This puts a mixed Hodge structure on $H^{k}(U)$, with the weight filtration given by $\widetilde{W}_{\bullet}=W_{\bullet}[-k]$. 

The spectral sequence associated to the Hodge filtration $F^{\bullet}$ degenerates at the first page, and one has
\begin{equation}\label{eqnHodgefiltrationdegeneration}
	\mathrm{gr}^{p}_{F}H^{p+q}(X,A_{X}^{\bullet}(\log D))\cong 
	H^{p+q}(X,\mathrm{gr}^{p}_{F}A_{X}^{\bullet}(\log D))\,.
\end{equation}
The spectral sequence associated to the weight filtration  $W_{\bullet}$ degenerates at its second page.
Its first page is conveniently described using the residue maps as follows.
The iterated Poincar\'e residue map gives a quasi-isomorphism of complexes of sheaves
\begin{equation}\label{eqnE0pageiso}
	\mathrm{res}_{[j]}=\bigoplus_{I:\,|I|=j}\mathrm{res}_{I}:
	\mathrm{gr}^{W}_{j}\mathcal{A}^{\bullet}_{X}(\log D)
	\cong (\iota_{[j]})_{*}\mathcal{A}^{\bullet}_{D_{[j]}}[-j]\,.
\end{equation}
Passing to cohomology, it gives an isomorphism of pure Hodge structures 
\begin{equation}\label{eqnE1pageiso}
	\mathrm{Res}_{[j]}=\bigoplus_{I:\,|I|=j}\mathrm{Res}_{I}:
H^{k}(X, \mathrm{gr}^{W}_{j}A^{\bullet}_{X}(\log D))\xrightarrow{\cong}
	H^{k}(D_{[j]}, A^{\bullet}_{D_{[j]}}[-j])\otimes \mathbb{Q}(-j)\,,
\end{equation}
where $\otimes \mathbb{Q}(-j)$ is the Tate twist.
On the first page of the spectral sequence associated to the $W_{\bullet}$ filtration, the entries are given by the spaces in \eqref{eqnE1pageiso} and the differentials by the Gysin maps.

Deligne's splitting on $(H^{\bullet}(U),F^{\bullet}, W_{\bullet})$ is given by
\begin{equation}\label{eqndfnDelignsplitting}
	H^{\bullet}(U)=\bigoplus_{p,q}I^{p,q}\,,\quad 
	I^{p,q}:=F^{p}\cap\widetilde{W}_{p+q}\cap \left(
	\overline{F^{q}}\cap \widetilde{W}_{p+q}+\sum_{i\geq 2}\overline{F^{q-i+1}}\cap \widetilde{W}_{p+q-i}
	\right)\,.
\end{equation}
The subspaces $I^{p,q}$ satisfy
\begin{equation}\label{splittingofWandF}
	\widetilde{W}_{k}\otimes \mathbb{C}=\bigoplus_{p+q\leq k}I^{p,q}\,,\quad
	F^{p}=\bigoplus_{i\geq p} \bigoplus_{q}I^{i,q}\,,
\end{equation}
\begin{equation}\label{Ipq=barIqp}
	I^{p,q} = \overline{I^{q,p}} \quad \mathrm{mod} \quad  \bigoplus_{\substack{r\leq p-1\\s\leq q-1}} I^{r,s} \,.
\end{equation}
In fact, the splitting $\{I^{p,q}\}_{p,q}$ in \eqref{eqndfnDelignsplitting} is the unique one that satisfies \eqref{splittingofWandF} \eqref{Ipq=barIqp}.

\begin{definition}
	Define the natural projections
	\[
	\operatorname{pr}_{I^{k,0}} : H^k(U) \to I^{k,0}H^k(U) \,,\quad 
	0\leq k\leq \dim_{\mathbb{C}}U=n\,.
	\]
	For $[\phi]_{\mathrm{dR}} \in H^n(U)$, denote   $[\phi]_{\mathrm{Hodge}}: = \operatorname{pr}_{I^{n,0}} [\phi]_{\mathrm{dR}}$.
\end{definition}

	Using \eqref{eqnHodgefiltrationdegeneration}, one has
\begin{equation}\label{eqnFnoHn}
	F^{n} H^n(X,A^{\bullet}_{X}(\log D)) = H^{n}(X,F^{n}A^{\bullet}_{X}(\log D)) 
	= \dfrac{\Gamma(X,\Omega^{n}_{X}(\log D))}{dF^{n}A^{n-1}_{X}(\log D)} = \Gamma(X,\Omega^{n}_{X}(\log D)) \,.
\end{equation}
Combining \eqref{eqndfnDelignsplitting}, this  implies that 
\begin{equation}\label{eqnInoHn}
I^{n,0}H^n(X,A^{\bullet}_{X}(\log D))=F^{n}H^n(X,A^{\bullet}_{X}(\log D))\cap H^{n}(X,W_{0}A_{X}^{\bullet}(\log D))=\Gamma(X,\Omega_{X}^{n})\,.
\end{equation}

\subsection{Kernel of residue maps}

Before studying $d$-closed Eisenstein-Kronecker forms,  we characterize $ I^{n,0}H^{n}(U) $ in terms of the kernel space  of certain residue maps.
This will be needed to compute the cohomology classes and Deligne splittings  in Section \ref{seccohomologyclassofbuildingblocks} below.

\begin{dfn}\label{dfnkersiduetoD(a)}
	Let the notation be as above.  Denote 
	the residue map on the sheaf level by
	\[
	\mathrm{res}_{a} : \mathcal{A}^{\bullet}_{X}(\log D) \to (\iota_{a})_{\ast} \mathcal{A}^{\bullet}_{D_{a}}(\log D(a))\,.
	\]
	Denote the corresponding residue map on the cohomology level by 
	\begin{equation}\label{eqnresiduemaptoD(a)}
	\mathrm{Res}_{a} : H^n(X,A_{X}^{\bullet}(\log D)) \to H^{n-1}(D_{a},A^{\bullet}_{D_{a}}(\log D(a)) ) \,.
	\end{equation}
	\end{dfn}
	\begin{lemma}\label{lemmakernelofRes}
		Let the notation be as before. 
Upon restricting to $F^{n}H^{n}(U)$, one has  
	\[
	\bigcap_{a=1}^{\mathcal{N}} \mathrm{ker}\, \mathrm{Res}_{a} \cap F^{n}H^{n}(U) = I^{n,0}H^{n}(U) \,.
	\]
\end{lemma}
\begin{proof}
A local calculation gives the short exact sequences of sheaves on $X$
	\begin{equation}\label{shortexactsequenceaboutres_a}
		0 \to \Omega^{\bullet}_{X} (\log \cup_{b:b\neq a}D_{b}) \to \Omega^{\bullet}_{X} (\log D) \xrightarrow{\mathrm{res}_{a}} (\iota_{a})_{\ast} \Omega^{\bullet-1}_{D_{a}}(\log D(a)) \to 0 \,.
	\end{equation}
	Taking the hypercohomology of the short exact sequence \eqref{shortexactsequenceaboutres_a}, we get
	\begin{equation*}
		\cdots \to \mathbb{H}^{n}(X,\Omega^{\bullet}_{X} (\log \cup_{b:b\neq a}D_{b}))
		\to \mathbb{H}^{n}(X,\Omega^{\bullet}_{X}(\log D)) 
		\xrightarrow{\mathrm{Res}_{a}}
		\mathbb{H}^{n}(X,(\iota_{a})_{\ast} \Omega^{\bullet-1}_{D_{a}}(\log D(a))) \to \cdots
	\end{equation*}
		For the closed immersion $\iota_{a} : D_{a} \hookrightarrow X$, the map $(\iota_{a})_{\ast}$ is exact. This shows that
	\[
	\mathbb{H}^{n}(X, (\iota_{a})_{\ast} \Omega^{\bullet-1}_{D_{a}}(\log D(a)))
	\cong
	\mathbb{H}^{n}(D_{a}, \Omega^{\bullet-1}_{D_{a}}(\log D(a)))
	\cong
	H^{n-1}(D_{a}, A^{\bullet}_{D_{a}}(\log D(a))) \,.
	\]
	Therefore, we obtain the long exact sequence
	\begin{equation*}\label{longexactsequenceofRes_a}
		\cdots 
		\to  
		H^{n}(X,A^{\bullet}_{X} (\log \cup_{b:b\neq a}D_{b}))
		\xrightarrow{\gamma_{a}}
		H^{n}(X,A^{\bullet}_{X}(\log D)) 
		\xrightarrow{\mathrm{Res}_{a}}
		H^{n-1}(D_{a}, A^{\bullet}_{D_{a}}(\log D(a))) \to \cdots
	\end{equation*}
    where $\gamma_{a}$ is the inclusion map. 
In particular, one has 
	\[
	\mathrm{ker}\, \mathrm{Res}_{a} 
	= \mathrm{im}\, \gamma_{a} \,.
	\]
	
We now restrict to $F^{n}H^{n}(U)$.
Due to the fact that $\gamma_{a}$ respects the Hodge filtration and the description \eqref{eqnFnoHn}, we have
	\[
	\mathrm{im}\, \gamma_{a} \cap F^{n}H^{n}(U) = \Gamma(X,\Omega^{n}_{X}(\log \cup_{b:b\neq a}D_{b})) \,.
	\]
This shows that (recall \eqref{eqnInoHn})
	\[
	\bigcap_{a=1}^{\mathcal{N}} \mathrm{ker}\, \mathrm{Res}_{a} \cap F^{n}H^{n}(U)
	= \bigcap_{a=1}^{\mathcal{N}} \Gamma(X,\Omega^{n}_{X}(\log \cup_{b:b\neq a}D_{b})) 
	= \Gamma(X,\Omega^{n}_{X}) 
	= I^{n,0}H^{n}(U) \,.
	\]

\end{proof}
	
As a generalization of Lemma \ref{lemmakernelofRes}, we have the following result. 
\begin{lemma}\label{lemkerofResa}
	Let the notation be as before. One has 
	\[
	\bigcap_{a=1}^{\mathcal{N}} \mathrm{ker} \,\mathrm{Res}_{a} 
	= \mathrm{im} \left(H^{n}(X, A_{X}^{\bullet})
	\rightarrow H^{n}(X,A_{X}^{\bullet}(\log D))\right) = \bigoplus_{0\leq j \leq n}I^{j,n-j}H^{n}(U) \,.
	\]
\end{lemma}
\begin{proof}
	Since smooth forms on $X$ have no logarithmic poles along any component $D_{a}$, one has 
	\[
	\bigcap_{a=1}^{\mathcal{N}} \mathrm{ker} \,\mathrm{Res}_{a} 
	\supseteq \mathrm{im}( H^n(X, A_{X}^{\bullet}) \to H^n(X,A_{X}(\log D) ) ) \,.
	\]
	Thus we only need to prove the inclusion.
	The proof below is essentially \cite[Proposition 5.17]{Griffiths:1973}.

 Take a closed form $\omega \in A_{X}^{n}(\log D)$ representing its cohomology class $[\omega]_{\mathrm{dR}} \in \bigcap_{a=1}^{\mathcal{N}} \mathrm{ker}\, \mathrm{Res}_{a} $. By construction, all iterated residues of $[\omega]_{\mathrm{dR}}$ vanish since it already vanishes upon taking $\mathrm{Res}_{[1]}$, that is,
	\begin{equation}
	\mathrm{Res}_{[j]} [\omega]_{\mathrm{dR}} = 0, \quad  j \geq 1\,.
	\end{equation}
Start with $\omega \in A_{X}^{n}(\log D) = W_{n}A_{X}^{n}(\log D) $. By 
$\mathrm{Res}_{[n]}[\omega]_{\mathrm{dR}}=0$ and \eqref{eqnE1pageiso}, we see that the image of $[\omega]_{\mathrm{dR}}\in  H^{n}(X,{W}_{n}A^{\bullet}_{X}(\log D))$ along the quotient map
$ H^{n}(X,W_{n}A^{\bullet}_{X}(\log D))\rightarrow H^{n}(X,\mathrm{gr}^{W}_{n}A^{\bullet}_{X}(\log D))$ is zero.
This shows that there exists $\theta_{n} \in W_{n}A^{n-1}_{X}(\log D)$, such that
	\[
	\omega - d\theta_{n} \in W_{n-1}A^{n}_{X}(\log D) \,.
	\]
	Since $d$ commutes with the residue map, we have  
	\[
	\mathrm{Res}_{a} [\omega - d\theta_{n}]_{\mathrm{dR}} 
	= \mathrm{Res}_{a} [\omega]_{\mathrm{dR}} - \mathrm{Res}_{a}[d\theta_{n}]_{\mathrm{dR}} 
	= [-d\mathrm{Res}_{a} \theta_{n}]_{\mathrm{dR}} 
	= 0 \,.
	\]
	Continuing this argument, we see that there exist $\theta_{j} \in W_{j}A^{n-1}_{X}(\log D), j = 1,\cdots,n$, such that
	\[
	\omega - \sum_{j=1}^{n} d\theta_{j} \in W_{0}A^{n}_{X}(\log D) \,.
	\]
Thus $[\omega]_{\mathrm{dR}}$ is represented by $\omega - \sum_{j=1}^{n} d\theta_{j} \in W_{0}A^{n}_{X}(\log D) $. By \eqref{eqnfiltrationsoncohomology} and \eqref{splittingofWandF}, this yields
	\[
	[\omega]_{\mathrm{dR}} \in \mathrm{im} \left(H^{n}(X, W_{0}A^{n}_{X}(\log D))
	\rightarrow H^{n}(X,A_{X}^{\bullet}(\log D))\right)
	= W_{0} H^{n}(U)
	= \bigoplus_{0\leq j \leq n}I^{j,n-j} \,.
	\]	
\end{proof}

\subsection{Cohomology classes of building blocks}
\label{seccohomologyclassofbuildingblocks}

Throughout this part, we set $U=X-\mathsf{D}$, where $X=E_{[n]}$.
By Definition \ref{dfnfunctionspaceassociatedtoD}, for generic values $c_{k,\ell},k,\ell\in [n]$ the divisor $\mathsf{D}$ is a simple normal crossing divisor.
By the construction of $\mathcal{R}_{\mathsf{D}}, \mathcal{V}_{\mathsf{D}}$ in Definition \ref{dfnfunctionspaceassociatedtoD} and Lemma \ref{lemanalyticpropertiesofemz} (ii), one has 
\begin{equation}\label{eqnEkformsarelogforms}
	\mathcal{R}_{\mathsf{D}}\Omega\subseteq  A_{X}^{n}(\star \mathsf{D})\,,\quad 
\mathcal{V}_{\mathsf{D}}\Omega\subseteq  A_{X}^{n}(\log \mathsf{D})\,.
\end{equation}
In this section, we  focus on the space $F^{n}H^{n}(U)$.  The following result shows that it is spanned by $\bar{\partial}$-closed Eisenstein-Kronecker forms in $\mathcal{V}_{\mathsf{D}}\Omega$.
\begin{prop}\label{propFnHndescription}
One has $F^{n}H^{n}(U)=\mathcal{V}_{\mathsf{D}}\Omega\cap  \ker\,d=\mathcal{V}_{\mathsf{D}}\Omega\cap  \ker\,\bar{\partial}$.

\end{prop}
\begin{proof}
    The proof relies on a standard but lengthy analysis using Jacobi theta functions. In order not to interrupt the flow, we defer the proof to Section \ref{secproofofpropFnHndescription}.
\end{proof}

We shall construct
a set of \emph{building blocks} of this space (detailed in  Proposition \ref{propcharacterizationofkerbarpartial} below). We then
compute their cohomology classes and Deligne splittings. 
Properties of these building blocks will be used for the applications in Section \ref{secGWapplication} below.
Discussions on more general forms, as well as on the case when $\mathsf{D}$ is not simple normal crossing (e.g., when $U$ is the configuration space), will appear in a future investigation.

\subsubsection{Building blocks of $d$-closed forms}
Since $\Omega$ is $\bar{\partial}$-closed, one has 
\begin{equation}\label{eqnholomorphicnformandholomorphicfunction}
 \mathcal{V}_{\mathsf{D}}\Omega\cap \ker\,\bar{\partial}=(
\mathcal{V}_{\mathsf{D}}\cap \ker\,\bar{\partial})\Omega\,.
\end{equation}
Thus to cook up $d$-closed  Eisenstein-Kronecker forms in $\mathcal{V}_{\mathsf{D}}\Omega$, it 
suffices to  construct suitable linear combinations of Eisenstein-Kronecker functions in $\mathcal{V}_{\mathsf{D}}$ that are $\bar{\partial}$-closed.
This is carried out below, based on the combinatorial results in Section \ref{seccombinatorialrelations}.

\begin{dfn}\label{dfnbuildingblockPsiMI}
	Let the notation be as in Definition \ref{dfnpsiandPsi}. 	For any sequence $I=(v_{1},\cdots, v_{p})$ of length $p\geq 2$ and $M\in \mathbb{N}$, denote 
	\begin{equation}\label{eqnPsiNIrecalled}
		\Psi_{M,I}:=\sum_{\substack{\mathrm{m}\in \mathbb{N}^{p}\\ |\mathrm{m}|=M}}\psi_{\mathrm{m},I} \,,\qquad \mathrm{where}\quad 		\psi_{\mathrm{m},I} := \prod_{k=1}^{p} \widehat{e}_{m_k}(s_{v_k,v_{k+1}})\in \mathcal{V}_{\mathsf{D}}
		\,.
	\end{equation}	
        Here we use the convention that $\psi_{0, I}=\Psi_{0,I}=1$ for any $I$ with $|I|\geq 1$.
    More generally, pick any sequence $I$ as above, and a sequence $\varepsilon=(\varepsilon_{1},\cdots, \varepsilon_{p})\in \{\pm 1\}^{p}$ such that 
    \begin{equation}\label{eqnepsiloncondition}
    \sum_{k=1}^{p} \varepsilon_{k}\cdot \bar{\partial}\bar{s}_{\varepsilon_{k}(v_{k},v_{k+1})}=0\,,
    \end{equation}
    where $\varepsilon_{k}(v_{k},v_{k+1}):=(v_{k}, v_{k+1})$ if $\varepsilon_{k}=1$ and $\varepsilon_{k}(v_{k},v_{k+1}):=(v_{k+1}, v_{k})$ if $\varepsilon_{k}=-1$.
    Denote
   	\begin{equation}\label{eqnPsiNIepsilon}
\Psi_{M,I}^{(\varepsilon)}:=\sum_{\substack{\mathrm{m}\in \mathbb{N}^{p}\\ |\mathrm{m}|=M}}\psi_{\mathrm{m},I}^{(\varepsilon)}\,,\qquad \mathrm{where}\quad 		\psi_{\mathrm{m},I}^{(\varepsilon)} := \prod_{k=1}^{p} \varepsilon_{k}^{m_{k}}\widehat{e}_{m_k}(s_{\varepsilon_{k}(v_{k},v_{k+1})})\in \mathcal{V}_{\mathsf{D}}
		\,.
	\end{equation}	
\end{dfn}
\begin{lemma}\label{lemmaaboutclosednessofdecompositionofloopterm}
	Let the notation be as in Definition \ref{dfnbuildingblockPsiMI}. Then one has 
	$\bar{\partial}
	\Psi^{(\varepsilon)}_{M,I} =0$.
	
\end{lemma}
\begin{proof}
Consider the generating series 
\[\prod_{i}\varepsilon_{i}\cdot 
\sum_{M\geq 0}c^{M-p}\Psi^{(\varepsilon)}_{M,I}=
\sum_{\mathrm{m}}\psi_{\mathrm{m},I}^{(\varepsilon)}\prod_{i=1}^{p}c^{m_{i}-1}=
\prod_{i=1}^{p}\widehat{S}_{c\varepsilon_{i}}(s_{\varepsilon_{i}(v_{i},v_{i+1})})\,.
\]
Using Definition \ref{dfnSzegokernel}, Definition \ref{dfnhatthetazeta} and \eqref{eqnepsiloncondition}, one has 
\[
\bar{\partial}\ln \prod_{i=1}^{p}\widehat{S}_{c \varepsilon_{i}}(s_{\varepsilon_{i}(v_{i},v_{i+1})})=\sum_{i=1}^{p}
\bar{\partial}\ln \widehat{S}_{c \varepsilon_{i}}(s_{\varepsilon_{i}(v_{i},v_{i+1})})
=c\sum_{i=1}^{p} \varepsilon_{i}\bar{\partial}\mathbold{A}(s_{\varepsilon_{i}(v_{i},v_{i+1})})=0\,.
\]
This gives $ \sum_{M\geq 0}c^{M-p}\bar{\partial}\Psi_{M,I}^{(\varepsilon)}=\bar{\partial}\sum_{M\geq 0}c^{M-p}\Psi_{M,I}^{(\varepsilon)}=0$, thus the desired claim.

\end{proof}
\begin{rem}
	Alternatively, this can be proved directly as follows.
	By  \eqref{eqnbarpartialem} and  \eqref{eqnepsiloncondition}, one has 
	\begin{eqnarray*}
		\bar{\partial}\sum_{\substack{\mathrm{m}\in \mathbb{N}^{p}\\|\mathrm{m}| = M}} 
		\psi_{\mathrm{m},I}^{(\varepsilon)} 
		&=& \sum_{|\mathrm{m}| = M}  \sum_{k:m_k \geq 1} \prod_{\substack{i=1\\ i \neq k}}^{p} \varepsilon_{i}^{m_{i}}\widehat{e}_{m_i}^{m_{i}-1}(s_{\varepsilon_{i}(v_{i},v_{i+1})}) \cdot \varepsilon_{k}^{m_{k}}\mathbold{Y} \widehat{e}_{m_k-1}(s_{\varepsilon_{k}(v_{k},v_{k+1})}) \cdot d\bar{s}_{\varepsilon_{k}(v_{k},v_{k+1})}  \\
		&=& \mathbold{Y} \sum_{|\mathrm{m}| = M-1} \psi_{\mathrm{m},I}^{(\varepsilon)} 
		\cdot \sum_{k=1}^{p} \varepsilon_{k}d\bar{s}_{\varepsilon_{k}(v_{k},v_{k+1})}= 0 \,.
	\end{eqnarray*}
\end{rem}

Lemma \ref{lemmaaboutclosednessofdecompositionofloopterm} implies
\[
\Psi_{M,I}^{(\varepsilon)}\in \mathcal{V}_{\mathsf{D}}\cap \ker\,\bar{\partial}
\,.
\]
The converse is also true and we in fact have the following result, which tells that the functions $\Psi_{M,I}^{(\varepsilon)}$ with $ M\leq |I|-1$ are building blocks of $\bar{\partial}$-closed functions in $\mathcal{V}_{\mathsf{D}}$.
\begin{prop}
    \label{propcharacterizationofkerbarpartial}
   Let the notation be as in Definition \ref{dfnbuildingblockPsiMI}. Then one has
    \[
\mathcal{V}_{\mathsf{D}}\cap \ker\,\bar{\partial}=\mathcal{V}_{\mathsf{D}}^{\mathrm{for}}\cap \mathbb{C}[\Psi_{M,I}^{(\varepsilon)}~|~I\,,~ M\leq |I|-1, \varepsilon]\,.
    \]
\end{prop}
\begin{proof}
   The proof is deferred to Section  \ref{secpropcharacterizationofkerbarpartial}.

\end{proof}

\subsubsection{Leading component $I^{n,0}$}

We make use of the following results from \cite{Zhou:2023cohomologicalpairings, Zhou:2026iteratedintegrals} to extract the leading component  of the cohomology class $[\Psi_{M,I}^{(\varepsilon)}\Omega]_{\mathrm{dR}}$
in the Deligne splitting.

\begin{prop}\label{propleadingcomponent}
Introduce the volume form on $E_{[n]}$ by 
\[
\mathrm{vol}=\bigwedge_{i=1}^{n}{\mathbf{i}\over 2\im \tau}dz_{i}\wedge d\bar{z}_{i}\in A_{X}^{n,n}\,.
\]
\begin{enumerate}[i).]
	\item 
One has 
\begin{equation}\label{eqnHodgeclassasintegral}
[\Psi_{M,I}^{(\varepsilon)}\Omega]_{\mathrm{Hodge}}=\int_{X} \Psi_{M,I}^{(\varepsilon)}\mathrm{vol}\cdot \Omega\,.
\end{equation}
\item 
If in addition $1\leq |M|\leq |I|-1$, then $[\Psi^{(\varepsilon)}_{M,I}\Omega]_{\mathrm{Hodge}}=0$.
\end{enumerate}
\end{prop}
\begin{proof}
	\begin{enumerate}[i).]
		\item 
By \eqref{eqnInoHn} and the dimension reason $H^{0}(X,\Omega_{X}^{n})=\mathbb{C}\Omega$, $[\Psi_{M,I}^{(\varepsilon)}\Omega]_{\mathrm{Hodge}}$ must be a constant multiple of $\Omega$.
	Observe from \eqref{eqnEkformsarelogforms} that $\Psi_{M,I}^{(\varepsilon)}\Omega\in A_{X}^{n}(\log D)$ is indeed integrable.
Applying \cite[Proposition 4.8]{Zhou:2023cohomologicalpairings}, one obtains (i).
\item 
This follows from the proof of \cite[Proposition 3.8 (ii)]{Zhou:2026iteratedintegrals} (where the restriction of the regularized integration to integrable forms reduces to the ordinary integration),
which only uses the Fubini theorem, Stokes theorem and the residue theorem for the integration.
\end{enumerate}
\end{proof}

\subsubsection{Full Deligne splitting}

In addition to the leading component $I^{n,0}$ for the cohomology class $[\Psi_{M,I}^{(\varepsilon)}\Omega]_{\mathrm{dR}}$, we also have the following result regarding the full Deligne splitting.

\begin{prop}\label{proppqtypeinDelignesplitting}
	Let the notation be as in Definition \ref{dfnbuildingblockPsiMI}.
For $I=(1,2,\cdots, n)$ and $0\leq j\leq n-1$, one has 
\[
[\Psi_{j,I}^{(\varepsilon)}\Omega]_{\mathrm{dR}}
  \in I^{n,j}H^{n}(U) \,.
\]
\end{prop}
\begin{rem}
By examining the types, we have  $[\Psi_{j,I}\Omega]_{\mathrm{dR}}\in F^{n}H^{n}(U)$ from \eqref{eqnFnoHn}.
By examining the weight filtration, we  have 
$
[\Psi_{j,I}^{(\varepsilon)}\Omega]_{\mathrm{dR}}
  \in W_{j}H^{n}(U)$.  According to \eqref{splittingofWandF}, we thus have 
  \begin{equation}\label{eqnPsijIdecomposition}
  [\Psi_{j,I}^{(\varepsilon)}\Omega]_{\mathrm{dR}}\in F^{n}H^{n}(U)\cap  W_{j}H^{n}(U)=\bigoplus_{0\leq q\leq j}I^{n,q}H^{n}(U)\,.
  \end{equation}
Proposition \ref{proppqtypeinDelignesplitting} then gives a finer result than this one obtained by type and weight reasons.
\end{rem}
\begin{proof}

	Note that the form $\Psi^{(\varepsilon)}_{j,I}\Omega$ is $\partial$-closed by type reasons.
By Lemma \ref{lemmaaboutclosednessofdecompositionofloopterm}, it is also $\bar{\partial}$-closed. 
	Therefore it is $d$-closed and defines a de Rham cohomology class in $F^{n}H^{n}(U)$.	
	We now prove the claim by induction on the tuple $(n,j)$ as follows.

	For the base case $(n,0)$, we have $\Psi^{(\varepsilon)}_{j,I}=1$ by the definition in \eqref{eqnPsiNIrecalled}. Thus the claim is trivial by the characterization of $I^{n,0}H^{n}(U)$ in \eqref{eqnInoHn}.
Suppose the claim holds for $(n-1,j-1)$ with $j\leq n-1$, we now consider the $(n,j)$ case.
	
For each $1\leq a\leq \mathcal{N}$, consider the residue map \eqref{eqnresiduemaptoD(a)}
	\[
	\mathrm{Res}_{a} : H^n(X, A^{\bullet}_{X}(\log D) ) \to H^{n-1}(D_{a},A^{\bullet}_{D_{a}}(\log D(a)) ) \,.
	\]
	For each $a$, $D_{a}$ is isomorphic to $E_{[n-1]}$ and $D(a)$  is a normal crossing divisor consisting of 
 finitely many  translated diagonals  inside $E_{[n-1]}$.
	Since $d$ commutes with the residue map, we see that $d\mathrm{Res}_{a}(\Psi^{(\varepsilon)}_{j,I}\Omega)=0$.
	Using \eqref{eqnresidueofemdz},
 it is straightforward to check that $\mathrm{Res}_{a}(\Psi^{(\varepsilon)}_{j,I}\Omega)$ is a linear combination of the form $\Psi^{(\varepsilon)}_{j-1, J}\Omega_{J}$, where $J$ is some index set with $|J|=n-1$.
 By the range $j-1\leq n-1$ and the induction hypothesis, one thus has 
 \begin{equation}\label{eqnindcutionhypothesis}
 \mathrm{Res}_{a}(\Psi^{(\varepsilon)}_{j,I}\Omega)\in I^{n-1,j-1}H^{n-1}(D_{a},A^{\bullet}_{D_{a}}(\log D(a)) )\,.
 \end{equation}
  Now using  \eqref{eqnPsijIdecomposition} we write
 \begin{equation}\label{eqnPsijIassumofvq}
  [\Psi^{(\varepsilon)}_{j,I}\Omega]_{\mathrm{dR}}=\sum_{q=0}^{j}v_{q}\in \bigoplus_{0\leq q\leq j}I^{n,q}H^{n}(U)\,.
 \end{equation}
By \eqref{eqnindcutionhypothesis} and the fact that $\mathrm{Res}_{a}$ is a morphism of mixed Hodge structures of type $(-1,-1)$, we have 
 \begin{equation}\label{eqnresidueofvneqj}
 \mathrm{Res}_{a}(\sum_{q\neq j}v_{q})=0\,.
 \end{equation}
This relation is true for any divisor $D_{a}$. Thus we obtain
 \begin{equation}\label{eqnvneqqrange}
\sum_{q\neq j}v_{q}\in \bigcap_{a=1}^{\mathcal{N}}\ker \mathrm{Res}_{a}\,.
 \end{equation}
 Combining \eqref{eqnPsijIassumofvq} and \eqref{eqnvneqqrange}, we are led to
   \[
 [\Psi^{(\varepsilon)}_{j,I}\Omega]_{\mathrm{dR}}\in I^{n, j}H^n(X, A^{\bullet}_{X}(\log D) )+\bigcap_{a=1}^{\mathcal{N}}\ker \mathrm{Res}_{a}  \cap F^{n}H^{n}(U)\,.
 \]
 According to Lemma \ref{lemmakernelofRes}, this gives
  \[
 [\Psi^{(\varepsilon)}_{j,I}\Omega]_{\mathrm{dR}}\in I^{n, j}H^n(X, A^{\bullet}_{X}(\log D) )+\ I^{n, 0}H^n(X, A^{\bullet}_{X}(\log D) )\,.
 \]
By 
Proposition \ref{propleadingcomponent} (ii), we know that the $I^{n,0}$ component must be zero if $j\neq 0$.

In sum, we have shown the claim for the base case $(n,0)$ for all $n$, and the induction step $(n-1,j-1)\Rightarrow (n,j)$.
This concludes the proof.

\end{proof}

\begin{ex}\label{exn=3}
	This example is taken from \cite[Example 4.22]{Zhou:2023cohomologicalpairings}.
	Consider the following function
	\[
	f=\left(\widehat{e}_{1}(s_{12})-\widehat{e}_{1}(s_{13})-\widehat{e}_{1}(s_{32})\right)
	\left(\widehat{e}_{1}(s_{21})-\widehat{e}_{1}(s_{23})-\widehat{e}_{1}(s_{31})\right)\in \mathcal{V}_{\mathsf{D}}\,.
	\]
	We apply the delooping map \eqref{eqnreducedformofED}  (using Corollary \ref{corforlooptermofwidehate_1}) to simplify the loop-terms in $f$
	\begin{eqnarray*}
		\mathfrak{d}(\widehat{e}_{1}(s_{ij})\widehat{e}_{1}(s_{ji}))&=&
		-(\widehat{e}_{2}(s_{ij})+\widehat{e}_{2}(s_{ji}))\\
		&&-\widehat{e}_{1}(-s_{ij}-s_{ji})(\widehat{e}_{1}(s_{ij})+\widehat{e}_{1}(s_{ji}))\\
		&&
		-\widehat{e}_{2}(-s_{ij}-s_{ji})\\
		&=& -(\widehat{e}_{2}(s_{ij})+\widehat{e}_{2}(s_{ji}))+\widehat{e}_{1}(s_{ij}+s_{ji})\Psi_{1, (i,j)}-\widehat{e}_{2}(s_{ij}+s_{ji})\,,\quad i< j\,. 
	\end{eqnarray*}
This leads to the following decomposition for $\mathfrak{d}(f)\in \mathcal{V}_{\mathsf{D}}^{\mathrm{for}}$
	\begin{eqnarray*}
		\mathfrak{d}(f)&=&\underbrace{-\Psi_{2, (123)}-\Psi_{2, (132)}+\Psi_{1, (23)}\Psi_{1, (13)}}_{f_{2}}
		+\underbrace{\sum_{i<j}\widehat{e}_{1}(s_{ij}+s_{ji})\Psi_{1, (i,j)}}_{f_{1}}+\underbrace{-\sum_{i<j}\widehat{e}_{2}(s_{ij}+s_{ji})}_{f_{0}}\,.
	\end{eqnarray*}
	It therefore gives the decomposition in Lemma \ref{lemdclosednessformodularweightgrading}  
	\[
	f\Omega=\sum_{j=0}^{2}f_{j}\Omega\in \bigoplus_{j=0}^{2}M_{j}\Omega\,.
	\]
	By Proposition \ref{proppqtypeinDelignesplitting}, one then has 
	$[f_{j}\Omega]_{\mathrm{dR}}\in I^{n,j}$.
\end{ex}

\section{Application: Gromov-Witten invariants of elliptic curves}
\label{secGWapplication}

This  section  is devoted to applications of the results in earlier sections.
We shall focus on a particular Eisenstein-Kronecker form \cite{Zhou:2023GWgeneratingseries, Zhou:2026iteratedintegrals} that arises from the studies of Gromov-Witten invariants of elliptic curves.
This particular form is given in terms of the determinant of a matrix.
\begin{dfn}\label{dfndeterminantformula}
Let the notation be as in 	
Definition \ref{dfnfunctionspaceassociatedtoD}.
Define 
\begin{equation}\label{eqndeterminantformula}
	\omega_{n}=(\det Z_{n})\Omega\,,\quad 
 Z_n  := \begin{pmatrix}
	0 & 1 & \cdots & 1 \\
	-1 & \widehat{e}_{1}(s_{1,1}) & \cdots &\widehat{e}_{1}(s_{1,n}) \\
	\vdots & \vdots &  & \vdots \\
	-1 & \widehat{e}_{1}(s_{n,1}) & \cdots & \widehat{e}_{1}(s_{n,n}) \\
\end{pmatrix} 
\,.
\end{equation}
\end{dfn}
\begin{lem} Let the notation be as above. Then 
one has \[
\omega_{n} \in \mathcal{V}_{\mathsf{D}}\Omega\,.
\]
\end{lem}
\begin{proof}
One can show that $\det Z_{n}$ coincides with the corresponding determinant with $e_{1}$ in place of $\widehat{e}_{1}$ \cite[Lemma 4.2]{Zhou:2026iteratedintegrals}. Thus $(\det Z_{n})\Omega$ is in fact a meromorphic form.
Expanding $\det\,Z_n$ according to its definition, one has 
\begin{equation}\label{eqndetZnexpansion}
	\mathrm{det} \, Z_n  = \sum_{\sigma \in \mathfrak{S}_{[n]\cup\{0\} } } \mathrm{sgn}(\sigma)\prod_{i=0}^{n} (Z_{n})_{i,\sigma(i)} 
	\,.
\end{equation}
Invoking Definition \ref{dfnfunctionspaceassociatedtoD},  we obtain the claim.
\end{proof}
As mentioned in the introduction, the class $[(\det Z_{n})\Omega]_{\mathrm{Hodge}}$ essentially gives the Gromov-Witten invariants of elliptic curves by 
\cite{Zhou:2026iteratedintegrals, Zhou:2023cohomologicalpairings}.
  See also \cite{Zhou:2023GWgeneratingseries} for related studies.
Furthermore, one has the following combinatorial formula for this class.
\begin{thm}\cite[Lemma 4.2, Theorem 4.1]{Zhou:2026iteratedintegrals}\label{mainthm1Zhou}
	Suppose in 
	Definition \ref{dfnfunctionspaceassociatedtoD} one has
	\begin{equation}\label{eqnassumptiononskl}
		s_{k,\ell}=z_{k}+w_{k}-z_{\ell}\,, \qquad k,\ell\in [n]\,.
	\end{equation}
Denote by $\Pi_{[n]\setminus \{j\}}$
the set of partitions of $[n]\setminus \{j\}$ consisting of elements of the form $\pi=\{\pi_{1},\cdots, \pi_{\ell}\}$.
Then one has
	\begin{equation}\label{eqnTngwpureweightstructure}
	[(\det Z_{n})\Omega]_{\mathrm{Hodge}}=
		\left(\sum_{j\in [n]}\sum_{\pi\in \Pi_{[n]\setminus \{j\}}}
		\prod_{k=1}^{\ell} (|\pi_{k}|-1)!\cdot \widehat{e}_{|\pi_{k}|}(\sum_{i\in \pi_{k}}w_{i})\right)\Omega\,.
	\end{equation}
	
\end{thm}

In this section, we study the cohomology class $[(\det Z_{n})\Omega]_{\mathrm{dR}}$.
 Along the way, we shall construct its Deligne splitting on the form level.

\subsection{Decomposition of the determinant}

To simplify the expression \eqref{eqndetZnexpansion}, we first  split $\det Z_{n}$ into a linear combination of simpler $\bar{\partial}$-closed  forms.
We introduce some notation as follows.

\begin{dfn}\label{dfnPsisigma}
	For $ \sigma \in \mathfrak{S}_{[n]}$, let $ \sigma = \sigma_1\cdots\sigma_{\ell} $ be the standard cycle decomposition of $\sigma$. The underlying set of  $\sigma_{k}$ is usually denoted by $\pi_{k}$ whose cardinality is $|\sigma_{k}|$.		
	For any $\pi_k=\{v_{1},\cdots, v_{m}\} \subseteq [n]$, define the set
	\[
	\mathrm{cyc} (\pi_k) :=\mathfrak{S}_{\pi_{k}}~/~ \mathrm{the~ subgroup ~generated~ by~ the ~permutation}~ (v_{1}v_{2}\cdots v_{m})\,,
	\]
	with cardinality $|	\mathrm{cyc} (\pi_k)|=(|\pi_{k}|-1)!$. This set   can be identified with the subgroup of $\mathfrak{S}_{\pi_{k}}$ consisting of permutations which fix a specified element.
\end{dfn}

	\begin{dfn}\label{dfnPhisigmak}
		 Following Definition \ref{dfnbuildingblockPsiMI}, for the cyclic permutation
$\sigma_k = (a_1\cdots a_p)$ with $p\geq 2$, define
	\begin{equation}
		\psi_{\mathrm{m},\sigma_k} := \prod_{t=1}^{p} \widehat{e}_{m_t}(s_{a_t,\sigma_k(a_{t})})
		\,,
		\quad 
			\Psi_{N,\sigma_k}=\sum_{\substack{\mathrm{m}\in \mathbb{N}^{p}\\|\mathrm{m}|=N}}\psi_{\mathrm{m},\sigma_k}\,.
	\end{equation}
	where we keep the convention that $a_{p+t}=a_t$.
    
	Denote 
	\begin{equation}\label{eqndnfPhir}
		\Phi_{r,\sigma_k} := \widehat{e}_{r}(-\sum_{i\in \sigma_k}s_{i,\sigma_k(i)})
		\sum_{|\mathrm{m}|=|\sigma_{k}|-r} \psi_{\mathrm{m},\sigma_k} \,~\mathrm{for}~ 0\leq r\leq |\sigma_{k}|\,,\quad 
		\Phi_{\geq 1,\sigma_k} := \sum_{r=1}^{|\sigma_{k}|} 	\Phi_{r,\sigma_k} \,.
	\end{equation}
		When $p=1$, say $\sigma_{k}=(a)$, we employ the convention that 
	\begin{equation}\label{eqnconventionPsisigmak}		 	
		\Phi_{\geq 1,\sigma_k}=\widehat{e}_{1}(-s_{a,a})\,.
	\end{equation}	
	For any nonempty set $S$ and any partition $\pi=\{\pi_{1},\cdots, \pi_{\ell}\}\in \Pi_{[n]\setminus S}$, denote 
	\[
			\Phi_{\geq 1,\pi_k} = \sum_{\sigma_k \in \mathrm{cyc} (\pi_k)} \Phi_{\geq 1,\sigma_k}\,,\quad 
\Phi_{\geq 1,\pi}=\prod_{k=1}^{\ell}\Phi_{\geq 1,\pi_{k}}\,.
	\]	
\end{dfn}

We can now express $	\det\,  Z_n$ as a summation of the building blocks $\Psi_{M,I}$ in Definition \ref{dfnbuildingblockPsiMI}, which as a byproduct gives an alternative proof of Theorem 	\ref{mainthm1Zhou}.
\begin{thm} \label{thmdecompositionofdeterminant}
	
	Let the notation be as above.
	\begin{enumerate}[i).]
		\item 

	One has
	\begin{equation}\label{finalexpansionofdetZn}
		\det\,  Z_n = (-1)^{n-1} 
		\sum_{\substack{S \subseteq [n] \\ |S| \geq 1}}  \Psi_{|S|-1,S}
		\sum_{\pi \in \Pi_{[n]\backslash S}} \Phi_{\geq 1,\pi} \,.
	\end{equation}
\item 
One has	\begin{equation}\label{eqnTngwpureweightstructuregeneral}
	[(\det Z_{n})\Omega]_{\mathrm{Hodge}}=
	\left(\sum_{j\in [n]}\sum_{\pi\in \Pi_{[n]\setminus \{j\}}}
	\prod_{k=1}^{\ell} \sum_{\sigma_{k}\in \mathrm{cyc}(\pi_{k})} \widehat{e}_{|\pi_{k}|}(\sum_{i\in \sigma_{k}}s_{i,\sigma_{k}(i)})\right)\Omega\,.
\end{equation}
In particular, Theorem \ref{mainthm1Zhou} holds.
	\end{enumerate}
\end{thm}
\begin{proof}
	See Appendix \ref{secproofofthmdecompositionofdeterminant}. 
\end{proof}

\subsection{Organization by graphs and modular weights}
\label{secmodularweightgradingongraphs}

From convention \eqref{eqnconventionPsisigmak} and Definition \ref{dfntreetypefunctions},
it is straightforward to see that the summands in the decomposition \eqref{finalexpansionofdetZn} in Theorem \ref{thmdecompositionofdeterminant} all lie in the space $ (\mathcal{V}_{\mathsf{D}}^{\mathrm{for}},\mathcal{B}(\mathcal{V}_{\mathsf{D}}^{\mathrm{for}}))$.
The polar divisor of each summand gives rise to an element in $\mathbb{C}\mathcal{G}_{\mathrm{for}}$ according to Definition \ref{dfndivisorgraph} and Definition \ref{dfntreetypefunctions}.
As in \eqref{eqndecomposeTDbygraphs}, one can expand $\det Z_n$ as
\begin{equation}\label{expressionofdetZ_nbygraphs}
	\det  Z_n = (-1)^{n-1} \sum_{\Gamma \in \mathcal{G}} c_{\Gamma} \, \widehat{e}_{\Gamma} \,,\qquad 
    \mathcal{G}:=\Big\{\Gamma\in \mathcal{G}_{\mathrm{for}}~|~c_{\Gamma}\neq 0\Big\}\,,
\end{equation}
where each $c_{\Gamma}$ is a constant independent of the variables $z_{1}, \cdots, z_{n}$.
This graph sum is particularly useful in tracking the weight filtration and modular weight grading of the cohomology class (see Definition \ref{dfnfiltrationongraphspace} below), although the coefficients $c_{\Gamma}$ are complicated.

We now further organize the decomposition 
according to the types of the graphs. This allows us to compute the Deligne splitting later in Section \ref{secDelignecanonicalsplittingonformlevel}.
The key ingredient is the  modular weight grading 
on the space $\mathcal{V}_{\mathsf{D}}^{\mathrm{for}}$, introduced earlier in Definition \ref{dfnmodularweightgradingonTD} in Section
\ref{secmodularweightgradingonTD}.

\begin{lem} \label{lemdecorateddivisorgraphs}
Suppose  $\Gamma\in \mathcal{G}$ as in the expansion \eqref{expressionofdetZ_nbygraphs}.
		\begin{enumerate}[i).]
			\item 
			One has $m(\Gamma)\leq n-1$ and $|E(\Gamma)|\leq m(\Gamma)$.
		\item  
		Each connected component $\Gamma_{k}$ must be a chain or a single vertex. In particular,
		\begin{equation}\label{eqnmGammakminusVGammak}
		|E(\Gamma_{k})|=	|V(\Gamma_{k})| - 1 \leq  m(\Gamma_{k}) \,.
		\end{equation}
	\item 
	 Furthermore, for each connected component $\Gamma_{k}$, the equality holds iff $\Gamma_{k} $
		is either a chain with the decoration on each edge being $1$,  or a single vertex.
	\end{enumerate}
\end{lem}
\begin{proof}
	\begin{enumerate}[i).]
		\item 
		The first relation follows from the decomposition \eqref{finalexpansionofdetZn} in Theorem \ref{thmdecompositionofdeterminant}. The second follows from the fact that $m(e)\geq 1$ for any $e\in E(\Gamma)$.
	\item  
	By convention \eqref{eqnconventionPsisigmak}, when $|S|=1$, the term $\Psi_{|S|-1,S}$ does not contribute any edge.
	Examining the support of the arguments in the decomposition, any connected component $\Gamma_{k}$ arises from a monomial of the form
	$\psi_{\mathrm{m},\sigma_{i}}$ with $\mathrm{m}\leq |\sigma_{i}|-1 $,
	according to  \eqref{eqndnfPhir} in Definition \ref{dfnPhisigmak}. 
		Since $m(e)\geq 1$ for each edge $e\in E(\Gamma_{k})$, this immediately proves the claim.
	\item 
For the degenerate case when $\Gamma_{k}$ is a single vertex, one has $V(\Gamma_{k})=1$ and $m(\Gamma_{k})=0$. The claim is automatic.
We only need to focus on the case where $\Gamma$ is a nontrivial chain,
in which case the claim follows from the fact that $m(e)\geq 1$. 

\end{enumerate}
\end{proof}

To record the geometric data conveniently, we introduce the following filtration and grading on the space $\mathbb{C}\mathcal{G}_{\mathrm{for}}$. 
\begin{dfn}
	\label{dfnfiltrationongraphspace}
	Define the filtration $W_{\bullet}^{\mathcal{G}_{\mathrm{for}}}$ and grading $M_{\bullet}^{\mathcal{G}_{\mathrm{for}}}$ on $\mathbb{C}\mathcal{G}_{\mathrm{for}}$ as 
	\[
	W_{j}^{\mathcal{G}_{\mathrm{for}}} := \mathbb{C}\Big\{\Gamma\in \mathcal{G}_{\mathrm{for}}~|~ |E(\Gamma)| \leq j \Big\}\,,\quad 
	M_{j}^{\mathcal{G}_{\mathrm{for}}} := \mathbb{C}\Big\{\Gamma\in \mathcal{G}_{\mathrm{for}}~|~ m(\Gamma) = j\Big\}\,,
	\]
	with the convention $W^{\mathcal{G}_{\mathrm{for}}}_{j}=0=M^{\mathcal{G}_{\mathrm{for}}}_{j}$ if $j<0$.	
	We call them the weight filtration and modular weight grading, respectively.
	\end{dfn}

For $\Gamma\in \mathcal{G}_{\mathrm{for}}$,
	one has $\widehat{e}_{\Gamma}\Omega\in A_{X}^{n}(\log D)$ by  \eqref{eqnEkformsarelogforms}. By Lemma \ref{lemanalyticpropertiesofemz} (ii), the $W^{\mathcal{G}_{\mathrm{for}}}_{\bullet}$ filtration coincides with the weight filtration $W_{\bullet}$ on $A_{X}^{n}(\log D)$.
	By Remark \ref{remmodularweightgradingonTD}, $M^{\mathcal{G}_{\mathrm{for}}}_{\bullet}$ coincides with the modular weight grading $M_{\bullet}$. 	This explains the nomenclature of Definition 	\ref{dfnfiltrationongraphspace}.

  The structures $W_{\bullet}^{\mathcal{G}_{\mathrm{for}}}, M_{\bullet}^{\mathcal{G}_{\mathrm{for}}}$  on the space
     $\mathbb{C}\mathcal{G}_{\mathrm{for}}$ restrict to
     corresponding structures, denoted by $W^{\mathcal{G}}_{\bullet}, M^{\mathcal{G}}_{\bullet}$,
     on the subspace $\mathbb{C}\mathcal{G}$.
     
     	In what follows, we often identify the above graph-theoretic structures  $W^{\mathcal{G}_{\mathrm{for}}}_{\bullet}, M^{\mathcal{G}_{\mathrm{for}}}_{\bullet}$  on $\mathbb{C}\mathcal{G}_{\mathrm{for}}$ with $W_{\bullet}, M_{\bullet}$ on $\mathcal{V}_{\mathsf{D}}^{\mathrm{for}}$,  without explicit mention.

\subsection{Deligne splitting on form level}
\label{secDelignecanonicalsplittingonformlevel}

We next realize the Deligne splitting of the class $[(\det \,Z_{n})\Omega]_{\mathrm{dR}}$ on form level, following the approach developed in \cite{Zhou:2023cohomologicalpairings}.
Concretely, we shall 
reorganize the summation \eqref{expressionofdetZ_nbygraphs} into a summation of $d$-closed functions that respects the $W_{\bullet}$-filtration and that descends to the Deligne splitting on cohomology level.

\begin{definition}\label{dfnDelignecomponentsonformlevel}
	Let the coefficients $c_{\Gamma},\Gamma\in \mathcal{G}$ be as in \eqref{expressionofdetZ_nbygraphs}.
	Define
		\begin{equation}\label{eqndfnphij}
		\phi_{j}=(-1)^{n-1}\sum_{\Gamma\in M_{j}^{\mathcal{G}} } c_{\Gamma}\widehat{e}_{\Gamma}\Omega\,,\quad 0\leq j\leq n-1\,.
	\end{equation}
	\end{definition}

\begin{lem}
	\label{lemdclosedresiduelift}
	Let the notation be as above.
	\begin{enumerate}[i).]
		\item One has	$\phi_{j}\in W_{j}\Omega$ for any $ 0\leq j\leq n-1$.

\item 
				One has
				\begin{equation}\label{eqdecompositionviagraph}
				(\det  Z_n )  \Omega
				=\sum_{j=0}^{n-1}	\phi_{j}\in \bigoplus_{j=0}^{n-1}M_{j}\Omega\,,\quad 	\bar{\partial}\phi_{j}=0\,.
			\end{equation}

\end{enumerate}
\end{lem}
\begin{proof}
	\begin{enumerate}[i).]
		\item 
 By the relations  in Lemma \ref{lemdecorateddivisorgraphs} (i), we see that $M^{\mathcal{G}}_{j}=0=\mathrm{gr}_{j}W^{\mathcal{G}}$ if $j\geq n$. The modular weight grading  $j=m(\Gamma)$  
		thus satisfies $0\leq j\leq n-1$. 
	The relation $|E(\Gamma)|\leq m(\Gamma)$ in Lemma \ref{lemdecorateddivisorgraphs} (i) also tells that $M^{\mathcal{G}}_{j}\subseteq W^{\mathcal{G}}_{j}$ and thus $\phi_{j}\in W^{\mathcal{G}}_{j}$ indeed.
		\item 
		This follows by 
	applying the modular weight decomposition in Lemma \ref{lemdclosednessformodularweightgrading} (ii)
		to the function $\det\,Z_{n}$ in \eqref{expressionofdetZ_nbygraphs}.
\end{enumerate}
\end{proof}

To find the Deligne splitting of $[\phi_{j}]_{\mathrm{dR}}$, we apply Proposition \ref{proppqtypeinDelignesplitting}. This then requires one to express $\phi_{j}$ in terms of the building blocks $\Psi_{M,I}$  in Definition  \ref{dfnbuildingblockPsiMI}.

\begin{lem}\label{lemexpressgradingjpartintermsofbuldingblocks}
Suppose $S\subseteq [n]$ satisfies $|S|\geq 1$, let 
\[
\tau\in \mathrm{cyc}(S)\,,\quad \sigma=\sigma_{1}\cdots\sigma_{\ell}\in \mathfrak{S}_{[n]\setminus S}\,,\quad 
\mathrm{r}=(r_{1},\cdots, r_{\ell})\in (\mathbb{N}_{+})^{\ell}~\text{with}~r_{k}\leq |\sigma_{k}|~,1\leq k\leq \ell\,.
\]
For each such tuple $(S ,\tau,\sigma,\mathrm{r}) $, denote 
\[
\widehat{e}_{r}(-s_{\sigma_{k}}):=\widehat{e}_{r}(-\sum_{i\in \sigma_k}s_{i,\sigma_k(i)})\,, \]
 and 
\begin{equation}\label{eqnphijbuildingblock}
	\phi_{(S ,\tau,\sigma,\mathrm{r})}:=(-1)^{n-1}\prod_{k=1}^{\ell}   \widehat{e}_{r_{k}}(-s_{\sigma_{k}})\cdot\Psi_{|S|-1,\tau}
	\prod_{k=1}^{\ell} \Psi_{|\pi_{k}|-r_{k},\sigma_{k}}\cdot \Omega\,.
\end{equation}
Then the component $\phi_{j}$ in \eqref{eqndfnphij} satisfies
\begin{equation}\label{eqnphijintermsofbuildingblock}
\phi_{j}=
\sum_{j(S ,\tau,\sigma,\mathrm{r})=j}\phi_{(S ,\tau,\sigma,\mathrm{r})}\,, \quad \mathrm{where}~
\quad j(S ,\tau,\sigma,\mathrm{r}):=n-1-|\mathrm{r}|\,.
\end{equation}

\end{lem}

\begin{proof}
	The identity \eqref{finalexpansionofdetZn} and Definition \ref{dfnPhisigmak} give 
\begin{eqnarray*} 
	(\det\,  Z_n)\Omega &=& (-1)^{n-1} 
	\sum_{\substack{S \subseteq [n] \\ |S| \geq 1}}  \Psi_{|S|-1,S}
	\sum_{\pi \in \Pi_{[n]\backslash S}} \Phi_{\geq 1,\pi}\cdot \Omega\\
		&=& (-1)^{n-1} 	
	\sum_{\substack{S \subseteq [n] \\ |S| \geq 1}}  
	\sum_{\substack{\tau\in\mathrm{cyc}(S)	\\ \sigma\in\mathfrak{S}_{[n]\setminus S}}}
	\Psi_{|S|-1,\tau}\Phi_{\geq 1,\sigma}\cdot \Omega\\	
&=&  (-1)^{n-1} 	
\sum_{\substack{S \subseteq [n] \\ |S| \geq 1}}  
\sum_{\substack{\tau\in\mathrm{cyc}(S)	\\ \sigma\in\mathfrak{S}_{[n]\setminus S}}}
\sum_{r_{1},\cdots r_{\ell}: r_{k}\geq 1}\prod_{k=1}^{\ell}   \widehat{e}_{r_{k}}(-s_{\sigma_{k}})\cdot 
  \Psi_{|S|-1,\tau}
	 \prod_{k=1}^{\ell} \Psi_{|\pi_{k}|-r_{k},\sigma_{k}} \cdot \Omega\\
	 &=&  
	 \sum_{\substack{S \subseteq [n] \\ |S| \geq 1}}  
	 \sum_{\substack{\tau\in\mathrm{cyc}(S)	\\ \sigma\in\mathfrak{S}_{[n]\setminus S}}}
	 \sum_{r_{1},\cdots r_{\ell}: r_{k}\geq 1}\phi_{(S ,\tau,\sigma,\mathrm{r})}\,.
\end{eqnarray*}
Examining the modular weight degree (i.e., the quantity $m(\Gamma)$), we see that 
\begin{equation}\label{eqnbuildingblockweight}
\phi_{(S ,\tau,\sigma,\mathrm{r})}\in M_{j}\Omega\,,\quad \mathrm{with}~j=|S|-1+\sum_{k=1}^{\ell}(|\pi_{k}|-r_{k})=j(S ,\tau,\sigma,\mathrm{r})\,.
\end{equation}
The claim then follows from the direct sum structure of $\oplus_{j}M_{j}$.
\end{proof}
\begin{rem}
	A few remarks are in order.
	\begin{enumerate}[i).]
		\item Having expressed $\phi_{j}$ in terms of the building blocks $\Psi_{M,I}$ in Lemma \ref{lemexpressgradingjpartintermsofbuldingblocks}, 
		the $d$-closedness of $\phi_{j}$ in Lemma	\ref{lemdclosedresiduelift} (ii)  can also be shown from Lemma 	\ref{lemmaaboutclosednessofdecompositionofloopterm}.
	\item For $0\leq j\leq n-1$, set 
	\[
		R_{j}=(-1)^{n-1}\sum_{\Gamma\in M^{\mathcal{G}}_{j} \cap (W^{\mathcal{G}}_{j}\setminus W^{\mathcal{G}}_{j-1})} c_{\Gamma}\widehat{e}_{\Gamma}\,, \quad 
		\varepsilon_{j}=(-1)^{n-1}\sum_{\Gamma\in M_{j}^{\mathcal{G}}\cap   W^{\mathcal{G}}_{j-1}} c_{\Gamma}\widehat{e}_{\Gamma}=\phi_{j}-	R_{j}\,,
		\]
		with the convention that $\varepsilon_{0}=0$.
		Then Lemma  \ref{lemdecorateddivisorgraphs} (iii) tells that for each edge in $\Gamma\in M^{\mathcal{G}}_{j} \cap (W^{\mathcal{G}}_{j}\setminus W^{\mathcal{G}}_{j-1})$, the decoration of each edge is $1$.
		The relation \eqref{eqnresidueofemdz} then shows that  $\mathrm{res}_{[j]}(\widehat{e}_{\Gamma}\Omega)\neq 0$. Similarly, $\mathrm{res}_{[j]}(\widehat{e}_{\Gamma}\Omega)= 0$ for each $\Gamma\in M^{\mathcal{G}}_{j}\cap W^{\mathcal{G}}_{j-1}$. That is, the term $R_{j}$ collects all terms inside $W^{\mathcal{G}}_{j}$ that have nontrivial $\mathrm{res}_{[j]}$-residues. The $\mathrm{res}_{[j]}$-free correction term $\varepsilon_{j}$ is added to make $\phi_{j}=R_{j}+\varepsilon_{j}$ a $d$-closed form. See \cite[Section 4.3]{Zhou:2023cohomologicalpairings} for related discussion.

	\end{enumerate}
\end{rem}

We can now show that 
the splitting in Lemma	\ref{lemdclosedresiduelift} (i) provides the desired Deligne splitting on the form level.

\begin{thm}\label{thmDelignecanonicalsplittingonformlevel}
Let the notation be as above.
	Then the decomposition 
\[
(\det  Z_n )\Omega  
=\sum_{j=0}^{n-1}	\phi_{j}\in \bigoplus_{j}M_{j}\,.
\]
gives the Deligne splitting of $[	(\det  Z_n )\Omega  ]_{\mathrm{dR}}$  on the form level. In particular, one has 
\[
[(\det Z_{n})\Omega]_{\mathrm{Hodge}}
=[\phi_{0}]_{\mathrm{dR}}\,.
\]
\end{thm}
\begin{proof}
	First, by Lemma	\ref{lemdclosedresiduelift} (ii), each $\phi_{j}$ is $d$-closed and 
	thus defines a cohomology class. We now show that $[\phi_{j}]_{\mathrm{dR}}\in I^{n,j}$ for any $0\leq j\leq n-1$.
	
For each $\phi_{(S ,\tau,\sigma,\mathrm{r})}$ in \eqref{eqnphijbuildingblock},
	observe that the variables in different factors in 
the factorization are disjoint: the argument $s_{\sigma_{k}}$ in $\widehat{e}_{r_{k}}(-s_{\sigma_{k}})$
is constant, while the remaining 
$\Psi$-factors live on disjoint vertex sets.
Suppose the polar divisor of $\phi_{(S ,\tau,\sigma,\mathrm{r})}$ is $\mathsf{D}_{\phi}$, then $E_{[n]}-\mathsf{D}_{\phi}$ factors according to the disjoint vertex blocks.
By the functoriality of mixed Hodge structure on Künneth isomorphism, this induces a factorization of its cohomology class in
\[
[\phi_{(S ,\tau,\sigma,\mathrm{r})}]_{\mathrm{dR}}\in H^{n}(E_{[n]}, A_{E_{[n]}}^{n}(\log \mathsf{D}))\,.
\]
Applying Proposition 	\ref{proppqtypeinDelignesplitting}, we thus have 
\[
[\phi_{(S ,\tau,\sigma,\mathrm{r})}]_{\mathrm{dR}}\in I^{n, j}H^{n}(U)\,,\quad 
j=|S|-1+\sum_{k=1}^{\ell}(|\pi_{k}|-r_{k})=j(S ,\tau,\sigma,\mathrm{r})\,.
\]
By \eqref{eqnphijintermsofbuildingblock} in Lemma \ref{lemexpressgradingjpartintermsofbuldingblocks}, this implies the desired claim for $\phi_{j}$.

\end{proof}
\begin{rem}
By Theorem \ref{thmDelignecanonicalsplittingonformlevel}  and 
\eqref{eqnbuildingblockweight},
we see that the $I^{n,0}$ part collects the contributions from the tuples $(S ,\tau,\sigma,\mathrm{r})$ that satisfy $j(S ,\tau,\sigma,\mathrm{r})=0$, that is,
\[
|S|=1\,,\quad r_{k}=|\pi_{k}|\,.
\]
This is exactly what singles out the
contribution \eqref{eqnTngwpureweightstructuregeneral} from \eqref{finalexpansionofdetZn}
given in Theorem \ref{thmdecompositionofdeterminant} (ii), see its proof in Theorem \ref{thmdecompositionofdeterminantrecalled} (ii).
However, here the Deligne splitting on form level makes the computations more transparent.
\end{rem}

	\begin{example}\label{exn=2}
	For $n=2$, one has 
	\begin{eqnarray*}
		\det\, Z_2=\begin{pmatrix}
			0 & 1 & 1  \\
			-1 & \widehat{e}_1(s_{1,1}) & \widehat{e}_1(s_{1,2}) \\
			-1 & \widehat{e}_1(s_{2,1}) & \widehat{e}_1(s_{2,2}) \\
		\end{pmatrix} 
		= \widehat{e}_1(s_{1,1}) + \widehat{e}_1(s_{2,2}) 
		- \widehat{e}_1(s_{1,2}) - \widehat{e}_1(s_{2,1})\,.
	\end{eqnarray*}
	Here $\widehat{e}_1(s_{1,1}) + \widehat{e}_1(s_{2,2})$ is the constant term, while $- \widehat{e}_1(s_{1,2}) - \widehat{e}_1(s_{2,1})$ is the residue term and is $\bar{\partial}$-closed.
	
Deligne's splitting in  Theorem \ref{thmDelignecanonicalsplittingonformlevel} tells  
	\[
	(\widehat{e}_1(s_{1,1}) + \widehat{e}_1(s_{2,2})) dz_{1} \wedge dz_{2} \in I^{2,0} \,,\quad 
	(-\widehat{e}_1(s_{1,2}) - \widehat{e}_1(s_{2,1})) dz_{1} \wedge dz_{2} \in I^{2,1} \,.
	\]
	The first relation is trivial according to \eqref{eqnInoHn}.
	We now provide a check of the second relation.
	Introduce (cf. \eqref{eqnhatthetazeta})
	\[
	|\theta|^{2}:=\theta\bar{\theta}\exp\left(-2\pi {(\mathrm {im}~z)^2\over \mathrm{im}~\tau}\right)\,.
	\]
	This quantity  gives a smooth function on $E$, as can be checked using the automorphy \eqref{eqnautomorphyoftheta}. In fact, it corresponds to a Hermitian metric on the line bundle $\mathcal{O}_{E}(0)$ which has $\theta$ as its holomorphic section.
	By \eqref{eqnfirstfewhatems}, one has $\widehat{e}_{1}(z)dz=\partial \log |\theta|^{2}$.
	We now write 
	\begin{eqnarray*}
	&&	(-\widehat{e}_1(s_{1,2}) - \widehat{e}_1(s_{2,1})) dz_{1} \wedge dz_{2}\\
	 &=& -\partial (\log |\theta|^2(s_{1,2}) dz_{2}) + \partial (\log |\theta|^2(s_{2,1}) dz_{2}) \\
		&=& \overline{\widehat{e}_{1}}(s_{1,2}) d\bar{z}_{1,2} \wedge dz_{2} - \overline{\widehat{e}_{1}}(s_{2,1}) d\bar{z}_{2,1} \wedge dz_{2} 
		- d (\log |\theta|^2(s_{1,2}) dz_{2}) +d (\log |\theta|^2(s_{2,1}) dz_{2})\,. 
	\end{eqnarray*}
Using the $C^{\infty}$ logarithmic Dolbeault complex  \cite{Burgos:1994} and \eqref{eqndfnDelignsplitting}, 
one obtains indeed
	\[
	(-\widehat{e}_1(s_{1,2}) - \widehat{e}_1(s_{2,1})) dz_{1} \wedge dz_{2} \in (F^{2}\cap W_{1}) \cap (\overline{F}^{1}\cap W_{1}) = I^{2,1} \,.
	\]

\end{example}

\section{Proofs of  Proposition \ref{propFnHndescription} and Proposition \ref{propcharacterizationofkerbarpartial}}
 \label{secpropFnHndescription}
 
This section is devoted to the proofs of Proposition \ref{propFnHndescription} and   Proposition \ref{propcharacterizationofkerbarpartial}.

To proceed, we introduce some notation.
For $k\neq \ell$, denote the corresponding edge by  $(k,\ell)$  or $e_{k\ell}$  interchangeably.  Let $\mathrm{V}=[n],\mathrm{E}$ be the set of all vertices and edges, respectively; and equip $\mathbb{C}^{\mathrm{V}}, \mathbb{C}^{\mathrm{E}}$ with the bases $e_{k}, k\in \mathrm{V}$ and $e_{k\ell}, (k,\ell)\in \mathrm{E}$, respectively.
Construct the incidence map 
\begin{equation}\label{eqnincidencemap}
A: \mathbb{C}^{\mathrm{E}}\rightarrow \mathbb{C}^{\mathrm{V}}\,,\quad 
e_{k\ell}\mapsto e_{k}-e_{\ell}\,.
\end{equation}

Consider the hyperplane $W:=\{\sum_{k\in \mathrm{V}}\lambda_{k}e_{k}=0\}$ in $\mathbb{C}^{\mathrm{V}}$. It is
isomorphic to the quotient space $\mathbb{C}^{\mathrm{V}}/\mathbb{C}\sum_{k\in \mathrm{V}}e_{k}$ and is
naturally equipped with the generating set  $e_{k}-e_{\ell},k\neq \ell$. 
  Clearly  $A$ is surjective onto the subspace $W$.  In particular, one has 
  \[
  \dim \ker A=|\mathrm{E}|-(n-1)=(n-1)^{2}\,.
  \]
 The space $\ker A$ is the usual circulation space of directed graphs.
  A basis of $\ker A$ can be chosen such that the entries of each basis vector lie in $\{0, \pm 1\}$.
    It is straightforward to show that such a $\mathbb{Z}$-basis of $\ker A$ is given by
  \[
  \gamma_{k}:=
  e_{kn}+e_{nk}\,,~ 1\leq k\leq n-1\,,\quad 
 \gamma_{ij}:= e_{ij}+e_{jn}+e_{ni}\,,~\quad 
  1\leq i\neq j\leq n-1\,.
  \]

\subsection{Proof of  Proposition \ref{propcharacterizationofkerbarpartial}}
 \label{secpropcharacterizationofkerbarpartial}

We now give the proof of Proposition \ref{propcharacterizationofkerbarpartial}, whose statement is recalled below for reference.
\begin{prop}    \label{propcharacterizationofkerbarpartialrecalled}
      Take any sequence $I=(v_{1},\cdots, v_{p})$ of length $p\geq 2$, $M\in \mathbb{N}$,   and a sequence $\varepsilon=(\varepsilon_{1},\cdots, \varepsilon_{p})\in \{\pm 1\}^{p}$ such that 
    \begin{equation} 
    \sum_{k=1}^{p} \varepsilon_{k}\cdot \bar{\partial}\bar{s}_{\varepsilon_{k}(v_{k},v_{k+1})}=0\,,
    \end{equation}
    where $\varepsilon_{k}(v_{k},v_{k+1}):=(v_{k}, v_{k+1})$ if $\varepsilon_{k}=1$ and $\varepsilon_{k}(v_{k},v_{k+1}):=(v_{k+1}, v_{k})$ if $\varepsilon_{k}=-1$.
    Define
   	\begin{equation}\label{eqndfnofPsiMIvarepsilonrecalled}
\Psi_{M,I}^{(\varepsilon)}:=\sum_{\substack{\mathrm{m}\in \mathbb{N}^{p}\\ |\mathrm{m}|=M}}\psi_{\mathrm{m},I}^{(\varepsilon)}\,,\qquad \mathrm{where}\quad 		\psi_{\mathrm{m},I}^{(\varepsilon)} := \prod_{k=1}^{p} \varepsilon_{k}^{m_{k}}\widehat{e}_{m_k}(s_{\varepsilon_{k}(v_{k},v_{k+1})})\in \mathcal{V}_{\mathsf{D}}
		\,.
	\end{equation}	   
   Then one has
    \[
\mathcal{V}_{\mathsf{D}}\cap \ker\,\bar{\partial}=\mathcal{V}_{\mathsf{D}}^{\mathrm{for}}\cap \mathbb{C}[\Psi_{M,I}^{(\varepsilon)}~|~I\,,~ M\leq |I|-1, ~\varepsilon]\,.
    \]
\end{prop}

\begin{proof}

Following Definition \ref{dfngraphterminologies}, 
and \eqref{eqneGamma}, 
we can write 
\[
\widehat{e}_{\Gamma}=\prod_{e\in E(\Gamma)}\widehat{e}_{m(e)}(s_{e})\in \mathcal{V}_{\mathsf{D}}\,.
\]
Set also 
\[
\widehat{e}_{\Gamma-1_{e}}:=\prod_{\substack{e'\in E(\Gamma)\\ e'\neq e}}\widehat{e}_{m(e')}(s_{e'})\cdot \widehat{e}_{m(e)-1}(s_{e})\,~\mathrm{for}~e\in E(\Gamma)\,.
\]
Then by Lemma \ref{lemanalyticpropertiesofemz} (iii), one has 
\begin{equation}\label{eqnbarpartialoneGamma}
\bar{\partial}\widehat{e}_{\Gamma}=\mathbold{Y}\sum_{e\in E(\Gamma)}
\widehat{e}_{\Gamma-1_{e}}d\bar{s}_{e}\,.
\end{equation}

Define the following map first on 
$\mathcal{B}(\mathcal{V}_{\mathsf{D}}^{\mathrm{for}})$ then linearly extend it to  $\mathcal{V}_{\mathsf{D}}^{\mathrm{for}}$
\begin{eqnarray}\label{eqnrhomapforkernelconsideration}
\rho: 
\mathcal{V}_{\mathsf{D}}^{\mathrm{for}}
\otimes (\mathbb{C}\oplus  \sum_{e\in \mathrm{E}}\mathbb{C}d\bar{s}_{e})
&\rightarrow &
 \mathbb{C}[x_{e}, e\in \mathcal{E}]\otimes (\mathbb{C}\oplus \mathbb{C}^{\mathrm{V}})\,,\\
 \prod_{e} \widehat{e}_{m_{e}}(s_{e})
 \otimes (\mu+\sum_{e} \lambda_{e}d\bar{s}_{e})
 &\mapsto & \prod_{e}{x_{e}^{m(e)}\over m(e)!}
 \otimes (\mu+\sum_{e} \lambda_{e} A(e))\,.\nonumber
\end{eqnarray}
Using  $d\bar{s}_{e}=d\overline{z}_{k}-d\overline{z}_{\ell}$,  it is straightforward to see that $\rho$ is injective.
By \eqref{eqnbarpartialoneGamma}, we see that under the map $\rho$, the derivation $\bar{\partial}$ is identified with the derivation 
\begin{equation}
\delta=\sum_{e\in \mathrm{E}}{\partial\over \partial x_{e}}\otimes  A(e) :  \mathbb{C}[x_{e}, e\in \mathcal{E}]
\rightarrow \mathbb{C}[x_{e}, e\in \mathcal{E}]\otimes \mathbb{C}^{\mathrm{V}}\,.
\end{equation}
Restricting $\delta$ to the vector space $V:=\oplus_{e\in \mathrm{E}}\mathbb{C}x_{e}$,  we see that 
\begin{equation}\label{eqnkerA}
K:=\ker \delta|_{V}=
\Big\{ \sum_{e}\lambda_{e}x_{e}\in V~|~ \sum_{e}A(e)\lambda_{e}=0
\Big\}\cong \ker A\,.
\end{equation}

Picking a basis $y_{1},\cdots, y_{p}$ for $K$  and then extending it to a full basis $y_{1},\cdots ,y_{p},y_{p+1},\cdots, y_{p+q}$  of $V$, one obtains a splitting 
\[
V=K\oplus \bigoplus_{i=1}^{q}\mathbb{C}y_{p+i}\,.
\]
With respect to this basis, one has 
\begin{equation}
\delta|_{V}=\sum_{i=1}^{q}{\partial\over \partial y_{p+i}}\otimes \delta y_{p+i}\,.
\end{equation}
By the construction of $K$, the quantities $\delta y_{p+i},1\leq i\leq q$ correspond to linearly independent 1-forms. Due to the fact that $\delta$ is a derivation and $\rho$ is an injection, we obtain 
\begin{equation}\label{eqnkerbarpartialassymK}
\ker \delta=\mathrm{Sym}K=\mathbb{C}[y_{1},\cdots ,y_{p}]\,,\quad 
\ker \bar{\partial}=\rho^{-1}(\mathrm{Sym}\,K)\,.
\end{equation}
Concretely, suppose 
\[
y_{i}=\sum_{e}\varepsilon_{e}x_{e},1\leq i\leq p\,,\quad 
\varepsilon=(\varepsilon_{e})_{e}\in \{0, \pm 1\}^{\mathrm{E}}\in \ker A\,.
\]
Then one has 
\[
{y_{i}^{M}\over M!}=\sum_{m(e): \sum_{e}m(e)=M}\prod_{e}\varepsilon_{e}^{m_{e}} {x_{e}^{m(e)}\over m(e)!}\,,
\]
and thus 
\begin{equation}
\rho^{-1}({y_{i}^{M}\over M!})
=\sum_{m(e): \sum_{e}m(e)=M}\prod_{e}\varepsilon_{e}^{m_{e}} \widehat{e}_{m_{e}}(s_{e})\,.
\end{equation}
Combining \eqref{eqnkerbarpartialassymK} and Definition \ref{dfnbuildingblockPsiMI}, we thus obtain 
\begin{equation}
\mathcal{V}^{\mathrm{for}}_{\mathsf{D}}\cap \ker\,\bar{\partial}\subseteq \mathcal{V}^{\mathrm{for}}_{\mathsf{D}}\cap\mathbb{C}[\Psi_{M,I}^{(\varepsilon)}~|~I,M, \varepsilon]\,.
\end{equation}
By Lemma \ref{lemreductionformulaforPsiMI} below, 
we obtain 
\begin{equation}
\mathcal{V}^{\mathrm{for}}_{\mathsf{D}}\cap \ker\,\bar{\partial}\subseteq \mathcal{V}^{\mathrm{for}}_{\mathsf{D}}\cap\mathbb{C}[\Psi_{M,I}^{(\varepsilon)}~|~I,M\leq |I|-1, \varepsilon]\,,
\end{equation}
where  $\Psi_{M,I}^{(\varepsilon)}$ involves no loops iff $M\leq |I|-1$ according to  Definition \ref{dfntreetypefunctions}.
Finally,  by Lemma \ref{lemexpressionbytreeterms} we see that as vector spaces $\mathcal{V}_{\mathsf{D}}=\mathcal{V}_{\mathsf{D}}^{\mathrm{for}}$,  and thus $\mathcal{V}_{\mathsf{D}}\cap \ker\,\bar{\partial}=\mathcal{V}^{\mathrm{for}}_{\mathsf{D}}\cap \ker\,\bar{\partial}$.
\end{proof}

\begin{lem}\label{lemreductionformulaforPsiMI}
The functions $\Psi_{M,I}^{(\varepsilon)},M \geq |I|=p$ can be expressed in terms of the building blocks $\Psi_{M,I}^{(\varepsilon)},M \leq p-1$ explicitly as 
	\begin{equation}\label{eqnPsiMIlargerintermsofsmaller}
	\Psi^{(\varepsilon)}_{M,I} =- \sum_{M'=0}^{p-1} \binom{M'-p}{M-p} \cdot \widehat{e}_{M-M'}(-\sum_{i=1}^{p} s_{\varepsilon_{i}(v_i,v_{i+1}})) \Psi^{(\varepsilon)}_{M',I} \,.
	\end{equation}
    \end{lem}
    \begin{proof}
Invoking Proposition \ref{propgeneralpropofloopterm}, one has
    \begin{eqnarray*}
        \Psi_{M,I} &=& \sum_{\substack{\mathrm{m}\in \mathbb{N}^{p}\\ |\mathrm{m}|=M}} \left(\delta_{\mu_{0}(\mathrm{m})\neq 0}\psi_{\mathrm{m},I} - 
        \delta_{\mu_{0}(\mathrm{m})=0}
	\sum_{\substack{m_0'+|\mathrm{m}'|=M\\ \mu_0(\mathrm{m}')\neq 0}}
	\prod_{i=1}^{p} \binom{m_i'-1}{m_i-1} \cdot
	\widehat{e}_{m_0'}(-\sum_{i=1}^{p}s_{v_i,v_{i+1}})    \psi_{\mathrm{m}',I} \right) \\
    &=& \sum_{\substack{|\mathrm{m}|=M\\\mu_0(\mathrm{m})\neq 0}} \psi_{\mathrm{m},I} - \sum_{\substack{|\mathrm{m}|=M\\\mu_0(\mathrm{m})= 0}} 
    \left(\sum_{\substack{m_0'+|\mathrm{m}'|=M\\ \mu_0(\mathrm{m}')\neq 0}}
	\prod_{i=1}^{p} \binom{m_i'-1}{m_i-1} \cdot
	\widehat{e}_{m_0'}(-\sum_{i=1}^{p}s_{v_i,v_{i+1}})    \psi_{\mathrm{m}',I}\right) \\
    &=& \sum_{\substack{|\mathrm{m}|=M\\\mu_0(\mathrm{m})\neq 0}} \psi_{\mathrm{m},I} - 
    \sum_{m_0'=0}^{M} \widehat{e}_{m_0'}(-\sum_{i=1}^{p}s_{v_i,v_{i+1}}) 
    \sum_{\substack{|\mathrm{m}'|=M-m_0'\\ \mu_0(\mathrm{m}')\neq 0}} \sum_{\substack{|\mathrm{m}|=M\\\mu_0(\mathrm{m})= 0}} \prod_{i=1}^{p} \binom{m_i'-1}{m_i-1} \cdot \psi_{\mathrm{m}',I} \,.
    \end{eqnarray*}
    Using the generating function technique, we can simplify the coefficient of $\psi_{\mathrm{m}',I}$ as follows.
One has 
\begin{eqnarray*}
\sum_{\substack{|\mathrm{m}|=M\\\mu_0(\mathrm{m})= 0}} \prod_{i=1}^{p} \binom{m_i'-1}{m_i-1} &=&  [t^{M-p}] \prod_{i=1}^{p} (1+t)^{m_i'-1} 
= [t^{M-p}] (1+t)^{M-m_0'-p} =
    \binom{M-m_0'-p}{M-p}\,.
\end{eqnarray*}
Here the notation $[t^{M-p}]$ represents the degree $t^{M-p}$ coefficient, and the usual  generalized binomial coefficients are used.
Putting $M'=M-m_0'$, we therefore obtain
\[
	\Psi_{M,I}= \sum_{\substack{|\mathrm{m}|=M\\\mu_0(\mathrm{m})\neq 0}} \psi_{\mathrm{m},I} - 
    \sum_{M'=0}^{M} \widehat{e}_{M-M'}(-\sum_{i=1}^{p}s_{v_i,v_{i+1}}) 
    \sum_{\substack{|\mathrm{m}'|=M'\\ \mu_0(\mathrm{m}')\neq 0}}  \binom{M'-p}{M-p}\psi_{\mathrm{m}',I}\,.
\]
The $M'=M$ term in the summation above gets canceled, and one obtains 
\[
	\Psi_{M,I}=  - 
    \sum_{M'=0}^{M-1} \widehat{e}_{M-M'}(-\sum_{i=1}^{p}s_{v_i,v_{i+1}}) 
     \binom{M'-p}{M-p}\Psi_{M',I}\,.\]
Since $M'<M$ and $M-p\geq 0$, the nonvanishing of the generalized binomial coefficient above forces $M'\leq p-1$.
This gives the desired claim for the $\varepsilon=\mathbold{1}$ case.

By Lemma \ref{lemanalyticpropertiesofemz} (i) and \eqref{eqndfnofPsiMIvarepsilonrecalled}, one has 
\[
\varepsilon_i^{m_i}\widehat{e}_{m_i}(s_{\varepsilon_i(v_i,v_{i+1})}) = \widehat{e}_{m_i}(\varepsilon_i s_{\varepsilon_i(v_i,v_{i+1})})\,.
\]
Replacing $s_{v_{i},v_{i+1}}$ by $\varepsilon_{i}s_{\varepsilon_{i}(v_{i},v_{i+1})}$, 
the general case follows from the $\varepsilon=\mathbold{1}$ case directly.
\end{proof}

\subsection{Proof of Proposition \ref{propFnHndescription}}
 \label{secproofofpropFnHndescription}

We first relate the space of global logarithmic forms to the familiar space of Jacobi theta functions, then to Eisenstein-Kronecker forms.
Our main references for this part are \cite[Chapter 2]{Mumford:tataI}
and \cite[Chapter 2.6]{Griffiths:1994}.
 
Using the fact $\Omega^{n}_{E_{[n]}}\cong \mathcal{O}_{E_{[n]}}$
and the trivialization in \eqref{eqnholomorphicvolumeform}, one has 
\begin{equation}\label{eqnidentificationofH0logwithlinearsystem}
H^{0}(E_{[n]},\Omega^{n}_{E_{[n]}}(\log \mathsf{D}))
=\mathcal{L}(D)\Omega\,,
\end{equation}
where $\mathcal{L}(\mathsf{D}):=\{f\in \mathbb{C}(E_{[n]})~|~\mathrm{div}\,f+\mathsf{D}\geq 0\}$.
Introduce the translation map
\begin{equation}\label{eqntranslationmap}
\phi: 
E_{[n]}\rightarrow E^{n-1}\times E_{n}\,,\quad 
(z_{1},\cdots, z_{n-1},z_{n})
\mapsto (w,z_{n}):=(z_{1}-z_{n},\cdots, z_{n-1}-z_{n},z_{n})\,.
\end{equation}
Below we introduce the convention $w_{n}:=0$.
Under the isomorphism $\phi$, one has 
\begin{equation}\label{eqntkl}
s_{k,\ell}=\phi^{*}t_{k,\ell}\,,\quad 
t_{k,\ell}:=
  w_{k}-w_{\ell}+c_{k,\ell}\,.
\end{equation}
The divisor $\mathsf{D}$ is the pull-back of the corresponding divisor $\overline{\mathsf{D}}\subseteq E^{n-1}$ along the projection $p_{1}: E^{n-1}\times E_{n}\rightarrow E^{n-1}$. Explicitly, the irreducible component $\mathsf{D}_{k,\ell}$ corresponds to 
\begin{equation}\label{eqntranslateddivisor}
\overline{\mathsf{D}}_{k,\ell}=(t_{k,\ell}=0)=t_{k,\ell}^{*}(z=0)\,,
\end{equation}
where in the notation $t_{k,\ell}^{*}$ the quantity $t_{k,\ell}$ is regarded as a morphism $E^{n-1}\rightarrow E$.
Combining \eqref{eqntranslationmap}  and \eqref{eqntranslateddivisor}, the Künneth formula gives
\begin{equation}\label{eqnODfactorization}
\mathcal{L}(\mathsf{D})\cong 
H^{0}(E^{n-1}, \mathcal{O}_{E^{n-1}}(\overline{\mathsf{D}}))\,,\quad 
\mathcal{O}_{E^{n-1}}(\overline{\mathsf{D}})\cong \bigotimes_{k\neq \ell}t_{k,\ell}^{*}\mathcal{O}_{E}(0)\,.
\end{equation}
Thanks to \eqref{eqnidentificationofH0logwithlinearsystem} and \eqref{eqnODfactorization}, the space of logarithmic forms is identified with global sections of the line bundle $\bigotimes_{k\neq \ell}t_{k,\ell}^{*}\mathcal{O}_{E}(0)$ on the Abelian variety $E^{n-1}$.

\subsubsection{Incidence matrix and polarization matrix}

We next compute the polarization matrix of the line bundle $\bigotimes_{k\neq \ell}t_{k,\ell}^{*}\mathcal{O}_{E}(0)$ on $E^{n-1}$.

The incidence map in \eqref{eqnincidencemap} induces a map
\begin{equation} \label{eqntranslatedincidencemap}
B: \mathbb{C}^{\mathrm{E}}\rightarrow W\,,\quad 
e_{k\ell}\mapsto e_{k}-e_{\ell}\,.
\end{equation}
Denote by $B_{\mathbb{R}},B_{\mathbb{Z}}$  the real and integral form of $B$, respectively.

Identify the linear map $B$ with the matrix whose columns are given by $B(e)$, denoted by $B\in M_{(n-1)\times n(n-1)}(\mathbb{C})$ by abuse of notation.
By \eqref{eqntkl}, one has 
\begin{equation}\label{eqnODfactorizationcomponentmap}
t_{e}=w^{t}\cdot B(e)+c_{e}\,,\quad \text{for}~e=e_{k\ell}\,.
\end{equation}
It follows from \eqref{eqnODfactorization} and  \eqref{eqnODfactorizationcomponentmap}
that the polarization matrix of $\mathcal{O}_{E^{n-1}}(\overline{\mathsf{D}})$ is given by
\begin{equation}\label{eqnpolarizationmatrix}
M=\sum_{e\in \mathrm{E}}B_{e}B_{e}^{t}=BB^{t}\,.
\end{equation}
Let $\mathbold{1}\in \mathbb{R}^{n-1}$ be the column vector whose entries are all $1$.
Then based on \eqref{eqntranslatedincidencemap} and \eqref{eqnpolarizationmatrix}, one sees that 
\begin{equation}\label{eqndetM}
M=2(n \cdot \mathrm{id}-\mathbold{1}\cdot\mathbold{1}^{t})\,,\quad 
\det M=2^{n-1}n^{n-2}\,.
\end{equation}

\subsubsection{Dimension from Hirzebruch-Riemann-Roch}

We now compute the dimension of $\bigotimes_{k\neq \ell}t_{k,\ell}^{*}\mathcal{O}_{E}(0)$ based on Hirzebruch-Riemann-Roch.

From \eqref{eqnpolarizationmatrix}  one sees that
\[
v^{t}Mv=vBB^{t}v=\sum_{e}(B_{e}^{t}v)^{t}(B_{e}^{t}v)>0\,,\quad \text{if}~v\neq 0\,,
\]
since the $B_{e}$'s span the full space $\mathbb{R}^{n-1}$.
According to the Appell-Humbert description of line bundles on complex tori, this shows that $\mathcal{O}_{E^{n-1}}(\overline{\mathsf{D}})$ is an ample\footnote{In contrast, the line bundle $\mathcal{O}_{E_{[n]}}(\mathsf{D})$ is not ample, due to the presence of the small diagonal $\phi^{-1}E_{n}$ in \eqref{eqntranslationmap} which has trivial intersection with $\mathsf{D}$.} line bundle on the Abelian variety $E^{n-1}$.
Kodaira vanishing tells that higher cohomology groups of $\mathcal{O}_{E^{n-1}}(\overline{\mathsf{D}})$ all vanish. By Hirzebruch–Riemann–Roch and \eqref{eqnODfactorization}, one thus has 
\begin{eqnarray}\label{eqnHRR}
h^{0}(E^{n-1}, \mathcal{O}_{E^{n-1}}(\overline{\mathsf{D}}))&=&\chi(E^{n-1}, \mathcal{O}_{E^{n-1}}(\overline{\mathsf{D}}))\nonumber\\
&=&{1\over (n-1)!} c_{1} (\mathcal{O}_{E^{n-1}}(\overline{\mathsf{D}}))^{n-1}\nonumber\\
&=&{1\over (n-1)!}  \left(\sum_{e}c_{1}(t_{e}^{*}\mathcal{O}_{E}(0))\right)^{n-1}\,.
\end{eqnarray}
A direct computation using the explicit volume form ${\mathbf{i}\over 2\im\tau}dz\wedge d\bar{z}$ on $E$ and the polarization matrix \eqref{eqnpolarizationmatrix} gives
\begin{equation}\label{eqnfirstChernclass}
\sum_{e}c_{1}(t_{e}^{*}\mathcal{O}_{E}(0))={\mathbf{i}\over 2\im \tau}\sum_{i,j=1}^{n-1}M_{ij}dw_{i}\wedge d\overline{w}_{\bar{j}}\,.
\end{equation}
Combining \eqref{eqnHRR}, \eqref{eqnfirstChernclass} and \eqref{eqndetM}, we obtain 
\begin{equation}\label{eqndimofOD}
h^{0}(E^{n-1}, \mathcal{O}_{E^{n-1}}(\overline{\mathsf{D}}))=\det M=2^{n-1}n^{n-2}\,.
\end{equation}

\subsubsection{Generating set in terms of Jacobi theta functions}

One explicit basis of $H^{0}(E^{n-1}, \mathcal{O}_{E^{n-1}}(\overline{\mathsf{D}}))$ is given by \cite[Chapter 2]{Mumford:tataI}
\[
F_{\mu}=\sum_{k\in \mathbb{Z}^{n-1}+M^{-1}\mu}\exp
\left({1\over 2}2\pi i k^{t}Mk+2\pi i k^{t}(Mw+\sum_{e}B(e)c(e))
\right)\,,\quad \mu\in \mathbb{Z}^{n-1}/M\mathbb{Z}^{n-1}\,.
\]
In what follows, a more convenient generating set is needed to relate the sections to the Eisenstein-Kronecker functions.

Recall that any section $F\in  H^{0}(E^{n-1}, \mathcal{O}_{E^{n-1}}(\overline{\mathsf{D}}))$ lifts to a holomorphic function on $\mathbb{C}^{n-1}$ which
admits a Fourier expansion of the form
\[
F(w)=\sum_{k\in \mathbb{Z}^{n-1}}C_{k}e^{2\pi iw^{t}k}\,.
\]
The automorphy of the line bundle $\mathcal{O}_{E^{n-1}}(\overline{\mathsf{D}})$ determines all of the $C_{k}$'s from those in $k\in \mathbb{Z}^{n-1}/M\mathbb{Z}^{n-1}$. 
Choose once and for all a set $S$ of representatives of $\mathbb{Z}^{n-1}/M\mathbb{Z}^{n-1}$.
Then the map
\begin{equation}\label{eqninjectivemapfromH0toFouriermodes}
\varrho: H^{0}(E^{n-1}, \mathcal{O}_{E^{n-1}}(\overline{\mathsf{D}}))
\rightarrow (C_{k})_{k\in S}
\end{equation}
is injective. Since both have the same dimension $\det M$, this is an isomorphism.

Explicit sections can now be constructed as follows.
Similar to the proof of Proposition \ref{propcharacterizationofkerbarpartial}, we 
consider the space 
\begin{equation}\label{eqncyclespacelambdarecalled}
\ker B=
\mathrm{span}
\Big\{\lambda=(\lambda_{e})_{e}\in \mathbb{C}^{\mathrm{E}}~|~\sum_{e\in \mathrm{E}}B(e)\lambda_{e}=0 \in \mathbb{C}^{n-1}\Big\}\,.
\end{equation}
Define 
\begin{equation}\label{eqngeneratingsetofOD}
\Theta_{\lambda}(w)=\prod_{e}\theta(s_{e}+\lambda_{e})\,,\quad \lambda \in \ker B\,.
\end{equation}
Applying standard facts on Jacobi theta functions, we have 
 the Fourier expansion
\begin{equation}
\Theta_{\lambda}(w)=\sum_{\nu\in (\mathbb{Z}+{1\over 2})^{|\mathrm{E}|}}c_{\nu}e^{2\pi i (c+\lambda)^{t} \nu}e^{2\pi iw^{t}B\nu}\,,
\end{equation}
where $c_{\nu}$ are nonzero universal coefficients determined by the Fourier expansion of $\theta(z)$.
 Since the edge set contains  both orientations, one has $B\mathbold{1}=0$ and thus 
 \[
B\nu\in \mathbb{Z}^{n-1}\,,\quad ~\text{for}~\nu\in (\mathbb{Z}+{1\over 2})^{\mathrm{E}}\,. 
 \]\
 Grouping the terms in the sum above according to $k=B\nu\in  \mathbb{Z}^{n-1}$, one has
 \begin{equation}\label{eqnThetalambda}
 \Theta_{\lambda}(w)=\sum_{k\in \mathbb{Z}^{n-1}}e^{2\pi iw^{t}k}C_{k}(\lambda)\,,\quad 
 C_{k}(\lambda):=
 \sum_{\nu: B\nu=k}c_{\nu}e^{2\pi i (c+\lambda)^{t} \nu}\,.
 \end{equation}

 \begin{lem}\label{lemlinearindependenceofcharacters} Let $S$ be a set of representatives of $\mathbb{Z}^{n-1}/M\mathbb{Z}^{n-1}$. Then the functions $\{C_{k}(\lambda)\}_{k\in S}$ are linearly independent holomorphic functions on $ \ker B$.
\end{lem}
\begin{proof}
    Suppose that there exists a relation
    \[
    \sum_{k\in S}a_{k}C_{k}(\lambda)=0\,.
    \]
    Restricting to $ \ker B_{\mathbb{R}}$ and shifting the index  $\nu$ by ${1\over 2}\mathbold{1}$, we obtain (recall that $B\mathbold{1}=0$)
    \begin{equation}\label{eqnofcharacters}
  \sum_{k\in S}a_{k}
e^{\pi i c_{e}^{t} \mathbold{1}}e^{\pi i \lambda_{e}^{t} \mathbold{1}}
 \sum_{\nu\in \mathbb{Z}^{|\mathrm{E}|}: B\nu=k}c_{\nu+{1\over 2}\mathbold{1}}e^{2\pi i c_{e}^{t} \nu}e^{2\pi i \lambda^{t} \nu}
=0\,,\quad \lambda\in \ker B_{\mathbb{R}}\,.
    \end{equation}
    Note that $B$ has integer coefficients and full rank, thus the space $T:=\ker B_{\mathbb{R}}/\ker B_{\mathbb{Z}}$ is a compact real torus.
    Therefore, the summation above gives an absolutely convergent Fourier series on $T$.

Consider the characters $\lambda\mapsto e^{2\pi i \lambda^{t}\nu}$ on $T$.
Two characters corresponding to $\nu,\nu'\in \mathbb{Z}^{\mathrm{E}}$ are the same
iff $\nu-\nu'=B_{\mathbb{Z}}^{t}r$ for some $r\in \mathbb{Z}^{n-1}$.
Indeed, one has 
    $e^{2\pi i \lambda^{t}(\nu-\nu') }=1$ for $ \lambda\in \ker B_{\mathbb{R}}$  is equivalent to
    \[
    (\nu-\nu')\in (\ker B_{\mathbb{R}})^{\perp}\cap \mathbb{Z}^{\mathrm{E}}=\im B_{\mathbb{R}}^{t}\cap \mathbb{Z}^{\mathrm{E}}=\im B_{\mathbb{Z}}^{t}\,.
    \]
    Now suppose $\nu-\nu'=B^{t}r$ for some $r\in \mathbb{Z}^{n-1} $. 
    Applying the map $B$ gives
    \[
   k-k'= B\nu-B\nu'=BB^{r}r=Mr\,.
    \]
    Since $k,k'\in S$ are representatives from $\mathbb{Z}^{n-1}/M\mathbb{Z}^{n-1}$, this shows that $k=k'\in S$.
 By the positivity of $M$, we see that $0=k-k'=Mr$ implies $r=0$, and thus $\nu=\nu'$.
        
    Consequently, for different values of $\nu$ in \eqref{eqnofcharacters},
    the characters in the summation are disjoint.
Every coefficient of the character occurring in the summation is nonzero since $c_{\nu+{1\over 2}\mathbold{1}}\neq 0$. Moreover, each preimage set $\{m\in \mathbb{Z}^{\mathrm{E}}~|~Bm=k\}$ is nonempty since $B_{\mathbb{Z}}$ is surjective (note that $B(e_{kn})=e_{k}$ for $ 1\leq k\leq n-1$).
Therefore, by the orthogonality of Fourier characters on the compact torus $T$, we see that $a_{k}=0$ for any $k\in S$.
This proves the desired linear independence.
\end{proof}

 \begin{lem}\label{lemspanningsetofH0OD} Let the notation be as above. 
 \begin{enumerate}[i).]
     \item 
One has \begin{equation}\label{eqnH0ODspace}
H^{0}(E^{n-1}, \mathcal{O}_{E^{n-1}}(\overline{\mathsf{D}}))
=\mathrm{span}\{\Theta_{\lambda}~|~\lambda\in \ker B\}\,.
\end{equation}
\item Set 
\begin{equation}\label{eqngeneratingsetofLD}
R_{\lambda}=\prod_{e\in \mathrm{E}}  {\theta(s_{e}+\lambda_{e})\over \theta(s_{e}) } \,,\quad \lambda=(\lambda_{e})_{e}\in \ker B\,.
\end{equation}
Then one has 
\begin{equation}\label{eqnLDspace}
\mathcal{L}(\mathsf{D})
=\mathrm{span}\{R_{\lambda}~|~\lambda\in \ker B\}\,.
\end{equation}
 \end{enumerate}
 
\end{lem}
\begin{proof}
 \begin{enumerate}[i).]
     \item 
Using the automorphy \eqref{eqnautomorphyoftheta}, one can check that 
\begin{equation*}\label{eqngeneratingsetofODareindeedsections}
\Theta_{\lambda}\in H^{0}(E^{n-1}, \mathcal{O}_{E^{n-1}}(\overline{\mathsf{D}}))\,.
\end{equation*}
The desired claim then follows from the
construction of $C_{k}$ in \eqref{eqnThetalambda}, Lemma \ref{lemlinearindependenceofcharacters},   and the
fact that the map $\varrho$ in \eqref{eqninjectivemapfromH0toFouriermodes}
is injective.
\item 
Examining the poles shows that
$
R_{\lambda}\in \mathcal{L}(\mathsf{D})
$ for $\lambda\in \ker B$.
The claim then follows from \eqref{eqnODfactorization} and (i).
\end{enumerate}
\end{proof}

\subsubsection{Proof of Proposition \ref{propFnHndescription}}

We now give the proof of Proposition \ref{propFnHndescription}, whose statement is recalled below for reference.
\begin{prop}\label{propFnHndescriptionrecalled}
One has $F^{n}H^{n}(U)=\mathcal{V}_{\mathsf{D}}\Omega\cap  \ker\,d=\mathcal{V}_{\mathsf{D}}\Omega\cap  \ker\,\bar{\partial}$.

\end{prop}

\begin{proof}
    
By type reasons and \eqref{eqnFnoHn}, we see that 
\[\mathcal{V}_{\mathsf{D}}\Omega\cap  \ker\,\bar{\partial}
=
\mathcal{V}_{\mathsf{D}}\Omega\cap \ker d
\subseteq H^{0}(E_{[n]}, \Omega^{n}_{X}(\log \mathsf{D}))
=F^{n}H^{n}(U)\,.
\]
We now prove the inverse inclusion, which is equivalent to $\mathcal{L}(\mathsf{D})\subseteq \mathcal{V}(\mathsf{D})\cap \ker \bar{\partial}$
by \eqref{eqnholomorphicnformandholomorphicfunction} and \eqref{eqnidentificationofH0logwithlinearsystem}.

Regard the quantity in \eqref{eqngeneratingsetofLD} as a holomorphic function $R:\ker B\rightarrow \mathcal{L}(\mathsf{D})$.
As in the proof of Proposition \ref{propcharacterizationofkerbarpartial}, one can pick a basis of $y_{1},\cdots, y_{N}$ for $\ker B$, with $d:=\dim \ker B$  which
extends to a full basis $y_{1},\cdots,y_{d},y_{d+1},\cdots, y_{|\mathrm{E}|}$  for $\mathbb{C}^{\mathrm{E}}$, with corresponding coordinates 
$u_{k},k=1,2,\cdots, d,d+1\cdots, |\mathrm{E}|$.
Set 
\[
W_{1}:=\mathrm{span} \Big\{R(\lambda)~|~\lambda\in \ker B\Big\}\,,\quad 
W_{2}:=\mathrm{span} \Big\{\partial_{u}^{\alpha}R(\lambda)|_{\lambda=0}~|~\alpha\in \mathbb{N}^{d}\Big\}\,.
\]
We claim that $W_{2}=W_{1}$ and thus $W_{2}=\mathcal{L}(\mathsf{D})$ by Lemma \ref{lemspanningsetofH0OD} (ii).
Note that the space $W_{1}=\mathcal{L}(\mathsf{D})$ is a finite dimensional closed space.
Since taking differences 
and limits stay in $W_{1}$, one sees that the jet space $W_{2}$ satisfies
$W_{2}\subseteq W_{1}$.
Conversely, consider the quotient map
\[
\pi: \mathcal{L}(\mathsf{D})\rightarrow \mathcal{L}(\mathsf{D})/W_{2}\,.
\]
All derivatives at $0$ of $\pi\circ R$ vanish.
Applying the identity theorem to the finitely many scalar components of $\pi\circ R$, one sees that $\pi\circ R\equiv 0$ on $\ker B$.
This proves $W_{1}\subseteq W_{2}$.

It remains to show $W_{2}\subseteq \mathcal{V}(\mathsf{D})\cap \ker \bar{\partial}$. Since $R(\lambda)$ is meromorphic in $s_{e},e\in\mathrm{E}$, we clearly have  
\begin{equation} \label{eqnW2barpartialclosed}
W_{2}\subseteq \ker\bar{\partial}\,.
\end{equation}
Now set (cf. Lemma \ref{lemmaaboutclosednessofdecompositionofloopterm})
\[
f(\lambda)=\prod_{e\in \mathrm{E}}  {\theta(s_{e}+\lambda_{e})\over \theta(s_{e})   }e^{\lambda_{e}\mathbold{Y}(\bar{s}_{e}-s_{e})}=R(\lambda)\prod_{e\in \mathrm{E}}  e^{\lambda_{e}\mathbold{Y}(\bar{s}_{e}-s_{e})} \,.
\]
Observe that for $\lambda\in \ker B$, the function $\prod_{e\in \mathrm{E}}e^{\lambda_{e}\mathbold{Y}(\bar{s}_{e}-s_{e})}$ is constant in $s_{e},e\in \mathrm{E}$ according to \eqref{eqncyclespacelambdarecalled}.
Thus $W_{2}$ is the same as the space of Taylor coefficients of the function $f|_{\ker B}$ at $\lambda=0$.
Set
\[
\widehat{S}(\lambda)=\prod_{e\in \mathrm{E}}  {\theta(s_{e}+\lambda_{e})\over \theta(s_{e}) \theta(\lambda_{e}) } e^{\lambda_{e}\mathbold{Y}(\bar{s}_{e}-s_{e})}
=f(\lambda)\prod_{e\in \mathrm{E}}{1\over \theta(\lambda_{e})}\,.
\]
It is straightforward to see by Definition \ref{dfnfunctionspaceassociatedtoD} that  every Taylor coefficient of $f|_{\ker B}$  lies in $\mathcal{V}_{\mathsf{D}}$, and thus 
\begin{equation}\label{eqnW2inVD}
W_{2}\subseteq \mathcal{V}_{\mathsf{D}}\,.
\end{equation}
Combining \eqref{eqnW2barpartialclosed} and \eqref{eqnW2inVD}, we complete the proof that
$W_{2}\subseteq \mathcal{V}(\mathsf{D})\cap \ker \bar{\partial}$.

\end{proof}

\begin{appendices}
	\section{Some combinatorial identities}
	\label{secappendixcombinatorialidentities}
		
In this part, we prove some combinatorial relations needed in the body of the paper.

	\subsection{A lemma for Proposition \ref{propgeneralpropofloopterm}}

	The following lemma is needed in the proof of Proposition \ref{propgeneralpropofloopterm}.
	
	\begin{lemma}\label{generallemmaforcomputingcoefficients} 
	
		Suppose $\sum_{i=0}^{n}m_i = M $, where $m_{i}\geq 0, i=0,1,\cdots,n$.
			Then one has 
			\[
			\sum_{k \in \mu_0}a_k^{m_0-1} \prod_{\substack{i=1\\i\neq k}}^{n} (a_i-a_k)^{m_i-1} = \delta_{m_{0},0}\prod_{i=1}^{n}a_i^{m_i-1} +
			(-1)^{m_0-1} 
			\sum_{\substack{d_1,\cdots,d_n \geq 0 \\ \sum_{i=1}^{n}d_i = M-n}}
			\prod_{i=1}^{n} \binom{m_i-1}{d_i} a_i^{d_i} \,,
			\]
			where $\mu_0(\mathrm{m}) = \{i \in \{1,\cdots,n\}|m_i=0\}$ and the binomial coefficients $\binom{m-1}{j} $ is the generalized binomial coefficient with $\binom{-1}{j} = (-1)^j$.	

	\end{lemma}
	
	\begin{proof}

			It suffices to show that the relation holds for generic values of the $a_{i}$'s since both sides are rational functions in them.
			Thus we can assume $a_{i}\neq 0$ for any $1\leq i\leq n$.
			Without loss of generality, we also assume that $ m_1 \leq \cdots \leq m_n, m_1=\cdots=m_s=0, m_{s+1}\geq 1$.			
			Consider the rational function
			\[
			F(z) = \frac{z^{m_0-1} \prod_{i=s+1}^{n} (z-a_i)^{m_i-1}}
			{\prod_{i=1}^{s} (z-a_i)} .
			\]
		When $m_{0}\neq 0$, it has simple poles at $z=a_1,\dots,a_s$ and a possible pole at $z=\infty$, and is holomorphic elsewhere.
			When $m_{0}=0$, it has an additional simple pole at $z=0$.
				By the global residue theorem, one has
			\[
			\sum_{k=1}^{s}\mathrm{res}_{a_k}(Fdz) +\mathrm{res}_{0}(Fdz) + \mathrm{res}_{\infty}(Fdz) = 0.
			\]
			
			The residue at $z=a_k$ is
			\begin{eqnarray*}
				\mathrm{res}_{a_k}(Fdz)
				= \frac{a_k^{m_0-1} \prod_{i=s+1}^{n} (a_k-a_i)^{m_i-1}}
				{\prod_{\substack{i=1\\i\neq k}}^{s} (a_k-a_i)}  
			= (-1)^{\sum_{i=s+1}^{n}(m_i-1)+s-1}\,
				a_k^{m_0-1}
				\prod_{\substack{i=1\\i\neq k}}^{n} (a_i-a_k)^{m_i-1}\,.
			\end{eqnarray*}		
		When $m_{0}=0$, the residue of $F$ at $z=0$ is
		\[
		\mathrm{res}_0 (Fdz) = \frac{\prod_{i=s+1}^{n} (-a_i)^{m_i-1}}
		{\prod_{i=1}^{s} (-a_i)}
		= (-1)^{\sum_{i=s+1}^{n}(m_i-1)-s}\,
	\prod_{i=1}^{n} a_i^{\,m_i-1}\,.	
		\]
			A direct computation gives 
			\[
			\mathrm{res}_{\infty}(Fdz) = (-1)^{M-n-1} \sum_{\substack{d_1,\cdots,d_n \geq 0 \\ \sum_{i=1}^{n}d_i = M-n}}
			\prod_{i=1}^{n} \binom{m_i-1}{d_i} a_i^{d_i} \,.
			\]
			Combining and simplifying the signs, we obtain the desired claim.

	\end{proof}

\subsection{Proof of Theorem \ref{thmdecompositionofdeterminant}}
\label{secproofofthmdecompositionofdeterminant}

In this part, we give the proof of Theorem \ref{thmdecompositionofdeterminant}, regarding the 
expansion of  $\det Z_{n}$ for the matrix $Z_{n}$ given in \eqref{eqndeterminantformula}.

As a preparation, we recall and introduce some shorthand notation.
\begin{dfn}\label{dfnPhisigmakappendix}
	Following Definition \ref{dfnPhisigmak}, for the cyclic permutation
	$\sigma_k = (a_1\cdots a_p)$ with $p\geq 2$, define
	\begin{equation}\label{eqndfnpsiPsi}
		\psi_{\mathrm{m},\sigma_k} := \prod_{t=1}^{p} \widehat{e}_{m_t}(s_{a_t,\sigma_k(a_{t})})
		\,,
		\quad 
		\Psi_{M,\sigma_k}=\sum_{\substack{\mathrm{m}\in \mathbb{N}^{p}\\|\mathrm{m}|=M}}\psi_{\mathrm{m},\sigma_k}\,.
	\end{equation}
	where we keep the convention that $a_{p+t}=a_t$.
	Set also
	\begin{equation}
		\Phi^{\circ}_{0,\sigma_k} := \sum_{\substack{|\mathrm{m}|=|\sigma_{k}| \\ \mathrm{m}\neq \mathbold{1}}}\psi_{\mathrm{m},\sigma_k} \,,\quad 
		\Phi^{\circ}_{1,\sigma_k} := \sum_{\substack{|\mathrm{m}|=|\sigma_k|-1 \\ m_t \leq 1 }} \psi_{\mathrm{m},\sigma_k} \,.
	\end{equation}
	Denote 
	\begin{equation}\label{eqndnfPhirappendix}
		\Phi_{r,\sigma_k} := \widehat{e}_{r}(-\sum_{i\in \sigma_k}s_{i,\sigma_k(i)})
		\sum_{|\mathrm{m}|=|\sigma_{k}|-r} \psi_{\mathrm{m},\sigma_k} \,~\mathrm{for}~ 0\leq r\leq |\sigma_{k}|\,,\quad 
		\Phi_{\geq 1,\sigma_k} := \sum_{r=1}^{|\sigma_{k}|} 	\Phi_{r,\sigma_k} \,.
	\end{equation}
	When $p=1$, say $\sigma_{k}=(a)$, we employ the convention that 
	\begin{equation}\label{eqnconventionPsisigmakappendix}		 	
		\Phi^{\circ}_{0,\sigma_k}=0\,,\quad  \Phi^{\circ}_{1,\sigma_k}=1\,,\quad   \Phi_{\geq 1,\sigma_k}=\widehat{e}_{1}(-s_{a,a})\,.
	\end{equation}	
\end{dfn}

The determinant $\det\, Z_{n}$ can now be organized as follows.
\begin{lemma}\label{lemdetZnintermsofPsisigma}
	Let the notation be as above. Then one has 
	\begin{eqnarray*}
		\mathrm{det}\,  Z_n  
		=
		(-1)^{n-1} \sum_{\substack{\pi_0 \subseteq [n] \\ |\pi_0| \geq 1}} \left(\sum_{\sigma_0 \in \mathrm{cyc} (\pi_0) }\Phi^{\circ}_{1,\sigma_0}\right)\left( \sum_{\substack{\sigma \in \mathfrak{S}_{[n]\backslash \pi_0} \\ \sigma = \sigma_1\cdots\sigma_{\ell} } }    
		\prod_{k=1}^{\ell} (\Phi^{\circ}_{0,\sigma_k} + \Phi_{\geq 1,\sigma_k})\right) \,,
	\end{eqnarray*}
	where $\bar{\partial} \Phi_{\geq 1,\sigma_k}=0$ for any $\sigma_{k}$.
\end{lemma}

\begin{proof}
	Expanding $\det\, Z_n$ according to definition
	\begin{equation}\label{eqndetZnexpansionappendix}
		\mathrm{det} \, Z_n  = \sum_{\sigma \in \mathfrak{S}_{[n]\cup\{0\} } } \mathrm{sgn}(\sigma)\prod_{i=0}^{n} (Z_{n})_{i,\sigma(i)} 
		\,.
	\end{equation}
 There are three types of summands in \eqref{eqndetZnexpansionappendix}.
	\begin{itemize}
		\item Type I. The cycle $\sigma_0$ contains only one element, i.e. $\sigma_0=(0)$. Then
		\[
		\mathrm{sgn}(\sigma) \prod_{i=0}^{n} (Z_{n})_{i,\sigma(i)} = 0.
		\]
		
		\item Type II. The cycle $\sigma_0$ contains two elements. Suppose that $\sigma_0=(0j)$.
		Then
		\[
		\mathrm{sgn}(\sigma) \prod_{i=0}^{n} (Z_{n})_{i,\sigma(i)} = \mathrm{sgn}(\sigma_{0})\cdot 1\cdot (-1)\cdot 
		\prod_{k=1}^{\ell}\mathrm{sgn}(\sigma_{k})\psi_{\mathbold{1},\sigma_k} \,.
		\]
		Applying  Corollary \ref{corforlooptermofwidehate_1} (which holds as long as $|\sigma_{k}|\geq 1$ by the convention in \eqref{eqnconventionPsisigmak} and the parity of $\widehat{e}_{1}$)  to each  $\sigma_{k}$, this gives 
		\[
		-	\mathrm{sgn}(\sigma_{0})(-1)^{\ell} \prod_{k=1}^{\ell}\mathrm{sgn}(\sigma_{k}) ( \Phi^{\circ}_{0,\sigma_k}+ \Phi_{\geq 1,\sigma_k})\,.
		\]
		It follows that the contribution of Type II terms is 
		\begin{eqnarray*}
			\mathrm{Type~ II}= \sum_{j=1}^{n} \sum_{\substack{\sigma \in \mathfrak{S}_{[n]\backslash\{j\}} \\ \sigma = \sigma_1\cdots\sigma_{\ell}} }  \mathrm{sgn}(\sigma) \prod_{k=1}^{\ell} \psi_{\mathbold{1},\sigma_k}
			=
			(-1)^{n-1} \sum_{j=1}^{n} \sum_{\substack{\sigma \in \mathfrak{S}_{[n]\backslash\{j\}} \\ \sigma = \sigma_1\cdots\sigma_{\ell}} } 
			\prod_{k=1}^{\ell} 
			(\Phi^{\circ}_{0,\sigma_k} + \Phi_{\geq 1,\sigma_k}) \,.
		\end{eqnarray*}

		\item  Type III. The cycle $\sigma_0$ contains more than two elements. Suppose that $\sigma_0=(0 a_1\cdots a_p)$ with $p\geq 2$.
		Then
		\[
	\mathrm{sgn}(\sigma)\prod_{i=0}^{n} (Z_{n})_{i,\sigma(i)} = \mathrm{sgn}(\sigma_{0}) \cdot 
		(-\widehat{e}_{1}(s_{a_1,a_2}) \cdots \widehat{e}_1(s_{a_{t-1},a_t}) )\cdot 	\prod_{k=1}^{\ell}	\mathrm{sgn}(\sigma_{k}) \psi_{\mathbold{1},\sigma_k} \,.
		\]
		
			By  the definition of $\psi_{\mathrm{m},\sigma_0'}$, for $\pi_0' \subseteq [n]$ with $|\pi_0'|\geq 2$,  one has 
		\begin{eqnarray*}
			\sum_{\substack{\sigma_0 \in \mathfrak{S}_{\pi_0' \cup \{0\}} \\ \sigma_0=(0 a_1\cdots a_p) }}
			\widehat{e}_{1}(s_{a_1,a_2}) \cdots \widehat{e}_1(s_{a_{p-1},a_p}) 
			= \sum_{\sigma_0' \in \mathrm{cyc} (\pi_0') } \sum_{\substack{|\mathrm{m}|=|\sigma_0'|-1 \\ m_t \leq 1 }} \psi_{\mathrm{m},\sigma_0'} 	= \sum_{\sigma_0' \in \mathrm{cyc} (\pi_0') }\Phi^{\circ}_{1,\sigma_{0}'}\,.
		\end{eqnarray*}	
		This gives 
		\begin{eqnarray*}
			\mathrm{Type~ III}~&=& \sum_{\substack{\pi_0' \subseteq [n] \\ |\pi_0'| \geq 2}} 
			\left(\sum_{\sigma_0' \in \mathrm{cyc} (\pi_0')   }
			\mathrm{sgn}(\sigma_{0}')			\Phi^{\circ}_{1,\sigma_{0}'} \right)
			\left(
			\sum_{\substack{\sigma' \in \mathfrak{S}_{[n]\backslash \pi_0'} \\ \sigma' = \sigma_1'\cdots\sigma_{\ell'}' } }    
			\mathrm{sgn}(\sigma')	
			\prod_{k=1}^{\ell'} \psi_{\mathbold{1},\sigma_k'} \right)\,.		
		\end{eqnarray*}
		We again apply Corollary \ref{corforlooptermofwidehate_1}
		to rewrite each $\psi_{\mathbold{1},\sigma_k'} $. 
	\end{itemize}
	Putting all three types of terms together and applying convention \eqref{eqnconventionPsisigmakappendix}  proves the desired identity.
	The $\bar{\partial}$-closedness of $\Phi_{\geq 1,\sigma_k}=0$ follows from Lemma \ref{lemmaaboutclosednessofdecompositionofloopterm}.
\end{proof}

The summation over permutations formula in Lemma \ref{lemdetZnintermsofPsisigma}
can be expressed as a summation over set partitions.
To see that we introduce some new shorthand notation.
\begin{definition}\label{dfnPsipi}
	Let the notation be as in Definition \ref{dfnPhisigmakappendix}.
	Define 
    \[
    \psi_{\mathrm{m},\pi_k}=\sum_{\sigma_k \in \mathrm{cyc} (\pi_k)}\psi_{\mathrm{m},\sigma_k}\,,\quad 
\Psi_{M,\pi_k}=\sum_{\sigma_k \in \mathrm{cyc} (\pi_k)}\Psi_{M,\sigma_k}\,,
    \]
    and 
	\[
	\Phi^{\circ}_{0,\pi_k} = \sum_{\sigma_k \in \mathrm{cyc} (\pi_k)} \Phi^{\circ}_{0,\sigma_k}\,,\quad 
	\Phi^{\circ}_{1,\pi_k} = \sum_{\sigma_k \in \mathrm{cyc} (\pi_k)} 	\Phi^{\circ}_{1, \sigma_k}
	\,,\quad 
	\Phi_{\geq 1,\pi_k} = \sum_{\sigma_k \in \mathrm{cyc} (\pi_k)} \Phi_{\geq 1,\sigma_k}\,.
	\]	
    	For any partition $\pi=\{\pi_{1},\cdots, \pi_{\ell}\}$, we denote 
	$\Phi^{\circ}_{0,\pi}=\prod_{k=1}^{\ell}\Phi^{\circ}_{0,\pi_{k}}$ and  similarly for 
	$\Phi^{\circ}_{1,\pi}$ and $ \Phi_{\geq 1,\pi}$. 
\end{definition}

\begin{corollary}\label{cordetZnintermsofPsipi}
	Let the notation be as above. Then the following statements hold.
	\begin{enumerate}
		\item [i).]
		One has 
		\[
		\det Z_{n}=
		(-1)^{n-1} \sum_{\substack{\pi_0 \subseteq [n] \\ |\pi_0| \geq 1}} \Phi^{\circ}_{1,\pi_0}\left( \sum_{\substack{\pi \in \Pi_{[n]\backslash \pi_0} \\ \pi =  \{\pi_{1},\cdots,\pi_{\ell} \} }}    
		\prod_{k=1}^{\ell} (\Phi^{\circ}_{0,\pi_k} + \Phi_{\geq 1,\pi_k})\right) \,.
		\]
		\item [ii).]
		One has 
		\begin{eqnarray*}
			\det\,  Z_n =(-1)^{n-1} \sum_{\substack{S \subseteq [n] \\ |S| \geq 1}} 
			\sum_{\substack{\pi_0 \subseteq S \\ |\pi_0| \geq 1}} \Phi^{\circ}_{1,\pi_0} \sum_{\pi' \in \Pi_{S\backslash \pi_0}} \Phi^{\circ}_{0,\pi'} 
			\sum_{\pi'' \in \Pi_{[n]\backslash S}} \Phi_{\geq 1,\pi''}\,.
		\end{eqnarray*}
	\end{enumerate} 
\end{corollary}
\begin{proof}
	\begin{enumerate}
		\item [i).]
		Suppose the  underlying set of $\sigma_{k}$ is  $\pi_{k}$, then the set of 
		possible cycles $\sigma_{k}$ is in bijection with $\mathrm{cyc} (\pi_k)$.
		The desired claim then follows immediately from Lemma \ref{lemdetZnintermsofPsisigma}.
		\item [ii).]
		
		Suppose $\pi_{0}\subseteq [n]$
		and $\pi=\{\pi_{1},\cdots, \pi_{\ell}\}\in \Pi_{[n]\setminus \pi_{0}}$.
		We use the notation $\pi=(\pi',\pi'')$ to denote a partition of 
		$[\ell]$ into an ordered tuple $(I, [\ell]-I)$, with 
		\[
		\pi'=\{\pi_{k}~|~k\in I\}\,,\quad 
		\pi''=\{\pi_{k}~|~k\in [\ell]\setminus I\}\,.
		\]
		Correspondingly, we denote 
		\[
		\Phi^{\circ}_{0,\pi'} = \prod_{k \in I} \Phi^{\circ}_{0,\pi_k'}\,,
		\quad 
		\Phi_{\geq 1,\pi''} = \prod_{k \in [\ell]\setminus I} \Phi_{\geq 1,\pi_k}\,.
		\]
		Then we have 
		\begin{eqnarray*}
			\sum_{\substack{\pi \in \Pi_{[n]\backslash \pi_{0}} \\ \pi = \{\pi_1,\cdots,\pi_\ell\}} } \prod_{k=1}^{\ell} 
			( \Phi^{\circ}_{0,\pi_k} + \Phi_{\geq 1,\pi_k} )
			&=& \sum_{\substack{\pi \in \Pi_{[n]\backslash \pi_{0}} \\ \pi = \{\pi_1,\cdots,\pi_l\}} } \sum_{I \subseteq [\ell]}\left( \prod_{k \in I} \Phi^{\circ}_{0,\pi_{k}} \prod_{k \in [\ell]\backslash I}\Phi_{\geq 1,\pi_{k}}\right) \\
			&=& \sum_{\pi \in \Pi_{[n]\backslash \pi_{0}}}  
			\sum_{\pi = (\pi',\pi'')}
			\Phi^{\circ}_{0,\pi'} \Phi_{\geq 1,\pi''}
			\\
			&=& \sum_{S \subseteq [n]\backslash \pi_{0}}
			\sum_{\pi' \in \Pi_S  }
			\sum_{\pi'' \in \Pi_{[n]\backslash (\pi_{0} \cup S) }}
			\Phi_{0,\pi'}^{\circ} \Phi_{\geq 1,\pi''}\,. \\
		\end{eqnarray*}
		It follows from (i) that 
		\begin{eqnarray*}
			\det\, Z_n &=&  (-1)^{n-1} \sum_{\substack{\pi_0 \subseteq [n] \\ |\pi_0| \geq 1}} \Phi^{\circ}_{1,\pi_0}
			\sum_{S \subseteq [n]\backslash \pi_{0}}
			\sum_{\pi' \in \Pi_S  }
			\sum_{\pi'' \in \Pi_{[n]\backslash (\pi_{0} \cup S) }}
			\Phi_{0,\pi'}^{\circ} \Phi_{\geq 1,\pi''}\\
			&=& (-1)^{n-1} \sum_{\substack{S \subseteq [n] \\ |S| \geq 1}} 
			\sum_{\substack{\pi_0 \subseteq S \\ |\pi_0| \geq 1}} \Phi^{\circ}_{1,\pi_0} \sum_{\pi' \in \Pi_{S\backslash \pi_0}} \Phi_{0,\pi'}^{\circ} 
			\sum_{\pi'' \in \Pi_{[n]\backslash S}} \Phi_{\geq 1,\pi''} \,.
		\end{eqnarray*}
	\end{enumerate} 
\end{proof}

We next reorganize the summation formula in Corollary \ref{cordetZnintermsofPsipi} according to the behavior under the $\bar{\partial}$-operator. We first establish the following technical result.

\begin{lem}\label{lemmessytermsindetZnexpansion}
	For any $S \subseteq [n], |S| \geq 1$, one has
	\[
	\sum_{\substack{\pi_0 \subseteq S \\ |\pi_0| \geq 1}} \left(\sum_{\sigma_0 \in \mathrm{cyc}(\pi_0)} \Phi^{\circ}_{1,\sigma_0}\right)\left(
	\sum_{\substack{\sigma' \in \mathfrak{S}_{S\backslash \pi_0} \\ \sigma' = \sigma_1'\cdots\sigma_{\ell'}'}} 
	\prod_{k=1}^{\ell'} \Phi^{\circ}_{0,\sigma_k'} \right)
	= \sum_{\sigma \in \mathrm{cyc} (S)} \sum_{|\mathrm{m}|=|S|-1} \psi_{\mathrm{m},\sigma}\,.
	\]
	In particular, it is $\bar{\partial}$-closed.
\end{lem}
\begin{proof}
	The $\bar{\partial}$-closedness of the right hand side follows from Lemma \ref{lemmaaboutclosednessofdecompositionofloopterm}. Thus, it is automatically $\bar{\partial}$-closed if  the identity is proved. 

	We expand the left hand side of the desired identity as a polynomial in the functions $\widehat{e}_{m}(s_{i,j})$.
	Consider the labeled decorated graphs for the monomials in the expansions, as in Definition \ref{dfndivisorgraph}.
	For the exceptional case where the monomial is $\widehat{e}_{0}(s_{j,j})=1$, corresponding to $\pi_{0}=\{j\}$ for some $j$, we associate a single vertex $j$ without edges. 		
	Any monomial (denoted by $\widehat{e}_{\Gamma}$) in the expansion has a uniquely determined graph (denoted by $\Gamma$) corresponding to it.		
	For example, any graph $\gamma_{0}$ arising from a monomial in $\Phi^{\circ}_{1,\sigma_0}$ must be a chain with the decoration on each edge being $1$, or a single vertex when $|\pi_{0}|=1$.
	
	Denote the set of such graphs by $\mathcal{G}^{\circ}$.
	These graphs satisfy the following conditions.

	\begin{itemize}
		\item  The properties in  Lemma \ref{lemdecorateddivisorgraphs}.

		\item Following the notation in Definition \ref{dfngraphterminologies}, 	for each 
		graph $\Gamma=\{\Gamma_{1},\cdots, \Gamma_{\ell}\}\in \mathcal{G}^{\circ}$,
		one has
		\begin{equation}
			\sum_{k=1}^{\ell} |V(\Gamma_{k})| = |S|\,,
			\quad 
			\sum_{k=1}^{\ell} m(\Gamma_{k}) = |S|-1 \,.
		\end{equation}
		In particular,
		\begin{equation}\label{eqnconstraintonGamma}
			m(\Gamma)-|V(\Gamma_{\Gamma})|=
			\sum_{k=1}^{\ell} (m(\Gamma_{k}) - |V(\Gamma_{k})|) = -1 \,.
		\end{equation}
		\item There must exist a  connected component (corresponding to the aforementioned graph $\gamma_{0}$ arising from monomials in $\Phi^{\circ}_{1,\sigma_0}$) which is either a chain with the decoration on each edge being $1$  or is a single vertex. 
		In fact, the existence  follows from \eqref{eqnmGammakminusVGammak}, \eqref{eqnconstraintonGamma} and  Lemma \ref{lemdecorateddivisorgraphs} (ii). 
	\end{itemize}	
	It is straightforward to check that these properties characterize all possible graphs arising from the monomials in the expansions of both the left hand side and right hand side.

	We denote the resulting expansions of the left and right hand sides as
	\[
	\sum_{\Gamma\in \mathcal{G}^{\circ}} n_{\Gamma}^{\mathrm{left}} \widehat{e}_{\Gamma}\,,\quad 	\sum_{\Gamma\in \mathcal{G}^{\circ}} n_{\Gamma}^{\mathrm{right}} \widehat{e}_{\Gamma}  \,.
	\]
	It suffices to show that  for any graph  $\Gamma\in \mathcal{G}^{\circ}$ one has 
	\begin{equation}\label{eqnmessytermsindetZnexpansion}
		n_{\Gamma}^{\mathrm{left}} =	n_{\Gamma}^{\mathrm{right}} \,.
	\end{equation}
	
	Consider the right hand side first.
	The outer summation over $\sigma\in \mathrm{cyc} (S)$
	corresponds to the summation over
	all possible configurations of labeled beads (with the labelings from $S$) placed on a circle, while the inner one corresponds to separating the beads on the circle into $\ell$ blocks.
	Each block gives rise to a connected component of the resulting graph $\Gamma\in \mathcal{G}^{\circ}$. Since $\widehat{e}_{\Gamma}$ is independent of the positions of the blocks on the circle, 
	for each $\Gamma$ there are $(\ell-1)!$ ways of expressing $\widehat{e}_{\Gamma}$ from the double sum. This gives 
	\[
	n_{\Gamma}^{\mathrm{right}} = (\ell-1)!\,.
	\]
	
	We next  consider $n_{\Gamma}^{\mathrm{left}}$.	
	By construction,	
	$n_{\Gamma}^{\mathrm{left}}$ is the number of ways of grouping the $\ell$ connected components $\Gamma_{1},\cdots, \Gamma_{\ell}$ in $\Gamma$ in the following way.
	\begin{itemize}
		\item Since $\Gamma\in \mathcal{G}^{\circ}$, among the connected components, there must exist a block say $\Gamma_{k_{0}}$ that is a chain with decoration being $1$ or a single vertex. We designate it as $\gamma_{0}$, corresponding to monomials in $\Phi^{\circ}_{1,\sigma_0}$. By the previous discussions, any $\Gamma_{k}$ satisfying $m(\Gamma_{k})-|V(\Gamma_{k})|=-1$  can be used to construct $\gamma_{0}$.
		
		\item By \eqref{eqnconstraintonGamma}, the rest of the $\ell-1$ connected components  now satisfy
		\[
		\sum_{k: k\neq k_{0}} (m(\Gamma_{k})-|V(E_{\Gamma_{k}})|)=0\,.
		\]  
		We group the rest of the connected components $\Gamma_{1},\cdots, \Gamma_{k_{0}-1},\Gamma_{k_{0}+1}, \Gamma_{\ell}$
		into  a set of $\ell'$ tuples $\{\gamma_{1},\cdots, \gamma_{\ell'}\}$, where
		\[
		\gamma_{1}=(\Gamma_{i})_{i\in I_{1}}\,\cdots\,,
		\gamma_{k}=(\Gamma_{i})_{i\in I_{k}}\,,\cdots\,,
		\gamma_{\ell'}=(\Gamma_{i})_{i\in I_{\ell'}}\,,\quad 
		\sqcup_{k=1}^{\ell'}I_{k}=	[\ell]-\{k_{0}\}\,.	
		\]
		Two tuples
		$\gamma_{i}=(\Gamma_{i_{1}},\cdots, \Gamma_{i_{m}}),\tilde{\gamma}_{i}=(\Gamma_{\tilde{i}_{1}},\cdots, \Gamma_{\tilde{i}_{n}})$ are considered to be the same if 
		$\{i_{1},\cdots, i_{m}\}=\{\tilde{i}_{1},\cdots, \tilde{i}_{n}\}$ and
		the cyclic permutations $(i_{1}\cdots i_{m})$, $(\tilde{i}_{1}\cdots \tilde{i}_{n})$ are the same.
		Each group is required to satisfy
		\[
		m(\gamma_{k})-|V(E_{\gamma_{k}})|=\sum_{i\in I_{k}}	m(\Gamma_{i})-|V(E_{\Gamma_{i}})|
		=0\,,~ k=1,2,\cdots, \ell'\,.
		\]
		The graph $\gamma_{k}$ corresponds to the (possibly disconnected) graph associated to monomials in $\Phi^{\circ}_{0,\sigma_k}, k=1,2,\cdots,\ell'$, respectively.
	\end{itemize} 				  
	
	By Lemma \ref{lemmaforcalculatingcoefficients} proved below, one has 
	\begin{equation}\label{eqnnGammaleft}
	n_{\Gamma}^{\mathrm{left}} = (\ell-1)!\,.
	\end{equation}
	This proves \eqref{eqnmessytermsindetZnexpansion}, and thus the claim.	    
\end{proof}

The following gives the proof of the identity \eqref{eqnnGammaleft}.

\begin{lemma}\label{lemmaforcalculatingcoefficients}
	Suppose the sequence of integers $a_1,\cdots,a_\ell$ satisfies 
	\begin{equation}\label{eqnsummationofdefectsrelation}
		a_{k}\geq -1\,,\quad 
		\sum_{k=1}^{\ell} a_k = -1\,.
	\end{equation}
	Let $N$ be the number of ways of grouping $a_1,\cdots,a_\ell$ into  a set of $\ell'+1$ tuples $\{\gamma_{0},\gamma_{1},\cdots, \gamma_{\ell'}\}$, where 
	
	\begin{itemize}
		\item the tuple $\gamma_{0}$ consists of a single element whose value is $-1$;
		\item 
		for each tuple $\gamma_{k},k=1,2,\cdots, \ell'$, one has $\sum_{a\in \gamma_{k}}a=0$;
		\item two tuples
		$\gamma_{i}=(a_{i_{1}},\cdots, a_{i_{m}}),\tilde{\gamma}_{i}=(a_{\tilde{i}_{1}},\cdots, a_{\tilde{i}_{n}})$ are considered to be the same if 
		$\{i_{1},i_{2},\cdots, i_{m}\}=\{\tilde{i}_{1},\cdots, \tilde{i}_{n}\}$ and
		the cyclic permutations $(i_{1}\cdots i_{m})$, $(\tilde{i}_{1}\cdots \tilde{i}_{n})$ are the same.
	\end{itemize}
	Then $N=(\ell-1)!$.
\end{lemma}
\begin{proof}
	We prove by induction on the number of integers $\ell$.
	
	For $\ell=1$, one has $a_1=-1$ and the claim holds trivially in this case.
	For $\ell=2$, we have that $a_1=0,a_2=-1$, or $a_1=-1,a_2=0$.
	There is $1$ possible way of grouping for each scenario.  The lemma also holds true.
	
	Assuming the claim holds when the number of the integers  is $\ell-1$. Now let's consider the case when the number is $\ell$.
	There are two cases.
	\begin{itemize}
		\item Case 1.
		There exists at least one $a_k>0$. 
		By a relabeling of indices if needed, we can assume that $a_\ell>0$.  Then 
		$a_\ell$ can not be grouped separately to form a tuple $\gamma_{0}$ or $\gamma_{k}$.  Using the identification of tuples under cyclic permutations, we 
		can assume that in the tuple containing $a_{\ell}$,  the term after $a_\ell$ is $a_s$.  That is, the tuple takes the form $(\cdots a_{\ell}a_{s}\cdots )$. There are $\ell-1$ possibilities for $a_{s}$ chosen from $\{a_{1},\cdots ,a_{\ell}\}\setminus \{a_{\ell}\}$.
			
		For any choice  of $a_{s}$ one has 
		$a_{\ell}+a_{s}\geq a_{s}\geq -1$. 
		We then replace the sequence $a_{1},\cdots ,a_{\ell}$ by the new sequence 
		$a_{1},\cdots ,\check{a}_{s},\cdots, a_{\ell-1},\check{a}_{\ell}, a_{\ell}+a_{s}$, here the notation $\check{a}_{s},\check{a}_{\ell}$ mean the omission of $a_{s},{a}_{\ell}$.		
		According to the inductive hypothesis,
				there are $(\ell-2)!$ possible ways to group the new sequence.		
		Therefore, the total number of possible ways of grouping the original sequence $a_1,\cdots,a_\ell$ is
		\[
		(\ell-1) \times (\ell-2)! = (\ell-1)! \,.
		\]
		\par
		\item Case 2.
		There is no $a_k$ satisfying $a_{k} > 0$. By the condition $\sum_{k=1}^{\ell} a_k = -1$, we see that all $a_{k}$'s take the value $0$ except for one which takes the value $-1$. Again by a relabeling of indices if needed, we can assume that $ a_1 = -1$ and $ a_k = 0, 2 \leq k \leq \ell$.
		Hence $a_1$ must be grouped separately to form $\gamma_{0}$. Now one  needs to group the remaining $\ell-1$ terms into several tuples. The ways of grouping are in one-to-one correspondence with the standard cycle decompositions of permutations in $\mathfrak{S}_{[\ell-1]}$.
		This shows
		\[
		N=|\mathfrak{S}_{[\ell-1]}| = (\ell-1)! \,.
		\]
	\end{itemize}

\end{proof}

We can finally prove Theorem \ref{thmdecompositionofdeterminant} that expresses $	\det\,  Z_n$ as a summation of the building blocks $\Psi_{M,I}$.
We recall its statement for reference.
\begin{thm} \label{thmdecompositionofdeterminantrecalled}
	
	Let the notation be as above.
	\begin{enumerate}[i).]
		\item 
		
		One has
		\begin{equation}\label{finalexpansionofdetZnappendix}
			\det\,  Z_n = (-1)^{n-1} 
			\sum_{\substack{S \subseteq [n] \\ |S| \geq 1}}  \Psi_{|S|-1,S}
			\sum_{\pi \in \Pi_{[n]\backslash S}} \Phi_{\geq 1,\pi} \,.
		\end{equation}
		\item 
		One has	\begin{equation}\label{eqnTngwpureweightstructuregeneralappendix}
			[(\det Z_{n})\Omega]_{\mathrm{Hodge}}=
			\left(\sum_{j\in [n]}\sum_{\pi\in \Pi_{[n]\setminus \{j\}}}
			\prod_{k=1}^{\ell} \sum_{\sigma_{k}\in \mathrm{cyc}(\pi_{k})} \widehat{e}_{|\pi_{k}|}(\sum_{i\in \sigma_{k}}s_{i,\sigma_{k}(i)})\right)\Omega\,.
		\end{equation}
		In particular, Theorem \ref{mainthm1Zhou} holds.
	\end{enumerate}
\end{thm}
\begin{proof}
	\begin{enumerate}[i).]
		\item 
		Invoking Definition \ref{dfnPsisigma}, Definition \ref{dfnPhisigmakappendix} and Definition \ref{dfnPsipi},  Lemma \ref{lemmessytermsindetZnexpansion} reads 
		\[
		\sum_{\substack{\pi_0 \subseteq S \\ |\pi_0| \geq 1}} \Phi^{\circ}_{1,\pi_0}
		\sum_{\pi' \in \Pi_{S\backslash \pi_0} } 
		\Phi^{\circ}_{0,\pi'} 
		=  \sum_{|\mathrm{m}|=|S|-1} \psi_{\mathrm{m},S}
		= \Psi_{|S|-1,S}  \,.
		\]
		Plugging this into Corollary \ref{cordetZnintermsofPsipi} (ii), one  obtains the result.
		
		\item Both $\Psi_{|S|-1,S}$ and $\Phi_{\geq 1,\pi}$ are  $\bar{\partial}$-closed by 
		Lemma \ref{lemmaaboutclosednessofdecompositionofloopterm}. 
		Recall that one has the natural isomorphism
		\begin{equation}\label{eqnfactorizationofholvolform}
			H^{0}(E_{[n]},\Omega_{E_{[n]}}^{n})\cong H^{0}(E_{S},\Omega_{E_{S}}^{|S|})
			\otimes H^{0}(E_{[n]\setminus |S|},\Omega_{E_{[n]\setminus S}}^{n-|S|})\,,
		\end{equation}
        Similar to the proof of Theorem \ref{thmDelignecanonicalsplittingonformlevel}, using the K\"unneth isomorphism, 
		to compute the class $[(	\det\,  Z_n)\Omega]_{\mathrm{Hodge}}$ it suffices to compute those for the individual factors of each summand.	
		By \eqref{eqndnfPhirappendix}, \eqref{eqnconventionPsisigmakappendix}	 and Proposition \ref{propleadingcomponent} (ii), we see that the  corresponding class of $\Psi_{|S|-1,S}$ is nonzero iff
		$
		|S|=1$ in which case 
		\[
		\Psi_{|S|-1,S}  =1\,.
		\]
		That of $\Phi_{\geq 1,\pi_{k}}$ is nonzero iff  $|\pi_{k}|=1$ in which case 
		\[
		\Phi_{\geq 1,\pi_{k}}=e_{1}(-\sum_{i\in \pi_{k}}s_{i,i})\,;\]
		or $|\pi_{k}|\geq 2$ in which case the nonzero contribution arises from
		\[
		\Phi_{|\pi_{k}|,\pi_{k}}=\sum_{\sigma_{k}\in \mathrm{cyc}(\pi_{k})}\widehat{e}_{r}(-\sum_{i\in \sigma_k}s_{i,\sigma_k(i)})\,.
		\]
		Combining, this yields the desired relation \eqref{eqnTngwpureweightstructuregeneral}.
		
		The extra assumption \eqref{eqnassumptiononskl}  ensures that 
		\[
		\sum_{i\in \sigma_{k}}s_{i,\sigma_{k}(i)}
		=\sum_{i\in \pi_{k}}w_{i}
		\]
		only depends on $\sigma_{k}$ through its underlying set $\pi_{k}$, thus 
		\[
		\sum_{\sigma_{k}\in \mathrm{cyc}(\pi_{k})} \widehat{e}_{|\pi_{k}|}(\sum_{i\in \sigma_{k}}s_{i,\sigma_{k}(i)})=(|\pi_{k}|-1)!\cdot \widehat{e}_{|\pi_{k}|}(\sum_{i\in \pi_{k}}w_{i})\,.
		\]
		This  proves Theorem \ref{mainthm1Zhou}.

	\end{enumerate}
\end{proof}

\end{appendices}

\bibliographystyle{amsalpha}

\providecommand{\bysame}{\leavevmode\hbox to3em{\hrulefill}\thinspace}
\providecommand{\MR}{\relax\ifhmode\unskip\space\fi MR }
\providecommand{\MRhref}[2]{%
  \href{http://www.ams.org/mathscinet-getitem?mr=#1}{#2}
}
\providecommand{\href}[2]{#2}

\bigskip{}

\noindent{\small Qiuzhen College, Tsinghua University, Beijing 100084, P. R. China}

\noindent{\small Email: \tt shengzx23@mails.tsinghua.edu.cn}\\

\noindent{\small Yau Mathematical Sciences Center, Tsinghua University, Beijing 100084, P. R. China}

\noindent{\small Email: \tt jzhou2018@mail.tsinghua.edu.cn}

\end{document}